\documentclass[11pt*]{article}

\usepackage{amsfonts, amsmath, amssymb, amsthm, mathrsfs, bbm, color, enumerate, graphicx, mathtools, tikz, hyperref, relsize, bm,soul}
\usepackage{dsfont} \usepackage{leftidx}
\usepackage[english]{babel}
\usepackage{tikz}
\usepackage{tikz-cd}  %交换图表
\usetikzlibrary{positioning, arrows.meta, calc}
\usepackage{enumerate}
\usepackage{appendix}
\usepackage[left=2.5cm, right=2.5cm, top=2.5cm, bottom=2.5cm]{geometry}
\usepackage{mnotes}%允许批注
\usepackage[
  backend=biber,
  style=alphabetic,
  maxbibnames=99,
  giveninits=true
]{biblatex}
\renewbibmacro{in:}{}

\usepackage{csquotes}%防止引用打架
\usepackage{nccmath}
\usepackage{xcolor}
\usepackage{enumitem}

\usetikzlibrary{arrows}

\numberwithin{equation}{section} 

\usepackage{setspace}
\usepackage{soul}%为了解决\st{}里面\ref{}不编译的问题
\soulregister\ref7

\theoremstyle{plain}

\newtheorem{theorem}{Theorem}[section]
\newtheorem{corollary}[theorem]{Corollary}
\newtheorem{example}[theorem]{Example}
\newtheorem{lemma}[theorem]{Lemma}
\newtheorem{proposition}[theorem]{Proposition}

\newtheorem{fact}[theorem]{Fact}
\newtheorem{definition}[theorem]{Definition}

\newcommand*\tageq{\refstepcounter{equation}\tag{\theequation}}

\newcommand{\function}{f_N}
\newcommand{\average}{\mathcal A}
\newcommand{\difference}{\nabla}
\newcommand{\noise}{\omega}%discrete white noise
\newcommand{\truncatednoise}{\tilde\omega_N}
\newcommand{\boxdomain}{\Lambda^N_{a,b}}%[0,an]\tiems[-bn,bn] 
\newcommand{\trapezoid}{\mathbb T^N_{a,b}}
\newcommand{\setsize}[1]{\left\vert#1\right\vert}
\newcommand{\absolute}[1]{\left\vert#1\right\vert}
\newcommand{\driving}{\phi}
\newcommand{\rate}{\varepsilon}%scaling rate

\newcommand{\kterm}[1]{K_N^{(#1)}}%renormalization term
\newcommand{\ktermnew}[1]{\tilde{K}_N^{(#1)}}% simplified version
\newcommand{\xterm}[1]{X_N^{(#1)}}
\newcommand{\error}[1]{\delta_N^{(#1)}}%error term
\newcommand{\cha}{\nabla}%黑边的权重
\newcommand{\tree}{T}%树
\newcommand{\vroot}{\bm{r}}%根
\newcommand{\extension}[1]{\mathfrak{E}\left(#1\right)}
\newcommand{\oneptunion}{\bigvee}
\newcommand{\one}{\emph{o}} %只有一条边的树
\newcommand{\treefamily}{\mathcal{T}}%每个\kterm的树表示
\newcommand{\embed}{\psi}
\newcommand{\embedset}[1]{\Psi_{#1}}
\newcommand{\embedsetnew}[1]{\tilde\Psi_{#1}}%xterm的embedding
\newcommand{\weight}[1]{c_{#1}}%边权，我给改成psi在下标了
\newcommand{\weightnew}[1]{c'_{#1}}

\newcommand{\moment}[2]{W_{#1}\left(#2\right)}%某个具体划分对应的w乘积之矩
\newcommand{\action}[2]{\mathcal{L}_{#1}\left(#2\right)}
\newcommand{\actionnew}[2]{\mathcal{L}'_{#1}\left(#2\right)}
\newcommand{\reddeg}{\mathrm{deg}_{R}}
\newcommand{\blackdeg}{\mathrm{deg}_{B}}
\newcommand{\freedom}{\mathrm{deg}_{F}}

\newcommand{\RomanNumeralCaps}[1] {\MakeUppercase{\romannumeral #1}}
\newcommand{\loglesssim}{\mathop{\lesssim}\limits^{\log}}
\newcommand{\distribution}[1]{\nu_{#1}}%表示t=#1这一行的noise的distribution

\newcommand{\bbE}{\mathbb{E}}

\newcommand{\bbP}{\mathbb{P}}

\newcommand{\bbR}{\mathbb{R}}

\newcommand{\bbZ}{\mathbb{Z}}

\newcommand{\R}{\mathbb{R}}  %实数
\newcommand{\Z}{\mathbb{Z}}  %整数
\newcommand{\p}{\partial}  %偏导数

\allowdisplaybreaks[4]

\begin{document}

\setcounter{tocdepth}{2}

\title{\LARGE\textbf{Subcritical limits of a 1D KPZ growth surface with finite moments}}
\date{\today }
\author{Dang-Zheng Liu\footnotemark[1]~, Fangzhou Luo\footnotemark[2]~,  Tian Wu\footnotemark[1]~, and Yuxuan Zong\footnotemark[2]}
\renewcommand{\thefootnote}{\fnsymbol{footnote}}
\footnotetext[1]{School of Mathematical Sciences, University of Science and Technology of China, Hefei 230026, P.R.~China. 
dzliu@ustc.edu.cn}
\footnotetext[1]{School of Mathematical Sciences, University of Science and Technology of China, Hefei 230026, P.R.~China. 
wt1997@ustc.edu.cn}

\footnotetext[2]{School of Mathematical Sciences, Peking University, Beijing 100871,  P.R.~China. 
luo\_joe2005@stu.pku.edu.cn}

\footnotetext[2]{School of Mathematical Sciences, Peking University, Beijing 100871,  P.R.~China. 
yxzong25@stu.pku.edu.cn}
\maketitle

\begin{abstract}
We consider the one-dimensional random growth surface introduced by Adhikari and Chatterjee \cite{adhikari2024invariance}, where the height function is defined recursively via the heights at two neighboring sites and an independent noise term at each space-time point. In the subcritical regime, assuming uniformly bounded eighth moments of the noise variables and appropriate regularity of the driving function, we establish convergence in law of the suitably rescaled height function to the solution of the additive stochastic heat equation.  Our approach employs a direct recursive renormalization scheme.  A central component is a graph representation of the renormalization terms, which enables the derivation of the high-moment bounds essential to the renormalization procedure.
\end{abstract}

%\tableofcontents

\section{Introduction}\label{S.Introduction}
\subsection{Models and main results}

In this paper, we consider a discrete 1D random growth surface proposed by Adhikari and Chatterjee \cite{adhikari2024invariance}. Let $\bbZ_{\ge 0}$ be the set of nonnegative integers and $\bbZ_+$ be the set of positive integers. 
Let $\driving:\bbR\to\bbR$ be a  driving function and $\{\noise(t,x)\}_{(t,x)\in\bbZ_+\times\bbZ}$ be a collection of independent random variables. 

\begin{definition}\label{defn:random growth surface}
The height function of the 1D random growth surface $f_N:\mathbb Z_{\ge0}\times \mathbb Z\to \mathbb R$ is defined by the recursion for $t\ge 1$,
\begin{equation}\label{E.def of f}
	\function(t,x)=\frac{\function(t-1,x-1)+\function(t-1,x+1)}{2}+\driving(\function(t-1,x-1)-\function(t-1,x+1))+N^{-\frac{1}{4}-\rate}\noise(t,x),
\end{equation}
with zero initial condition, i.e., $\function(0,x)=0,~\forall x\in\bbZ.$ 
\end{definition}
We impose the following assumptions:

\begin{itemize}
    \item\textbf{\emph{Distributions.}} Let $\distribution{t}, t \in \bbZ_+$ be a sequence of distributions such that $\noise(t,x)\sim\distribution{t},x\in\bbZ$. All $\distribution{t}$ have mean zero and the same second and third moments. 

    \item\textbf{\emph{Finite moments.}} For $k=2,3$, let $\mu_k=\bbE[\noise(1,0)^k]$. Also, $\sup_{t\in\bbZ _+}\bbE [|\noise(t,0)|^k]<\infty$ for all $0<k\le M$, where $M> 3$.
    \item\textbf{\emph{Regularity.}} The driving function $\driving$ is even, satisfies $\driving(0)=0$ and $\driving''(0)=\beta/4$, where $\beta$ is a fixed constant. $\driving$ is $C^\alpha$ in a neighborhood of 0 for some integer $\alpha \ge 2$.
\end{itemize}

\begin{example}
    We list three classical examples of $\eqref{E.def of f}$ for which $\driving$ is an even $C^{\infty}$ function.

    \begin{itemize}
        \item[(1)] Take $\driving(u)=u^2$. The model represents a discrete version of the KPZ growth behavior.
        \item[(2)] Take $\driving(u)=\sqrt{1+u^2}$. The model was considered in the original paper of Kardar, Parisi, and Zhang \cite{kardar1986dynamic}.
        \item[(3)] Take    
        $$
        \driving(u)=\frac{1}{\beta}\log\left(\frac{\mathrm{e}^{\beta u/2}+\mathrm{e}^{-\beta u/2}}{2}\right).
        $$
        This choice corresponds to the 1D directed polymer model. It is straightforward to verify that the corresponding normalized log-partition function is given by
        \begin{equation*}
    	f_N^{\text{poly}}(t,x)=\frac{1}{\beta} \log
        \left(\frac{1}{2^t}\sum_{Y\in \text{RW}(t,x)}\prod_{s=1}^t \exp{\{\beta N^{-\frac{1}{4}-\rate}\noise(s, Y(s))\}}\right), 
    \end{equation*}
    where $\text{RW}(t,x)$ denotes the set of all simple symmetric random walk paths on $\mathbb{Z}$ that start at time $0$ and terminate at position $x$ at time $t$. 
    \end{itemize}
\end{example} 

For comparison, another important choice is $\phi(u)=|u|/2$, although it falls outside the regularity assumptions above. In this case, \eqref{E.def of f} becomes

\begin{equation*}
	\function(t,x)=\max\{ \function(t-1,x-1),\function(t-1,x+1)\}+N^{-\frac{1}{4}-\rate}\noise(t,x).
\end{equation*}
The solution can be represented as the maximum of accumulated random variables over directed paths, yielding a directed last passage percolation (DLPP) model; see \cite{Rost1981,Joh2000}.

The critical intermediate disorder scale for the 1D directed polymer model is $N^{-1/4}$, at which the  rescaled partition function converges to the solution of the multiplicative stochastic heat equation \cite{alberts2014intermediate}. The same \(N^{-1/4}\) critical scaling was established in \cite{adhikari2024invariance} for the 1D random growth surface in Definition \ref{defn:random growth surface}. 
Accordingly, we refer to the cases $\rate>0$, $\rate=0$, and $\rate<0$ as the subcritical, critical, and supercritical regimes, respectively.  In the subcritical regime, the nonlinear contributions become negligible under the appropriate rescaling, leading to Edwards--Wilkinson fluctuations; compare \cite{dey2016high} for the 1D directed polymer model.

In this paper, we focus on the subcritical regime $0<\rate<1$. We establish that, after subtracting the renormalization terms and rescaling,  $\function(t,x)$  converges to a solution of the stochastic heat equation (SHE) with additive noise. For any $(t,x)\in\bbR_+\times\bbR$ such that $t$ is an integer multiple of $N^{-1}$ and $x$ is an integer multiple of $N^{-1/2}$, let 
\begin{equation}\label{E.def of F_N}
F_N(t,x)=N^{\rate}\function(Nt,\sqrt{N}x)-\frac{\beta}{2}\mu_2N^{\frac{1}{2}-\rate }t-\frac{\beta^2}{6}\mu_3N^{\frac{1}{4}-2\rate }t.
\end{equation}
For all other $(t,x)$, $F_N(t,x)$ is defined by the space-time linear interpolation on the lattice $N^{-1}\mathbb{Z}_{\ge 0}\times N^{-1/2}\mathbb{Z}$ as in \cite[Section 10]{adhikari2024invariance}. We equip $C(\bbR\times\bbR_+)$ with the topology of uniform convergence on compact sets.
\begin{theorem}\label{T.main}
    For any fixed constant $\rate \in(0,1)$, let $F_N$ be defined in \eqref{E.def of F_N}. If $M\ge 8$ and $\alpha\ge4/\rate$, then $F_N$ converges in law to the solution $f$ of the stochastic heat equation with additive noise (additive SHE)
    \begin{equation}\label{E.heat equation}
    \partial_t f=\frac{1}{2}\partial_x^2 f+\sqrt{2\mu_2}\xi,\quad f(0,\cdot)=0,
    \end{equation}
    where $\xi$ is standard space-time white noise. 
\end{theorem}
We refer to \cite[Section~3.2]{alberts2014intermediate} for a precise definition of the space-time white noise $\xi$ in~\eqref{E.heat equation}.  The solution of~\eqref{E.heat equation} is given explicitly by the stochastic convolution of the heat kernel with $\xi$. Our results can also be adapted to other initial conditions and to models with boundary conditions. In the critical regime, one example is the Neumann boundary condition considered in \cite{tang2024invariance}.

We obtain the same conclusion under a different parameter regime $(M,\alpha)$.

\begin{theorem}\label{T.main 2}
    For any fixed constant $\rate\in(0,1)$, if $M\ge 20/\rate$ and $\alpha\ge4$, then the conclusion of Theorem \ref{T.main} also holds.

\end{theorem}
\subsection{Background and related results}
The 1D KPZ equation \cite{kardar1986dynamic}, which describes the evolution of a random interface growth model, is a stochastic partial differential equation given by
\begin{equation}\label{eq:KPZ}
    \partial_t h = \frac{1}{2} \partial_x^2 h + \frac{\beta}{2} (\partial_x h)^2 + \gamma \xi,\quad  t\ge 0, x\in\R, 
\end{equation}
where $\xi$ is a space-time white noise, and $\beta, \gamma$ are real-valued parameters. For $\beta\neq 0$, one can overcome the difficulties of the singular term $(\partial_x h)^2$ by giving a rigorous meaning to \eqref{eq:KPZ} through the Cole--Hopf transform $h=\beta^{-1}\log Z$, where $Z$ solves the stochastic heat equation with multiplicative noise (multiplicative SHE)
\begin{equation}
\label{eq:SHE multi noise}
\partial_t Z = \frac{1}{2} \partial_x^2 Z + \beta \gamma  Z \xi,\quad  t\ge 0, x\in\R.
\end{equation}
This is the so-called Cole-Hopf solution of the 1D KPZ equation; see \cite{bertini1997stochastic}. A pathwise solution theory was later developed by \cite{hairer2013solving},  followed by the general framework of regularity structures \cite{hairer2014theory}.

Research on the Kardar-Parisi-Zhang (KPZ) universality class has developed rapidly over the past two decades and has become an important topic in mathematics and physics. Fluctuations in the KPZ universality class were first established in the longest increasing subsequence problem  \cite{BDJ1999} and subsequently in exactly solvable directed last passage percolation \cite{Joh2000}. For a review of recent progress, see \cite{baik2022kpz} and \cite{Quastel2026KPZ}.  A central goal is to understand the strong KPZ universality conjecture, which states that for a large class of random interface growth processes, the height functions under the $3:2:1$ time-space-fluctuation scaling should converge to a universal  random field, also known as the KPZ fixed point  \cite{matetski2021kpz}.

The weak universality conjecture is that a class of 1D random growth surface driven by microscopic fluctuations, with the local growth depending nontrivially on neighboring heights should converge to the KPZ equation under appropriate scaling. A related class of continuum interface models is obtained by replacing $(\partial_x h)^2$ by a general nonlinearity $F(\partial_x h)$. Under suitable regularity conditions on $F$,  rescaled solutions of these models converge to solutions of the KPZ equation \cite{hairer2018class,hairer2019large,kong2025frequency}.

Related work concerns the convergence of discrete and microscopic models to stochastic PDEs. General frameworks and applications for discretizations of singular SPDEs can be found in \cite{cannizzaro2018space,erhard2019discretisation,hairer2018discretisations,erhard2024scaling}. At the microscopic level, \cite{diehl2017kardar} proved that the large-scale fluctuations of weakly asymmetric interacting Brownian motions converge to the KPZ equation via a martingale formulation. These results motivate analogous scaling limits for discrete growth models. The random growth surface considered here is defined by a discrete nonlinear recursion with independent space-time noise variables, rather than as a discretization of an SPDE or an interacting diffusion system.

Another related discrete model is the 1D directed polymer in a random environment, first introduced by Huse and Henley \cite{huse1985pinning} as a model for interfaces in disordered Ising systems;  see \cite{comets2017directed,zygouras2024directed} for reviews. At the critical scaling \(\rate=0\), corresponding to the intermediate disorder scaling for directed polymers, the  rescaled point-to-point partition function, viewed as a space-time random field, converges in law  to the solution of the multiplicative SHE. This was first established by \cite{alberts2014intermediate}, after which \cite{caravenna2017polynomial} extended the convergence at fixed space-time points to long-range directed polymer models.  Both results were initially established under finite exponential moment assumptions on disorder variables.  For the point-to-line partition function at a fixed macroscopic time,  this convergence was later established under a finite six-moment assumption in \cite{dey2016high}, confirming a conjecture of \cite{alberts2014intermediate}.  \cite{dey2016high} also obtained Gaussian fluctuation limits in the subcritical regime.

Recently, \cite{adhikari2024invariance} extended the critical regime convergence result of \cite{alberts2014intermediate} from the 1D directed polymer model to the random growth surface with a general driving function defined in \eqref{E.def of f}. The result assumes that the driving function $\driving$ is $C^6$.  Their proof is based on an invariance principle that compares the random growth surface with the directed polymer model through a renormalization argument, and requires  finite exponential moments of random variables.

In this paper, we study a class of 1D KPZ random growth surfaces \eqref{E.def of f} in the subcritical regime. We assume that the random variables have uniformly bounded eighth moments and are identically distributed in space. Under these assumptions, we show that with high probability, the height function can be represented as a sum of linear and multilinear polynomials in the random variables up to a small error. This representation implies convergence in law of the rescaled and renormalized height function to the solution of the additive SHE. Our proof is based on a direct recursive renormalization argument, without comparison to the directed polymer model. A key ingredient is a combinatorial method, which we call graph lifting, for obtaining the required high-moment bounds on the renormalization terms.

\subsection{Proof strategy}

To outline the proof of Theorem \ref{T.main}, we start with the simplest case where $\driving(x)=x^2$. The proof for general $\driving$ follows the same strategy. We first reduce to working with the truncated version
\[
\truncatednoise(t,x)
:=
\noise(t,x) \mathbf{1}_{\{|\noise(t,x)|\leq N^{1/4}\log N\}}-
\mathbb{E}[\noise(t,0)
\mathbf{1}_{\{|\noise(t,0)|\leq N^{1/4}\log N\}}].
\]
Then, we directly analyze the solution to

\begin{equation}\label{E.simplest equation}
    \function(t,x)=\frac{\function(t-1,x-1)+\function(t-1,x+1)}{2}+(\function(t-1,x-1)-\function(t-1,x+1))^2+N^{-\frac{1}{4}-\rate}\truncatednoise(t,x),
\end{equation}
by introducing a sequence of renormalization terms, defined recursively by substituting lower-order terms back into the equation above.
The first renormalization term $\xterm{1}$ is defined to satisfy

$$\xterm{1}(t,x)=\frac{\xterm{1}(t-1,x-1)+\xterm{1}(t-1,x+1)}{2}+N^{-\frac{1}{4}-\rate}\truncatednoise(t,x).$$
Equivalently, $\xterm{1}$ is a linear combination of random variables
\begin{equation*}
    \xterm{1}(t,x)=\sum_{s=1}^t\sum_{z\in\bbZ}p(t-s,x-z)N^{-\frac{1}{4}-\rate}\truncatednoise(s,z),
\end{equation*}
where $p(\cdot,\cdot)$ is the transition probability for one-dimensional simple random walk.

Subtracting the equation for \(X_N^{(1)}\) from \eqref{E.simplest equation}, the noise term \(N^{-1/4-\rate}\truncatednoise(t,x)\) cancels, while the nonlinear term $\bigl(\function(t-1,x-1)-\function(t-1,x+1)\bigr)^2$ remains.  We next approximate this nonlinear term by replacing \(\function\) with \(X_N^{(1)}\) and define the second renormalization term by
$$\xterm{2}(t,x)=\frac{\xterm{2}(t-1,x-1)+\xterm{2}(t-1,x+1)}{2}+(\xterm{1}(t-1,x-1)-\xterm{1}(t-1,x+1))^2.$$
We then subtract both $\xterm{1}$ and $ \xterm{2}$ from $\function$ and continue the construction. This iterative procedure can be continued through $\xterm{L}$. The error term $\function-\xterm{1}-\dots-\xterm{L}$ also admits a recursive formula which involves the discrete gradients $\kterm{l}:=\xterm{l}(t-1,x-1)-\xterm{l}(t-1,x+1),~1\le l\le L$. This reduces the error estimate to obtaining high-moment bounds on $\kterm{l}$, which are stated in Theorem \ref{T.moments upper bound on K 1}.

The main novelty is a combinatorial framework based on tree expansions and graph lifting  developed in Sections \ref{S.Tree representation and reductions} and \ref{S.Graph Lifting}.
This framework is used to prove Theorem \ref{T.moments upper bound on K 1} by representing the high-moment expansions of $\kterm{l}$ as graph sums. 
In Section \ref{S.Tree representation and reductions}, we represent each $\left(\kterm{l}(t,x)\right)^n$  as a linear combination of weighted rooted-tree actions in Corollary \ref{C.kterm represented by tree}: 
\begin{equation*}\left(\kterm{l}(t,x)\right)^n=\sum_{\tree\in\treefamily_l^n}c(\tree,l,n)\action{t,x}{\tree}.
\end{equation*}
Here $\treefamily_l^n$ is a family of trees determined by $l,n,$ and $c(\tree,l,n)$ are constants. $\action{t,x}{\tree}$ is obtained by summing over all the ways to embed $\tree$ in the lattice $\bbZ_+\times\bbZ$ with the root fixed at $(t,x)$. For each such embedding, the contribution to $\action{t,x}{\tree}$ is given by the product of random variables $N^{-1/4-\rate}\truncatednoise$ at each leaf and the corresponding values of the discrete gradient SRW kernel \(\cha\) along each edge. After expanding moments, one is led to sums indexed by partitions of the leaves, corresponding to possible identifications of random variables. The problem is then reduced to estimating the contribution of each tree together with each proper partition on the leaves-glued graph.

Combining this representation with the lemmas in Appendix \ref{S.Some Lemmas} and ignoring the influence of the logarithmic term, the basic power counting is as follows: each edge  encoded $\cha (t,\cdot)$  contributes no more than order  $O(t^{-1})$, whereas summation over the time-space variable associated with each non-root vertex contributes order $N^{1+\frac12}=N^{\frac32}$. For example,
\[
    \sum_{\substack{1\le t\le N \\ x\in \mathbb Z}}|\cha(t,x)|\lesssim \sum_{1\le t\le N}t^{-\frac{3}{2}}\sum_{x\in\mathbb Z} |x|\text{e}^{-\frac{x^2}{2t}} \lesssim\sum_{1\le t\le N} t^{-\frac{1}{2}}\lesssim\sqrt{N}=N^{-1}\cdot N^{\frac{3}{2}}.
\]
In fact, the conditions of proper partitions in Section~\ref{S.Tree representation and reductions} and graph lifting in Section~\ref{S.Graph Lifting} are precisely to guarantee that this counting can be carried out: each contribution of leaves-glued graph with partition $\pi$ and $|L(T)|$ leaves is bounded by $
N^{(|L(T)|-2-\sum_{\pi_i\in\pi} (|\pi_i|-4)^+)/4}$. With $|\truncatednoise|\loglesssim N^{1/4}$ (the notation $\loglesssim$ is defined in Subsection \ref{S.notation}),  a block \(\pi_i\) of size $|\pi_i|$ contributes $N^{(|\pi_i|-8)^+/4}$. Therefore each proper partition is bounded by
$$
N^{-(\frac{1}{4}+\rate)|L(T)|}N^{\frac{1}{4}\sum_{\pi_i\in\pi}(|\pi_i|-8)^+}N^{\frac{1}{4}(|L(T)|-2-\sum_{\pi_i\in\pi} (|\pi_i|-4)^+)}
\loglesssim N^{-\rate |L(T)|}.
$$
For any even $n$ and any tree $\tree$ appearing in the expansion of $\mathbb E [(\kterm{l}(t,x))^n]$, we have $|L(T)|\ge ln$. Hence $$ \mathbb E [(\kterm{l}(t,x))^n] \loglesssim N^{-ln\rate}.$$ In particular, with  probability $1-o(1)$, $|\kterm{l}(t,x)|\le N^{-0.6l\rate}$ uniformly on the relevant region.

With these moment bounds, we can approximate $\function$ by the sum of the first $L$ renormalization terms $\xterm{1}+\dots+\xterm{L}$, with an error of order $o(N^{-\rate})$. We then show that, except for $\xterm{1}$, all these renormalization terms have fluctuations of order $o(N^{-\rate})$. This is also proved by bounding high moments. Finally, we compute the expectations of the renormalization terms explicitly, which completes the proof of Theorem \ref{T.main}.

\subsection{Structure of the paper}
The rest of the paper is organized as follows.  In Section \ref{S.Renormalization}, we prove Theorem \ref{T.main} using Theorems \ref{T.moments upper bound on K 1}, \ref{T.expectation of Xl 1}, and \ref{T.concentration on Xl 1}, whose proofs are given in the next three sections.  In Section \ref{S.Tree representation and reductions}, we use a graph representation to reduce moment bounds to estimates on sums over graphs. Then, in Sections \ref{S.Graph Lifting} and \ref{S.expectation and concentration of X} we use combinatorial techniques to estimate these graph sums and prove Theorems \ref{T.moments upper bound on K 1}, \ref{T.expectation of Xl 1}, and \ref{T.concentration on Xl 1}. Finally, Theorem \ref{T.main 2} is established in Section \ref{S.Proof of T.main 2}.

\subsection{Notations}\label{S.notation}
\paragraph{Geometry} For any $a,b>0$, we define a rectangular box
$$\boxdomain:=([0,aN]\times[-bN,bN])\cap(\bbZ_{\ge 0}\times\bbZ).$$
We also define the trapezoid-shaped domain
$$\trapezoid:=\{(t,x)\in\bbZ_{\ge 0}\times\bbZ:0\le t\le aN,-(a+b)N+t\le x\le(a+b)N-t\}.$$
\paragraph{Conventions} Throughout the rest of the paper, we will adopt the convention that $A\lesssim B$ 
means that $A\leq C B$ and 
$A\loglesssim B$ 
means that $A\leq C(\log N)^C B$ for some deterministic positive real number $C$ that does not depend on $N,x$, or $t$, as long as $(t,x)$ is in some given rectangle of the form $\boxdomain$. Here, $N$ is the parameter that we will eventually send to infinity, and $x$ and $t$ are specific choices of space and time points where we want to prove something. We will write $A=O(B)$ if $|A|\lesssim|B|$, and $A = o(B)$ if $A/B\to 0$ uniformly over $(t,x)\in\boxdomain$ as $N\to\infty$. We will often write sentences like: ``with probability $1-o(1)$, for all $(t,x)\in\boxdomain$,  

$$|G(t,x)| \loglesssim N^{-\alpha},$$
where $G$ is a random function and $\alpha$ is a constant." This means that there is an event $\Omega$ that may potentially vary with $N$, with $\mathbb{P}(\Omega)\to 1$ as $N\to\infty$, and there is some deterministic constant $C$, independent of $N$, such that on $\Omega$ we have
$$\max_{(t,x)\in\boxdomain} |G(t,x)| \leq C (\log N)^C N^{-\alpha}.$$
Similarly, we say ``with probability $1-o(1)$, for all $(t,x)\in\boxdomain$,
$$|G(t,x)| =o(N^{-\alpha})$$
where $G$ is a random function and $\alpha$ is a constant.’’   This means that there exist an event $ \Omega $ that may potentially vary with $ N $, with $ \mathbb{P}(\Omega) \to 1 $ as $ N \to \infty $, and a sequence of constants $c_{N}\downarrow 0$ such that on $\Omega$ we have 
$$\max_{(t,x)\in\boxdomain} |G(t,x)| \leq c_{N} N^{-\alpha}.$$

\paragraph{\textbf{Operators}} For a real function $f=f(t,x)$ on $\bbZ_{\ge 0}\times\bbZ$ with $f(0,x)=0~\forall x\in\bbZ$, define the operators $\average$ and $\difference$ for $t\ge 1$ by 
$$\average f(t,x)=\frac{f(t-1,x-1)+f(t-1,x+1)}{2},\quad \difference f(t,x)=f(t-1,x-1)-f(t-1,x+1).$$

\section{Renormalization}\label{S.Renormalization}
In this section, we introduce a sequence of renormalization terms that yields an expansion for $\function$ with a small error term. We then prove Theorem \ref{T.main} assuming several moment bounds for the renormalization terms. These bounds are proved in the following sections.

\subsection{Definition of renormalization terms}
Intuitively, $\difference\function(t,x)$ is of order $o(1)$ with probability $1-o(1)$, which allows us to expand the $\driving$ term in \eqref{E.def of f}. Suppose $\driving$ has the following Taylor expansion near 0,

$$\driving(x)=\sum_{k=2}^{\alpha-1}a_kx^k+R(x):=P(x)+R(x), $$
where $a_k=\driving^{(k)}(0)/k!$ ($a_2=\beta/8$) are constants, $R(x)$ is the Taylor remainder, and $P(x)=\sum_{k=2}^{\alpha-1}a_kx^k$ is a polynomial. The remainder $R$ satisfies the bound $\absolute{R(x)}\le C|x|^{\alpha}$ in a neighborhood of $0$. Thus, \eqref{E.def of f} can be rewritten as 
\begin{equation}\label{E.redefine f}
\function(t,x)=\average\function(t,x)+N^{-\frac{1}{4}-\rate}\noise(t,x)+P(\difference\function(t,x))+R(\difference\function(t,x)).
\end{equation}
Heuristically, $R(\difference\function(t,x))$ is typically small. We therefore omit the last term in \eqref{E.redefine f} and regard $\function$ as satisfying a polynomial recursion in $\noise$. This suggests that we approximate $\function$ by a deterministic polynomial of high degree in $\noise$. The most naive approximation is by ignoring $P$ and $R$ together. We write $\xterm{1}$ for the solution to 

$$\xterm{1}(t,x)=\average\xterm{1}(t,x)+N^{-\frac{1}{4}-\rate}\noise(t,x),$$
or, equivalently,
\begin{equation*}
    \xterm{1}(t,x)=\sum_{s=1}^t\sum_{z\in\bbZ}p(t-s,x-z)N^{-\frac{1}{4}-\rate}\noise(s,z).
\end{equation*}
However, this approximation is still far from the true $\function$. Let $\error{1}=\function-\xterm{1}$. Subtracting the above equation from \eqref{E.redefine f}, we obtain
$$\error{1}(t,x)=\average\error{1}(t,x)+P\left(\difference\error{1}(t,x)+\difference\xterm{1}(t,x)\right)+R(\difference\function(t,x)).$$
Here, $\xterm{1}$ is explicit. We therefore extract the contribution obtained by replacing $\difference\function(t,x)$ with $\difference\xterm{1}$ and define the second renormalization term $\xterm{2}$, which approximates $\error{1}$, by
$$\xterm{2}(t,x)=\average\xterm{2}(t,x)+P\left(\difference\xterm{1}(t,x)\right).$$
Denoting by $\error{2}$ the remaining error term $\error{1}-\xterm{2}$, we can derive the equation for $\error{2}$ and extract terms related to $\xterm{1}$ and $\xterm{2}$ to construct the third renormalization term $\xterm{3}$. Iterating this procedure, we obtain renormalization terms $\xterm{l}$ for any $l\ge 1$. The explicit formulas are given in the following definition. 
\begin{definition}\label{D.def of Xl} 
    We say that the random variables $\{\xterm{l}(t,x):l\ge 1,~(t,x)\in\bbZ_{\ge0}\times\bbZ\}$ are \emph{driven by} $\{\noise(t,x)\}_{(t,x)\in\bbZ_+\times\bbZ}$  if they satisfy the following equations
        \begin{equation}\label{E.def of X1}
            \xterm{1}(t,x)=\average\xterm{1}(t,x)+N^{-\frac{1}{4}-\rate}\noise(t,x),
        \end{equation}
        \begin{equation}\label{E.def of Xl}
        \xterm{l}(t,x)=\average\xterm{l}(t,x)+P\left(\sum_{m=1}^{l-1}\difference\xterm{m}(t,x)\right)-P\left(\sum_{m=1}^{l-2}\difference\xterm{m}(t,x)\right),\quad l\ge 2
        \end{equation}
    for any $(t,x)\in\bbZ_+\times\bbZ$ and $l\ge 1$, 
    with initial values $\xterm{l}(0,x)=0$ for all $x\in\bbZ$ and $l\ge 1$. 
\end{definition}
We also define the error terms as follows.
\begin{definition}\label{D.def of error l}
    With $\{\xterm{l}(t,x):l\ge 1,(t,x)\in\bbZ_{\ge0}\times\bbZ\}$ defined above, we define the error terms $\{\error{l}(t,x):l\ge 1,~(t,x)\in\bbZ_{\ge0}\times\bbZ\}$ by
    \begin{equation}\label{E.def of error l}
        \error{l}(t,x)=\function(t,x)-\sum_{m=1}^{l}\xterm{m}(t,x).
    \end{equation}
\end{definition}
Note that the error terms $\error{l}(t,x)$ satisfy the recursive formula
\begin{equation}\label{E.redefine of error l}
\begin{aligned}
    \error{l}(t,x)=\average\error{l}(t,x)&+P\left(\difference\error{l}(t,x)+\sum_{m=1}^{l}\difference\xterm{m}(t,x)\right)-P\left(\sum_{m=1}^{l-1}\difference\xterm{m}(t,x)\right)\\
    &+R(\difference\function(t,x)).
\end{aligned}
\end{equation}
If we can prove a uniform bound of the form $\error{L}(t,x)=o(N^{-\rate})$, then the following expansion holds:
$$\function(t,x)=\sum_{l=1}^{L}\xterm{l}(t,x)+o(N^{-\rate}).$$
The preceding argument remains heuristic, as the smallness of
$\difference\function(t,x)$ is not presupposed. $R(\difference\function(t,x))$ may not be small and then $\error{l}(t,x)$ may be large. A rigorous inductive procedure is therefore required.

Since $\difference\xterm{l}$ appears frequently, we introduce the notation $\{\kterm{l}(t,x):l\ge 1,(t,x)\in\bbZ_{\ge0}\times\bbZ\}$.
\begin{definition}
     For $(t,x)\in\bbZ_{\ge0}\times\bbZ$, set $\kterm{l}(0,x)=0$  and 
    $$\kterm{l}(t,x)=\xterm{l}(t-1,x-1)-\xterm{l}(t-1,x+1),\qquad \text{if }t\ge 1.$$
    
\end{definition}

Since $\noise$ has finite moments only up to order $M$, we will truncate the random variables to $\{\truncatednoise(t,x)\}_{(t,x)\in\bbZ_+\times\bbZ}$ in the next subsection. In the rest of the paper before the end of Section \ref{S.expectation and concentration of X}, the renormalization terms $\xterm{l}$ and $\kterm{l}$ are all driven by truncated random variables $\{\truncatednoise(t,x)\}_{(t,x)\in\bbZ_+\times\bbZ}$.
\subsection{Proof of Theorem \ref{T.main}}
Until the end of Section \ref{S.expectation and concentration of X}, we assume $M\ge8$ and $\alpha\ge4/\rate$.
We fix $a,b>0$ and define the truncated random variables on $\trapezoid$ by 
$$\truncatednoise(t,x):=\noise(t,x)\mathbf{1}_{\{|\noise(t,x)|\le N^{\frac{1}{4}}\log N\}}-c_{N,t},$$
where $c_{N,t}:=\bbE\left[\noise(t,0)\mathbf{1}_{\{|\noise(t,0)|\le N^{1/4}\log N\}}\right]$.
\begin{lemma}\label{L.truncation}
    There is an event $\Omega_{a,b}$ with $\bbP(\Omega_{a,b})=1-o(1)$ such that $\truncatednoise(t,x)=\noise(t,x)-c_{N,t}$ for all $(t,x)\in\trapezoid$.
\end{lemma}
\begin{proof}
    Let
    $$\Omega_{a,b}:=\{|\noise(t,x)|\le N^{\frac{1}{4}}\log N,\forall (t,x)\in\trapezoid\}.$$
By Markov's inequality and the existence of the eighth moment of $\noise$, a union bound gives
    $$\bbP(\Omega_{a,b}^c)\le \setsize{\trapezoid}\cdot{\bbE[\noise^8]}\cdot\left(N^{\frac{1}{4}}\log N\right)^{-8}=o(1).$$
\end{proof}
Since $\bbE[\noise]=0$, we have
$$\absolute{c_{N,t}}=\absolute{\bbE\left[\noise(t,0)\mathbf{1}_{\{|\noise(t,0)|>N^{\frac{1}{4}}\log N\}}\right]}\le\sqrt{\bbE[\noise(t,0)^2]\cdot\bbP\left(|\noise(t,0)|> N^{\frac{1}{4}}\log N\right)}=o(N^{-1}).$$
A direct induction shows that $\function(t,x)$ and $\xterm{1}(t,x)$ defined using $\noise(t,x)$ and $\noise(t,x)-c_{N,t}$ differ by exactly the constant term $N^{-1/4-\rate}(c_{N,1}+\dots+c_{N,t})$, which is of order $o(N^{-1/4-\rate})$ as long as $t\le aN$. The remaining renormalization terms ${\xterm{l}}, {l\ge 2}$ remain unchanged. Hence, we can restrict our discussion to $\noise-c_{N,t}$, and therefore to $\truncatednoise$ via Lemma \ref{L.truncation}. In the rest of the paper except for Section \ref{S.Proof of T.main 2}, we will always assume $\{\xterm{l}\}_{l\ge 1}$ are driven by $\{\truncatednoise(t,x)\}_{(t,x)\in\bbZ_+\times\bbZ}$ instead of $\{\noise(t,x)\}_{(t,x)\in\bbZ_+\times\bbZ}$, and $\kterm{l}=\difference\xterm{l}$.

After truncation, $\xterm{l}$ and $\kterm{l}$ satisfy the following upper bound on the moments of $\kterm{l}$. The proof is given in Section \ref{S.Graph Lifting}.
\begin{theorem}\label{T.moments upper bound on K 1}
     For any fixed $l$ and even $ n\ge 2$, we have
    $$\bbE \left[\left(\kterm{l}(t,x)\right)^n\right]\loglesssim N^{-ln\rate}.$$
\end{theorem}
Under the same assumption on $\xterm{l}$ and $\kterm{l}$, the following upper bound is a direct corollary of Theorem \ref{T.moments upper bound on K 1}.
\begin{corollary}\label{C.upper bound on K 1}
    For fixed $a,b>0$ and $l\ge 1$,
    $$\bbP\left(\left|\kterm{l}(t,x)\right|\le N^{-0.6l\rate},\forall~(t,x)\in\trapezoid\right)=1-o(1).$$
\end{corollary}
\begin{proof}
    For any fixed $(t,x)\in\trapezoid$,
    we take $n=2\lceil5/\rate\rceil$ in Theorem \ref{T.moments upper bound on K 1} and apply Markov's inequality to obtain
    $$\bbP\left(\left|\kterm{l}(t,x)\right|\ge N^{-0.6l\rate}\right)\le\frac{\bbE\left[(\kterm{l}(t,x))^n\right]}{N^{-0.6ln\rate}}\loglesssim N^{-0.4 ln\rate}\le N^{-3}.$$
    A union bound over $(t,x)\in\trapezoid$ gives the desired result.
\end{proof}
Assuming Corollary \ref{C.upper bound on K 1} holds, we can expand $\function$ into a sum of $\xterm{l}$.
\begin{theorem}\label{T.expansion 1}
     Let $L=\max\{2, \lceil4/\rate\rceil\}$, and let $a,b>0$ be fixed. Then there is an event $\Omega_e$ with $\bbP(\Omega_e)=1-o(1)$, such that on $\Omega_e$,    $\truncatednoise(t,x)=\noise(t,x)-c_{N,t}$ for all $(t,x)\in\trapezoid$,
     \begin{equation}\label{E.expansion 1}
     \error{L}(t,x)=o(N^{-\rate}).
     \end{equation}
     %i.e., the following expansion is valid 
     That is, on $\Omega_e$
     \begin{equation}\label{E.expansion 2}
     \function(t,x)=\sum_{l=1}^{L}\xterm{l}(t,x)+o(N^{-\rate}),\qquad\forall(t,x)\in\trapezoid.
    \end{equation}
\end{theorem}
\begin{proof}
    The following inductive argument was originally introduced in \cite{adhikari2024invariance}.
    Let
    $$\Omega_K=\left\{\absolute{\kterm{l}(t,x)}\le N^{-0.6l\rate},\forall~(t,x)\in\trapezoid~\text{and}~1\le l\le L\right\}$$
    and take $\Omega_e=\Omega_{a,b}\cap\Omega_K$. We will prove by induction on $t$ that, on $\Omega_e$, for all $(t,x)\in\trapezoid$,
    \begin{equation}\label{E.control on error 1}
        \absolute{\error{L}(t,x)}\le tN^{-1-1.1\rate},
    \end{equation}
   provided that $N\ge N_0$, where $N_0$ is a deterministic threshold depending only on $a,b,\alpha,\rate$ and $\driving$. Throughout, we will work under the assumption that $\Omega_e$ holds. Note that \eqref{E.control on error 1} holds trivially when $t=0$, since $\error{L}(0,x)=0$ for all $x$. Fix some $(t,x)\in\trapezoid$, and assume \eqref{E.control on error 1} has been proved up to time $t-1$. Since $(t-1,x-1)$ and $(t-1,x+1)$ also belong to $\trapezoid$, by the induction hypothesis,

    $$\absolute{\error{L}(t-1,x-1)},\absolute{\error{L}(t-1,x+1)}\le (t-1)N^{-1-1.1\rate}.$$
    Recall that by definition, $\function=\xterm{1}+\dots+\xterm{L}+\error{L}$.   Therefore, 
    \begin{equation}\label{E.upper bound on difference function}
    \begin{aligned}    \absolute{\difference\function(t,x)}=\absolute{\sum_{l=1}^L\kterm{l}(t,x)+\difference\error{L}(t,x)}&\le\sum_{l=1}^{L}N^{-0.6l\rate}+2(t-1)N^{-1-1.1\rate}\\
    &\le LN^{-0.6\rate}+2aN^{-1.1\rate}.
    \end{aligned}
    \end{equation}
    Suppose $\driving$ is $\alpha$ times continuously differentiable in an open interval $U$ that contains 0.  
    $R_0:=\underset{x\in U}{\sup}\left\{\absolute{\driving^{(\alpha)}(x)/\alpha!}\right\}$. Then, $\difference\function(t,x)\in U$, and $\absolute{R(\difference\function(t,x))}\le R_0\absolute{\difference\function(t,x)}^\alpha$, provided $N\ge N_1$, where $N_1$ is a deterministic threshold depending only on $a,b,\alpha,\rate$ and $\driving$.
    
    Using the recursion formula \eqref{E.redefine of error l} and recalling that $P(x)=\sum_{k=2}^{\alpha-1}a_kx^k$ is a polynomial, we obtain
    \begin{align*}
        \error{L}(t,x)=&\average\error{L}(t,x)+P\left(\difference\error{L}(t,x)+\sum_{m=1}^{L}\kterm{m}(t,x)\right)-P\left(\sum_{m=1}^{L-1}\kterm{m}(t,x)\right)+R(\difference\function(t,x))\\
        =&\average\error{L}(t,x)+\difference\error{L}(t,x)\cdot\Bigg[\sum_{k=2}^{\alpha-1}a_k\sum_{j=1}^{k}\binom{k}{j}        \left(\difference\error{L}(t,x)\right)^{j-1}\left(\sum_{l=1}^{L}\kterm{l}(t,x)\right)^{k-j}\Bigg]\\
        &+\kterm{L}(t,x)\cdot\Bigg[\sum_{k=2}^{\alpha-1}a_k\sum_{j=1}^{k}\binom{k}{j}\left(\kterm{L}(t,x)\right)^{j-1}\left(\sum_{l=1}^{L-1}\kterm{l}(t,x)\right)^{k-j}\Bigg]+R(\difference\function(t,x)).
    \end{align*}  
    The sum of the first two terms can be written as
    $$\left(\frac{1}{2}+B\right)\error{L}(t-1,x-1)+\left(\frac{1}{2}-B\right)\error{L}(t-1,x+1)$$
    where $B$ denotes the expression in the first square brackets, which, by the induction hypothesis and the definition of $\Omega_e$, satisfies $|B|\le 1/4$,  provided $N\ge N_2$, where $N_2$ is a deterministic threshold depending only on $a,b,\alpha,\rate$ and $\driving$. 

    For the third term, the coefficient multiplying $\kterm{L}(t,x)$ is also bounded by $1$, provided $N\ge N_3$, where $N_3$ is a deterministic threshold depending only on $a,b,\alpha,\rate$ and $\driving$. In this case, the third term is bounded by
    $$\absolute{\kterm{L}(t,x)}\le N^{-0.6L\rate}\le N^{-2.4}.$$   
    By \eqref{E.upper bound on difference function}, the fourth term is bounded by
    $$\absolute{R(\difference\function(t,x))}\le R_0(LN^{-0.6\rate}+2aN^{-1.1\rate})^\alpha\le N^{-1-1.1\rate}/2,$$
    provided $N\ge N_4$, where $N_4$ is a deterministic threshold depending only on $a,b,\alpha,\rate$ and $\driving$. Thus, if $N\ge \max\{N_1,N_2,N_3,N_4\}$, we have
    \begin{align*}
         \absolute{\error{L}(t,x)}
         &\le \absolute{\left(\frac{1}{2}+B\right)\error{L}(t-1,x-1)+\left(\frac{1}{2}-B\right)\error{L}(t-1,x+1)}+N^{-2.4}+\frac{1}{2}N^{-1-1.1\rate}\\
         &\le \max\left\{\absolute{\error{L}(t-1,x-1)},\absolute{\error{L}(t-1,x+1)}\right\}+N^{-1-1.1\rate}\le tN^{-1-1.1\rate}
    \end{align*}
    Thus, if we choose $N_0=\max\{N_1,N_2,N_3,N_4\}$, the induction step follows, completing the proof of \eqref{E.control on error 1}. Note that \eqref{E.expansion 1} is a direct corollary of \eqref{E.control on error 1}, and \eqref{E.expansion 2} follows from the definition of $\error{L}$. This completes the proof.

\end{proof}

Theorem \ref{T.expansion 1} shows that the renormalization procedure is an appropriate method to approximate $\function$. The next two theorems control the behavior of $\xterm{l}$. Their proofs are given in Section \ref{S.expectation and concentration of X}.
\begin{theorem}\label{T.expectation of Xl 1}
    Let $\mu_k=\bbE[\noise^k]$ and fix an $a>0$. For any $x$ and $t\le aN$, we have
    \begin{align*}
        \bbE[\xterm{2}(t,x)]=&\frac{\beta}{2}\mu_2N^{-\frac{1}{2}-2\rate}t+o(N^{-\rate}),\qquad \
        \bbE[\xterm{3}(t,x)]=\frac{\beta^2}{6}\mu_3N^{-\frac{3}{4}-3\rate}t+o(N^{-\rate}),\\
        \bbE[\xterm{l}(t,x)]=&o(N^{-\rate}), \quad \forall l\ge 4.
    \end{align*}
\end{theorem}

\begin{theorem}\label{T.concentration on Xl 1}
   For any fixed $l\ge 2$ and even $n\ge 2$, the following estimate holds for any $x$ and $t\le aN$:
    $$\bbE \left[\left(\xterm{l}(t,x)-\bbE[\xterm{l}(t,x)]\right)^n\right]\loglesssim N^{-ln\rate}.$$
\end{theorem}
The following upper bound is a direct corollary of Theorem \ref{T.concentration on Xl 1}. Since the proof is identical to the proof of Corollary \ref{C.upper bound on K 1}, we omit it for brevity.
\begin{corollary}\label{C.concentration of Xl}
   For fixed $a,b>0$ and $l\ge 2$, 
    $$\bbP\left(\absolute{\xterm{l}(t,x)-\bbE[\xterm{l}(t,x)]}\le N^{-1.1\rate},\forall~(t,x)~\in\trapezoid\right)=1-o(1).$$
\end{corollary}
\begin{proof}[Proof of Theorem \ref{T.main}]
    For any fixed $a,b>0$, Theorems \ref{T.expansion 1} and \ref{T.expectation of Xl 1}, together with Corollary \ref{C.concentration of Xl}, imply that, with probability $1-o(1)$, for any $(t,x)\in\trapezoid$,
    $$\function(t,x)=\xterm{1}(t,x)+\frac{\beta}{2}\mu_2N^{-\frac{1}{2}-2\rate}t+\frac{\beta^2}{6}\mu_3N^{-\frac{3}{4}-3\rate}t+o(N^{-\rate}).$$
    The proof is completed by combining this result with the convergence in law of $N^\rate\xterm{1}\left(Nt,\sqrt{N}x\right)$ (after applying the linear interpolation procedure of \cite[Section 10]{adhikari2024invariance})  to the solution $f$ of \eqref{E.heat equation}.
\end{proof}

\section{Tree representation and reductions}\label{S.Tree representation and reductions}
In this section, we represent all renormalization terms using tree expansions to provide some preliminaries for the proof of Theorems \ref{T.moments upper bound on K 1}, \ref{T.expectation of Xl 1} and \ref{T.concentration on Xl 1}. We write the transition probability of a simple random walk as
$$
    p(t,x)
    = 2^{-t}\binom{t}{\frac{t-x}{2}} \mathbf{1}_{\{t\ge0,\ |x|\le t,\ t+x \equiv 0 \ (\mathrm{mod}\ 2)\}}.
$$
The recursion \eqref{E.def of X1} is equivalent to 
\begin{equation}\label{E.redefine X1}
    \xterm{1}(t,x)=\sum_{s=1}^t\sum_{z\in\bbZ}p(t-s,x-z)N^{-\frac{1}{4}-\rate}\truncatednoise(s,z),
\end{equation}
and for $l\ge 2$, \eqref{E.def of Xl} can be rewritten as
\begin{align*}
    \xterm{l}(t,x)
    =&\sum_{s=1}^{t}\sum_{z\in\bbZ}p(t-s,x-z)\left[P\left(\sum_{m=1}^{l-1}\kterm{m}(s,z)\right)-P\left(\sum_{m=1}^{l-2}\kterm{m}(s,z)\right)\right]\\
    =&\sum_{s=1}^{t}\sum_{z\in\bbZ}p(t-s,x-z)\sum_{k=2}^{\alpha-1}a_k\left[\left(\sum_{m=1}^{l-1}\kterm{m}(s,z)\right)^k-\left(\sum_{m=1}^{l-2}\kterm{m}(s,z)\right)^k\right].\tageq \label{E.redefine Xl}
\end{align*}
Set $\cha(t,x)=p(t-1,x-1)-p(t-1,x+1)$. Recalling the definition $\kterm{l}=\difference\xterm{l}$, we have the following for $l\ge 2$:
\begin{align*}\label{E.redefine Kl}
    \kterm{l}(t,x)&=\sum_{s=1}^{t-1}\sum_{z\in\bbZ}\cha(t-s,x-z)\sum_{k=2}^{\alpha-1}a_k\left[\left(\sum_{m=1}^{l-1}\kterm{m}(s,z)\right)^k-\left(\sum_{m=1}^{l-2}\kterm{m}(s,z)\right)^k\right]\\
    &=\sum_{s=1}^{t-1}\sum_{z\in\bbZ}\cha(t-s,x-z)\sum_{k=2}^{\alpha-1}a_k\sum_{j=1}^{k}\binom{k}{j}\left(\kterm{l-1}(s,z)\right)^j\left(\sum_{m=1}^{l-2}\kterm{m}(s,z)\right)^{k-j}.\tageq 
\end{align*}
Thus, there is a natural tree representation in the definition of ${\kterm{l}}$: each $\kterm{l}$ can be represented as a linear combination of lower-order products.

\subsection{Tree representation}\label{Subs.Tree representation}
We represent each $\kterm{l}$ by a family of rooted trees arising from \eqref{E.redefine Kl}. For each tree, we define an associated quantity called its action.
\paragraph{\textbf{Trees}}

A tree $T=(V,E)$ is a finite, simple, undirected, connected graph without cycles. We always work with rooted trees, denote the root by $\vroot$, and let $V(\tree)$ and $E(\tree)$ be the set of vertices and the set of edges of $\tree$. For $v\in V(T)$, let $|v|$ be the graph distance from $v$ to $\vroot$. If $e\in E(T)$ connects $u$ and $v$ and $|u|=|v|-1$, we write $e=[u,v]$. The \emph{parent} of a vertex $v$ is the vertex connected to $v$ on the path
from $v$ to the root; every vertex has a unique parent, except the root, which has no parent. A \emph{child} of $v$ is a vertex having $v$ as parent. A \emph{leaf} is a vertex with no children. Let $L(\tree)$ denote the set of all leaves of $\tree$. A \emph{star} $S_n$ is a tree with exactly $n$ leaves and $n$ edges connecting every leaf to the root.

For a rooted tree $\tree$ with root $\vroot$, we define its \emph{extension} $\extension{\tree}$ by adding a new vertex $v$ to $\tree$ and a new edge connecting $v$ and $\vroot$, and then identifying $v$ as the root of $\extension{\tree}$.
For rooted trees $\tree_1,\ldots,\tree_k$, their \emph{one-point union} is a rooted tree formed by identifying their roots $\vroot_1,\ldots, \vroot_k$ into one vertex, which becomes the root of this new tree. We will write it as 
$$\oneptunion_{i=1}^{k}\tree_i.$$
When $k=2$, we also write $T_1\vee T_2$. Given any set of trees $\treefamily$, define its $n$-th power by
\begin{equation}\label{E.def of treefamily^n}
\treefamily^{n}:\left\{\oneptunion_{i=1}^n\tree_i: \tree_1,\ldots,\tree_n\in\treefamily\right\}.
\end{equation}
Figure \ref{F.Illustrations of rooted trees} illustrates the above definitions.
\begin{figure}[htbp]
\centering
\resizebox{\textwidth}{!}{%
\begin{tikzpicture}[
    x=0.7cm, y=0.7cm,
    vertex/.style={circle, fill=black, inner sep=1.5pt},
    bedge/.style={draw=black,line width=0.95pt},
    lab/.style={font=\footnotesize},
    title/.style={font=\footnotesize},
    >=Latex
]

% =========================================================
% (a) rooted tree T
% =========================================================
\node[vertex] (r)  at (0,0) {};
\node[vertex] (r1) at (-1.1,-1.1) {};
\node[vertex] (r2) at (0.7,-1.1) {};
\node[vertex] (r3) at (-1.8,-2.2) {};
\node[vertex] (r4) at (-0.7,-2.2) {};
\node[vertex] (r5) at (0.7,-2.2) {};

\draw[bedge] (r)--(r1) (r)--(r2) (r1)--(r3) (r1)--(r4) (r2)--(r5);

\node[lab] at (0.4,0.35)   {$\vroot$ (root)};
\node[lab] at (-1.7,-0.85)  {\(r_1\)};
\node[lab] at (1.1,-0.85)   {\(r_2\)};
\node[lab] at (-2.2,-2.40)  {\(r_3\)};
\node[lab] at (-0.3,-2.40)  {\(r_4\)};
\node[lab] at (1.1,-2.40)   {\(r_5\)};

\node[lab,anchor=west] at (1.8,-0.15) {\(\centerdot\) \(r_1\) is a child of $\vroot$};
\node[lab,anchor=west] at (1.8,-0.75) {\(\centerdot\) \(r_2\) is the parent of \(r_5\)};
\node[lab,anchor=west] at (1.8,-1.35) {\(\centerdot\) \(r_3\) is a leaf};
\node[lab,anchor=west] at (1.8,-1.95) {\(\centerdot\) \(|r_1|=1,\;|r_4|=2\)};

\node[title] at (1.5,-3.60) {(a) rooted tree};

% =========================================================
% (b) star S_3
% =========================================================
\node[vertex] (sr) at (9,0) {};
\node[vertex] (sa) at (8.15,-1.25) {};
\node[vertex] (sb) at (9,-1.45) {};
\node[vertex] (sc) at (9.85,-1.25) {};

\draw[bedge] (sr)--(sa) (sr)--(sb) (sr)--(sc);

\node[lab] at (9,0.30)  {$\vroot$};
\node[title] at (9,-3.60) {(b) \(S_3\)};

% =========================================================
% (c) one-point union
% =========================================================
% ---- T1 ----
\node[vertex] (t1r) at (12.0,0) {};
\node[vertex] (t1a) at (11.3,-1.1) {};
\draw[bedge] (t1r)--(t1a);
\node[lab] at (12.2,0.28) {$\vroot$};
\node[title] at (11.65,-2.85) {\(T_a\)};

% ---- T2 ----
\node[vertex] (t2r) at (14.2,0) {};
\node[vertex] (t2c) at (14.6,-1.1) {};
\node[vertex] (t2l) at (13.9,-2.1) {};
\node[vertex] (t2rr) at (15.3,-2.1) {};
\draw[bedge] (t2r)--(t2c) (t2c)--(t2l) (t2c)--(t2rr);
\node[lab] at (14.4,0.28) {$\vroot$};
\node[title] at (14.6,-2.85) {\(T_b\)};

% ---- T1 vee T2 ----
\node[vertex] (ur) at (17.3,0) {};
\node[vertex] (ua) at (16.5,-1.1) {};
\node[vertex] (uc) at (18.0,-1.1) {};
\node[vertex] (ub1) at (17.4,-2.1) {};
\node[vertex] (ub2) at (18.6,-2.1) {};
\draw[bedge] (ur)--(ua) (ur)--(uc) (uc)--(ub1) (uc)--(ub2);
\node[lab] at (17.5,0.28) {$\vroot$};
\node[title] at (17.5,-2.85) {\(T_a \vee T_b\)};

% (c) panel label
\node[title] at (15.2,-3.60) {(c) one-point union $T_a\vee T_b$ };

% =========================================================
% (d) extension E(T1 ∨ T2)
% =========================================================
\node[vertex] (er)  at (21.8,0) {};
\node[vertex] (e0)  at (21.8,-0.9) {};
\node[vertex] (ea)  at (21.0,-1.9) {};
\node[vertex] (ec)  at (22.5,-1.9) {};
\node[vertex] (eb1) at (21.9,-2.8) {};
\node[vertex] (eb2) at (23.1,-2.8) {};

\draw[bedge] (er)--(e0) (e0)--(ea) (e0)--(ec) (ec)--(eb1) (ec)--(eb2);

\node[lab] at (22.5,0.35) {$\vroot$ (new root)};
\node[lab,anchor=west] at (21.95,-0.40) {new edge};
\node[title] at (22.2,-3.60) {(d) extension $T_c=\extension{T_a\vee T_b}$};

\end{tikzpicture}%
}
\caption{Illustrations of rooted trees, the star \(S_3\), the one-point union, and the extension operation.}
\label{F.Illustrations of rooted trees}
\end{figure}
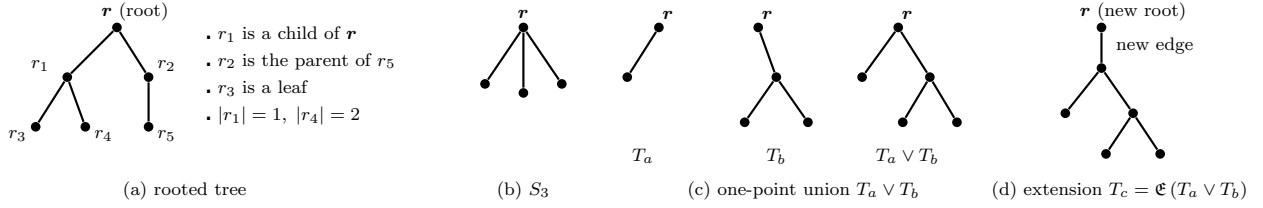
\paragraph{\textbf{Tree Family}} Let $\one$ denote the star $S_1$, a tree with only two vertices and one edge (it no longer looks like a star). The tree set corresponding to $\kterm{1}$, denoted by $\treefamily_1$, contains only one element.
\begin{equation}\label{E.def of treefamily 1}
    \treefamily_1:=\left\{\one\right\}.
\end{equation}
The tree family corresponding to the next renormalization term $\kterm{2}$ is defined by
\begin{equation}\label{E.def of treefamily 2}
    \treefamily_2:=\left\{\extension{\tree}:2\le k\le \alpha-1,\tree\in\treefamily_1^k\right\}.
\end{equation}
For $l\ge 3$, we construct $\treefamily_l$ based on $\treefamily_1,\ldots,\treefamily_{l-1}$ as follows:
\begin{equation}\label{E.def of treefamily l}
\treefamily_l:=\bigcup_{k=2}^{\alpha-1}\left\{\extension{\tree_1}:\tree_1\in\treefamily_{l-1}^k\right\} \cup \bigcup_{k=2}^{\alpha-1}\bigcup_{j=1}^{k-1}\left\{\extension{\tree_1\oneptunion\tree_2}:\tree_1\in\treefamily_{l-1}^j,\tree_2\in\left(\bigcup_{m=1}^{l-2}\treefamily_{m}\right)^{k-j}\right\}.
\end{equation}
If one sets $l=2$ in \eqref{E.def of treefamily l}, the set $\treefamily_{1}\cup\ldots\cup\treefamily_{l-2}$ is empty. So $\treefamily_2$ contains $\extension{\tree_1}:\tree_1\in\treefamily_1^k$, which coincides with the definition in \eqref{E.def of treefamily 2}. Referring to Figure \ref{F.Illustrations of rooted trees}, we see that $T_a \in \treefamily_1, T_b \in \treefamily_2, T_c\in \treefamily_3$.
\paragraph{Action of trees}
Fix a tree $\tree$ with root $\vroot$. We call a map $\embed$ an \emph{embedding} if $$\embed=\left(\embed^{(1)},\embed^{(2)}\right):V(\tree)\to \bbZ_+\times\bbZ, $$

and for every $e=[u,v]\in E(\tree)$, $\embed^{(1)}(u)>\embed^{(1)}(v)$. Given an edge $e=[u,v]$, define its weight associated with $\embed$ by
$$\weight{\embed}(e):=\cha\left(\embed^{(1)}(u)-\embed^{(1)}(v),\embed^{(2)}(u)-\embed^{(2)}(v)\right)$$.
Let $\embedset{t,x}(\tree)$ denote the set of all embeddings $\embed$ of $\tree$ such that $\embed(\vroot)=(t,x)$. 
\begin{definition}
    The \emph{action} of $\tree$ at $(t,x)$ is defined as 
\begin{equation}\label{E.def of action of T}\action{t,x}{\tree}:=\sum_{\embed\in\embedset{t,x}(\tree)}\prod_{e\in E(T)}\weight{\embed}(e)\left[\prod_{v\in L(\tree)}N^{-\frac{1}{4}-\rate}\truncatednoise(\embed(v))\right].
\end{equation}
\end{definition}

\begin{example}
    We illustrate the definition using the tree \(T_b\) in Figure \ref{F.Illustrations of rooted trees}. Let $\vroot$ be the root, \(u\) be the unique child of $\vroot$, and 
    \(v_1,v_2\) be the two leaves emanating from \(u\). Write
$$
\embed(\vroot)=(t,x),\qquad \embed(u)=(t_1,x_1),\qquad
\embed(v_1)=(t_2,x_2),\qquad
\embed(v_2)=(t_3,x_3)
$$
    where $1\le t_1<t, 1\le t_2,t_3<t_1$. Then
$$
\mathcal L_{t,x}(T_b)
=
N^{-1/2-2\rate}\sum_{\substack{
1\le t_1<t \\ 1\le t_2,t_3<t_1\\
x_1,x_2,x_3\in\mathbb Z
}}
\cha(t-t_1,x-x_1)\,
\cha(t_1-t_2,x_1-x_2)\,
\cha(t_1-t_3,x_1-x_3)\,
\truncatednoise(t_2,x_2)\truncatednoise(t_3, x_3).
$$

\end{example}
Here are two facts that will be used frequently in the rest of this section.

\begin{itemize}
    \item $\action{t,x}{\cdot}$ is multiplicative, i.e., for any $\tree_1,\ldots,\tree_n$ we have
    \begin{equation}\label{E.multiplicative}
        \action{t,x}{\oneptunion_{i=1}^n \tree_i}=\prod_{i=1}^{n}\action{t,x}{\tree_i}.
    \end{equation}
    \item For any tree $\tree$, the action of $\extension{\tree}$ can be expressed as
    \begin{equation}\label{E.action of extension}
        \action{t,x}{\extension{\tree}}=\sum_{s=1}^{t-1}\sum_{z\in\bbZ}\cha(t-s,x-z)\action{s,z}{\tree},\qquad\forall(t,x)\in\bbZ_+\times\bbZ.
    \end{equation}
\end{itemize}
The connection between the preceding definitions and \eqref{E.redefine Kl} is established by the following proposition.

\begin{proposition}\label{P.kterm represented by tree}
    Given $l\ge 1$, there exist constants $c(\tree,l)$ depending on $\tree\in\treefamily_l$ such that 
    \begin{equation}\label{E.kterm represented by tree}\kterm{l}(t,x)=\sum_{\tree\in\treefamily_l}c(\tree,l)\action{t,x}{\tree}\qquad\forall(t,x)\in\bbZ_+\times\bbZ.
    \end{equation}
\end{proposition}
\begin{proof}
    We prove \eqref{E.kterm represented by tree} by induction on $l$. The case $l=1$ follows directly:
    \begin{align*}
        \kterm{1}(t,x)=\sum_{s=1}^{t-1}\sum_{z\in\bbZ}\cha(t-s,x-z)N^{-\frac{1}{4}-\rate}\truncatednoise(s,z)=\action{t,x}{\one}.
    \end{align*}.
    
    Suppose now that $l\ge 2$ and \eqref{E.kterm represented by tree} holds for $1,\ldots,l-1$. $\kterm{l}$ is defined by
    \begin{align*}
        \kterm{l}(t,x)&=\sum_{s=1}^{t-1}\sum_{z\in\bbZ}\cha(t-s,x-z)\sum_{k=2}^{\alpha-1}a_k\sum_{j=1}^{k}\binom{k}{j}\left(\kterm{l-1}(s,z)\right)^j\left(\sum_{m=1}^{l-2}\kterm{m}(s,z)\right)^{k-j}
    \end{align*}
    By the induction hypothesis, $\kterm{l-1}(s,z)$ is a linear combination of $\action{s,z}{\tree}$ for $\tree\in\treefamily_{l-1}$. Consequently, $(\kterm{l-1}(s,z))^j$ is a linear combination of $\action{s,z}{\tree_1}$ for $\tree_1\in\treefamily_{l-1}^j$ since $\action{s,z}{\cdot}$ is multiplicative (see \eqref{E.multiplicative}). Similarly, $\left(\sum_{m=1}^{l-2}\kterm{m}(s,z)\right)^{k-j}$ is a linear combination of $\action{s,z}{\tree_2}$ for $\tree_2\in\left(\treefamily_{1}\cup\ldots\cup\treefamily_{l-2}\right)^{k-j}$. Thus, there are constants $c_1(\tree_1)$ and $c_2(\tree_2)$ such that 
    \begin{align*}
    \kterm{l}(t,x)=&\sum_{s=1}^{t-1}\sum_{z\in\bbZ}\cha(t-s,x-z)\sum_{k=2}^{\alpha-1}\sum_{j=1}^{k}a_k\binom{k}{j}\cdot\sum_{\tree_1}c_1(\tree_1)\action{s,z}{\tree_1}\sum_{\tree_2}c_2(\tree_2)\action{s,z}{\tree_2}\\
    =&\sum_{k=2}^{\alpha-1}\sum_{j=1}^{k}a_k\binom{k}{j}\sum_{\tree_1,\tree_2}c_1(\tree_1)c_2(\tree_2)\sum_{s=1}^{t-1}\sum_{z\in\bbZ}\cha(t-s,x-z)\action{s,z}{\tree_1\oneptunion\tree_2}\\
    =&\sum_{k=2}^{\alpha-1}\sum_{j=1}^{k}\sum_{\tree_1,\tree_2}a_k\binom{k}{j}c_1(\tree_1)c_2(\tree_2)\action{t,x}{\extension{\tree_1\oneptunion\tree_2}}
    \end{align*}
    Here we suppress the constraints $\tree_1\in\treefamily_{l-1}^j$ and $\tree_2\in\left(\treefamily_{1}\cup\ldots\cup\treefamily_{l-2}\right)^{k-j}$ on the indices of $\sum$ for simplicity. Since $\treefamily_l$ is defined precisely as the set of all such $\extension{\tree_1\oneptunion\tree_2}$,  the proof is complete.

\end{proof}
As a corollary of Proposition \ref{P.kterm represented by tree} and the multiplicativity property \eqref{E.multiplicative}, we have
\begin{corollary}\label{C.kterm represented by tree}
     Fix $l,n\ge 1$. There exist constants $c(\tree,l,n)$ for each $\tree\in\treefamily_l^n$ so that for any $(t,x)\in\bbZ_+\times\bbZ$,
    \begin{equation}\left(\kterm{l}(t,x)\right)^n=\sum_{\tree\in\treefamily_l^n}c(\tree,l,n)\action{t,x}{\tree}.
    \end{equation}
\end{corollary}
As Corollary \ref{C.kterm represented by tree} shows, in order to prove Theorem \ref{T.moments upper bound on K 1}, it is sufficient to bound $\bbE[\action{t,x}{\tree}]$ for each $\tree\in\treefamily_l^n$. 

\subsection{Leaf partitions and gluing}
First, we take the expectation of \eqref{E.def of action of T}. Each $\embed\in\embedset{t,x}(\tree)$ induces a partition of $L(\tree)$ in which any two leaves $v,v'$ are in the same part if and only if $\embed(v)=\embed(v')$. Denote by $\Pi(L(\tree))$ the set of all partitions of $L(\tree)$ and $\embedset{t,x}^{\pi}(\tree)$ the set of all $\embed\in\embedset{t,x}(\tree)$ which induce exactly the partition $\pi$ of $L(\tree)$. For $\embed\in\embedset{t,x}(\tree)$ and a part $\pi_i \in \pi$, since $\embed$ takes a constant value on $\pi_i$, we define $
\embed^{(1)}(\pi_i) := \embed^{(1)}(v),$ for any $v \in \pi_i$. Let $\tilde\mu_{k}(t)=\bbE[\truncatednoise(t,0)^k]$. For $\embed\in\embedset{t,x}^\pi(\tree)$, we write
$$\moment{\embed}{\pi}:=\prod_{\pi_i\in\pi}\tilde\mu_{\setsize{\pi_i}}(\embed^{(1)}(\pi_i)).$$

By taking expectations in \eqref{E.def of action of T}, we obtain

$$\bbE[\action{t,x}{\tree}]=N^{-(\frac{1}{4}+\rate)\setsize{L(\tree)}}\sum_{\pi\in\Pi(L(\tree))}\sum_{\embed\in\embedset{t,x}^{\pi}(\tree)}\prod_{e\in E(\tree)}\weight{\embed}(e)\moment{\embed}{\pi}.$$

A partition $\pi\in\Pi(L(\tree))$ is said to be \emph{proper} if any part $\pi_i$ of $\pi$ satisfies the following:
\begin{itemize}
    \item[(1)] $\setsize{\pi_i}\ge 2$.
    \item[(2)] If $\setsize{\pi_i}=2$, suppose $\pi_i=\{v_1,v_2\}$. Then, there is \textbf{no} non-root vertex $u$ such that $v_1$ and $v_2$ are the only two children of $u$. 
    \item[(3)] If $\setsize{\pi_i}=3$, suppose $\pi_i=\{v_1,v_2,v_3\}$. Then, there is \textbf{no} non-root vertex $u$ such that $v_1,v_2$ and $v_3$ are the only three children of $u$. Also, there is \textbf{no} pair of vertices $u,w$ such that the only two children of $u$ are $w$ and one of $v_1,v_2,v_3$, and the only two children of $w$ are the other two of $v_1,v_2,v_3$.
\end{itemize}
The partitions that violate (1) (resp., (2) and (3)) are shown in Figure \ref{F.proper-counterexamples} (a) (resp., (b) and (c)).
If a non-root vertex $u$ satisfies either of the latter two conditions above, we call $u$ a \emph{bad point} of $\pi$. 

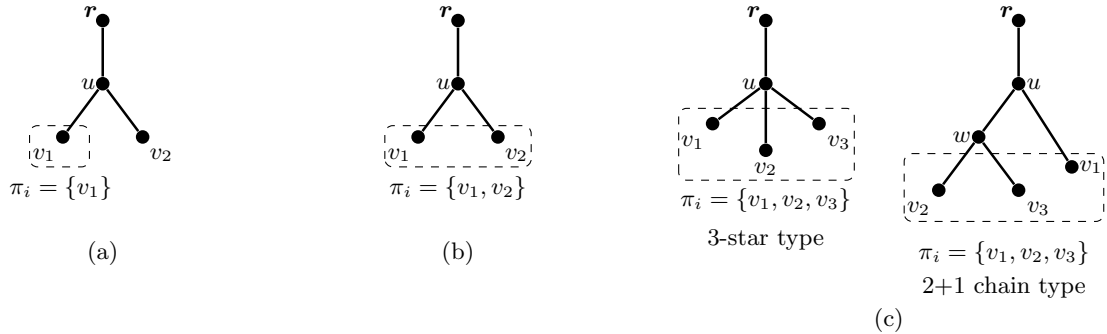
\begin{figure}[htbp]
\centering
\begin{tikzpicture}[
    x=0.88cm,y=0.88cm,
    vertex/.style={circle, fill=black, inner sep=1.8pt},
     bedge/.style={draw=black,line width=0.95pt},
  dbedge/.style={draw=black,line width=0.75pt,double,double distance=1.1pt},
    lab/.style={font=\small},
    title/.style={font=\small},
    part/.style={draw, dashed, rounded corners=3pt}
]

% =========================================================
% (a) counterexample to condition (i): |\pi_i| >= 2
% =========================================================
\begin{scope}[xshift=0cm,yshift=0cm]
    % vertices
    \node[vertex] (r)  at (0,0) {};
    \node[vertex] (u)  at (0,-0.95) {};
    \node[vertex] (v1) at (-0.60,-1.75) {};
    \node[vertex] (v2) at (0.60,-1.75) {};

    % edges
    \draw[bedge] (r) -- (u);
    \draw[bedge] (u) -- (v1);
    \draw[bedge] (u) -- (v2);

    % labels
    \node[lab] at (-0.18,0.14) {$\vroot$};
    \node[lab] at (-0.22,-0.95) {$u$};
    \node[lab] at (-0.88,-2.02) {$v_1$};
    \node[lab] at (0.88,-2.02) {$v_2$};

    % partition block: include vertex v1 and its label
    \draw[part] (-1.1,-2.2) rectangle (-0.20,-1.58);
    \node[lab] at (-0.63,-2.52) {$\pi_i=\{v_1\}$};

    % title
    \node[title] at (0,-3.50) {(a)};
\end{scope}

% =========================================================
% (b) counterexample to condition (ii): |\pi_i| = 2
% =========================================================
\begin{scope}[xshift=4.7cm,yshift=0cm]
    % vertices
    \node[vertex] (r)  at (0,0) {};
    \node[vertex] (u)  at (0,-0.95) {};
    \node[vertex] (v1) at (-0.60,-1.75) {};
    \node[vertex] (v2) at (0.60,-1.75) {};

    % edges
    \draw[bedge] (r) -- (u);
    \draw[bedge] (u) -- (v1);
    \draw[bedge] (u) -- (v2);

    % labels
    \node[lab] at (-0.18,0.14) {$\vroot$};
    \node[lab] at (-0.22,-0.95) {$u$};
    \node[lab] at (-0.88,-2.02) {$v_1$};
    \node[lab] at (0.88,-2.02) {$v_2$};

    % partition block: include vertices v1,v2 and labels
    \draw[part] (-1.10,-2.2) rectangle (1.10,-1.58);
    \node[lab] at (0,-2.52) {$\pi_i=\{v_1,v_2\}$};

    % title
    \node[title] at (0,-3.50) {(b) };
\end{scope}

% =========================================================
% (c) counterexample to condition (iii): |\pi_i| = 3
%     left = 3-star type, right = 2+1 chain type
% =========================================================
\begin{scope}[xshift=10.4cm,yshift=0cm]

    % ---------- left: direct 3-children bad point ----------
    \node[vertex] (r1)  at (-1.85,0) {};
    \node[vertex] (u1)  at (-1.85,-0.95) {};
    \node[vertex] (a1)  at (-2.65,-1.55) {};
    \node[vertex] (a2)  at (-1.85,-1.95) {};
    \node[vertex] (a3)  at (-1.05,-1.55) {};

    \draw[bedge] (r1) -- (u1);
    \draw[bedge] (u1) -- (a1);
    \draw[bedge] (u1) -- (a2);
    \draw[bedge] (u1) -- (a3);

    % labels
    \node[lab] at (-2.03,0.14) {$\vroot$};
    \node[lab] at (-2.10,-0.95) {$u$};
    \node[lab] at (-2.95,-1.83) {$v_1$};
    \node[lab] at (-1.85,-2.23) {$v_2$};
    \node[lab] at (-0.75,-1.83) {$v_3$};

    % partition block: include vertices v1,v2,v3 and labels
    \draw[part] (-3.16,-2.4) rectangle (-0.52,-1.32);
    \node[lab] at (-1.85,-2.7) {$\pi_i=\{v_1,v_2,v_3\}$};
    \node[lab] at (-1.85,-3.3) {3-star type};

    % ---------- right: nested 2+1 bad configuration ----------
    \node[vertex] (r2)  at (1.95,0) {};
    \node[vertex] (u2)  at (1.95,-0.95) {};
    \node[vertex] (w)   at (1.35,-1.75) {};
    \node[vertex] (b1)  at (2.75,-2.2) {};
    \node[vertex] (b2)  at (0.75,-2.55) {};
    \node[vertex] (b3)  at (1.95,-2.55) {};

    \draw[bedge] (r2) -- (u2);
    \draw[bedge] (u2) -- (w);
    \draw[bedge] (u2) -- (b1);
    \draw[bedge] (w) -- (b2);
    \draw[bedge] (w) -- (b3);

    % labels
    \node[lab] at (1.77,0.14) {$\vroot$};
    \node[lab] at (2.18,-0.95) {$u$};
    \node[lab] at (1.10,-1.75) {$w$};
    \node[lab] at (3.04,-2.2) {$v_1$};
    \node[lab] at (0.45,-2.82) {$v_2$};
    \node[lab] at (2.25,-2.82) {$v_3$};

    % partition block: include vertices v1,v2,v3 and labels
    \draw[part] (0.23,-3.02) rectangle (3.23,-2);
    \node[lab] at (1.73,-3.5) {$\pi_i=\{v_1,v_2,v_3\}$};
    \node[lab] at (1.73,-4) {$2{+}1$ chain type};

    % title for panel (c)
    \node[title] at (0,-4.5) {(c)};
\end{scope}
\end{tikzpicture}
\caption{Counterexamples in the definition of a proper partition.}
\label{F.proper-counterexamples}
\end{figure}

The following proposition and corollary allow us to restrict our discussion to \emph{proper} partitions.
\begin{proposition}\label{P.restrict to proper partitions}
    For a fixed partition $\pi\in\Pi(L(\tree))$, for any $x$ and $t\le aN$, 
    \begin{equation}\label{E.restrict to proper partitions}
        \absolute{\sum_{\embed\in\embedset{t,x}^{\pi}(\tree)}\prod_{e\in E(\tree)}\weight{\embed}(e)\moment{\embed}{\pi}}\loglesssim \sum_{\substack{\tilde{\pi}\in\Pi(L(\tree))\\ \tilde{\pi}~\text{is proper}}}N^{\sum_{\tilde{\pi}_i\in\tilde{\pi}}\frac{1}{4}(\setsize{\tilde{\pi}_i}-8)^+}\cdot\sum_{\embed\in\embedset{t,x}^{\tilde{\pi}}(\tree)}\prod_{e\in E(\tree)}\absolute{\weight{\embed}(e)}.
    \end{equation}
\end{proposition}
\begin{proof}
    Recall that the moments of $\truncatednoise(t,0)$ satisfy the uniform bound
    $$\absolute{\tilde{\mu}_k(t)}\le \bbE[\truncatednoise(t,0)^8]\cdot(N^{\frac{1}{4}}\log N)^{(k-8)^+}\loglesssim N^{\frac{1}{4}(k-8)^+}~\forall~k\ge 1,$$ and thus for any $\embed\in\embedset{t,x}^\pi(\tree)$, 
    $$\absolute{\moment{\embed}{\pi}}\loglesssim N^{\sum_{\pi_i\in\pi}\frac{1}{4}(\setsize{\pi_i}-8)^+}$$
    holds uniformly. Thus, if $\pi$ is a proper partition, then the left-hand side of \eqref{E.restrict to proper partitions} is bounded by the corresponding term on the right-hand side.
    
    It remains to consider an improper partition $\pi$. If $\pi$ contains a set of size $1$, then the identities $\bbE[\truncatednoise(t,0)]=0,~\forall~t$, imply that $\moment{\embed}{\pi}=0$ for any $\embed\in\embedset{t,x}^\pi(\tree)$.
    Now suppose that every part of $\pi$ contains at least two elements. If $\pi$ has a bad point $u$, then there exists a part of $\pi$ (without loss of generality $\pi_1$) such that $\pi_1$ consists precisely of the leaf descendants of $u$. As we observed before, $u\neq\vroot$. Denote by $v$ the parent of $u$ and $\tree_u$ the subgraph of $T$ induced by $u$ and all of its descendants; this subgraph forms a subtree of $T$ rooted at $u$. Similarly, denote by $\tree \setminus \tree_u$ the subgraph of $T$ induced by all vertices other than $u$ and its descendants. By considering the restrictions $\embed_1 := \embed|_{\tree \setminus \tree_u}$ and $\embed_2 := \embed|_{\tree_u}$, we can then factorize

    \begin{align*}
        &\sum_{\embed\in\embedset{t,x}^{\pi}(\tree)}\prod_{e\in E(\tree)}\weight{\embed}(e)\moment{\embed}{\pi}      =\sum_{\embed_1\in\embedset{t,x}^{\pi\setminus\{\pi_1\}}(\tree\setminus\tree_u)}\prod_{e\in E(\tree\setminus\tree_u)}\weight{\embed_1}(e)\moment{\embed_1}{\pi\setminus\{\pi_1\}}\cdot\\
        &\sum_{s=1}^{\embed_1^{(1)}(v)-1}\sum_{z\in\bbZ}\cha\left(\embed_1^{(1)}(v)-s,\embed_1^{(2)}(v)-z\right)\sum_{\substack{\embed_2\in\embedset{s,z}^{\{\pi_1\}}(\tree_u)\\ \embed_2(L(\tree_u))\cap \embed_1(L(\tree\setminus\tree_u))=\emptyset}}\prod_{e\in E(\tree_u)}\weight{\embed_2}(e)\tilde{\mu}_{\setsize{\pi_1}}(\embed_2^{(1)}(\pi_1)),
    \end{align*}
    where $\embed_2(L(\tree_u))\cap \embed_1(L(\tree\setminus\tree_u))=\emptyset$ is necessary and sufficient for $\embed\in\embedset{t,x}^{\pi}(\tree)$. For fixed $s$ and $\embed_2^{(1)}$, 
    $$\sum_{\substack{\embed_2\in\embedset{s,z}^{\{\pi_1\}}(\tree_u)}}\prod_{e\in E(\tree_u)}\weight{\embed_2}(e)\tilde{\mu}_{\setsize{\pi_1}}(\embed_2^{(1)}(\pi_1))$$
    does not depend on $z$ and 
    $$\sum_{z\in\bbZ}\cha\left(\embed_1^{(1)}(v)-s,\embed_1^{(2)}(v)-z\right)=0.$$
    We write $\pi'\prec_1\pi$ if $\pi'$ is obtained from $\pi$ by gluing $\pi_1$ with any other part of $\pi$. Then
    \begin{align*}
        \sum_{\embed\in\embedset{t,x}^{\pi}(\tree)}\prod_{e\in E(\tree)}\weight{\embed}(e)\moment{\embed}{\pi}=
        -\sum_{\pi'\prec_1\pi} \sum_{\embed\in\embedset{t,x}^{\pi'}(\tree)}\prod_{e\in E(\tree)}\weight{\embed}(e)\moment{\embed_1}{\pi\setminus\{\pi_1\}}\tilde{\mu}_{\setsize{\pi_1}}(\embed_2^{(1)}(\pi_1)),
    \end{align*}
    where $\embed_1$ and $\embed_2$ are the restrictions of $\embed$ on $\tree\setminus\tree_u$ and $\tree_u$, respectively. For any $\pi'\prec_1\pi$, $u$ is no longer a bad point of $\pi'$. Note that any bad point of $\pi'$ is also a bad point of $\pi$, so the number of bad points of $\pi'$ is at least one less than that of $\pi$.
    If $\pi'$ remains improper, we apply the same procedure again. 
    
    We write $\pi\preceq\pi'$ if $\pi'$ is coarser than $\pi$, i.e., for any part $\pi_i\in\pi$ there is a $\pi_j'\in\pi'$ such that $\pi_i\subset\pi_j'$. After finitely many applications of the preceding identity, the process terminates, and we can express 

    $$\sum_{\embed\in\embedset{t,x}^{\pi}(\tree)}\prod_{e\in E(\tree)}\weight{\embed}(e)\moment{\embed}{\pi}$$
    as a linear combination of 
    $$\sum_{\embed\in\embedset{t,x}^{\tilde\pi}(\tree)}\prod_{e\in E(\tree)}\weight{\embed}(e)A_{\tilde\pi},~\text{$\tilde\pi$ is proper},$$
    where $A_{\tilde\pi}$ is a linear combination of terms like $\prod_{\pi_i\in\pi}\tilde{\mu}_{\setsize{\pi_i}}(t_i)$, and are all $\loglesssim N^{\sum_{\pi_i\in\pi}(\setsize{\pi_i}-8)^+/4}$ uniformly. Note that
    $$\frac{1}{4}\sum_{\pi_i\in\pi}(\setsize{\pi_i}-8)^+\le \frac{1}{4}\sum_{\tilde\pi_i\in\tilde\pi}(\setsize{\tilde\pi_i}-8)^+$$
    for every $\pi\preceq\tilde\pi$. This completes the proof.

\end{proof}
Summing \eqref{E.restrict to proper partitions} over $\pi\in\Pi(L(\tree))$, we obtain the following corollary.
\begin{corollary}\label{C.restrict to proper partitions}
    For a fixed tree $\tree$, for any $x$ and $t\le aN$,
    \begin{equation}\label{E.restrict to proper partitions corollary}
        \absolute{\bbE[\action{t,x}{\tree}]}\loglesssim N^{-(\frac{1}{4}+\rate)\setsize{L(\tree)}}\sum_{\substack{\pi\in\Pi(L(\tree))\\ \pi~\text{is proper}}}N^{\frac{1}{4}\sum_{\pi_i\in\pi}(\setsize{\pi_i}-8)^+}\cdot\sum_{\embed\in\embedset{t,x}^{\pi}(\tree)}\prod_{e\in E(\tree)}\absolute{\weight{\embed}(e)}.
    \end{equation}
\end{corollary}

\paragraph{Leaf-glued graph.}
Given a tree $T$ and a partition $\pi$ of its leaves, let $T^{\pi}$ be the graph obtained from $T$ by gluing together (equivalently, identifying) all leaves that belong to the same part of $\pi$. We call $T^{\pi}$ the \emph{leaf-glued graph} generated by $T$ and $\pi$. Any edges that become parallel after this identification are retained as multiple edges.

To visualize the contribution of a fixed partition $\pi$, we glue together all leaves belonging to the same part of $\pi$. In this way, the constraint $\embed(v)=\embed(v')$ for $v$ and $v'$ in the same part is represented by a single vertex in the resulting graph. Figure~\ref{fig:fixed-pi-gluing} illustrates this for the tree $T=T_a\vee T_b$ shown in Figure~\ref{F.Illustrations of rooted trees}, in the case where $v_1$, $v_2$, and $v_3$ belong to the same part $\pi_1=\{v_1,v_2,v_3\}$.

\begin{figure}[htbp]
\centering
\begin{tikzpicture}[
    x=0.72cm,y=0.72cm,
    vertex/.style={circle, fill=black, inner sep=1.8pt},
     bedge/.style={draw=black,line width=0.95pt},
  dbedge/.style={draw=black,line width=0.75pt,double,double distance=1.1pt},
    lab/.style={font=\small},
    >=Latex
]

% =========================================================
% Left: original tree with a chosen block pi_1
% =========================================================
\begin{scope}[xshift=0cm,yshift=0cm]
    % vertices
    \node[vertex] (r)  at (0,1.55) {};
    \node[vertex] (v1) at (-0.95,0.55) {};
    \node[vertex] (u)  at (0.62,0.98) {};
    \node[vertex] (v2) at (0.25,0.02) {};
    \node[vertex] (v3) at (1.00,0.28) {};

    % edges
    \draw[bedge] (r) -- (v1);
    \draw[bedge] (r) -- (u);
    \draw[bedge] (u) -- (v2);
    \draw[bedge] (u) -- (v3);

    % labels
    \node[lab] at (-0.02,1.86) {$\vroot$};
    \node[lab] at (0.88,1.18) {$u$};
    \node[lab] at (-1.08,0.25) {$v_1$};
    \node[lab] at (0.08,-0.28) {$v_2$};
    \node[lab] at (1.15,-0.06) {$v_3$};

    % dashed box around the block
    \draw[dashed, rounded corners=2pt]
        (-1.35,-0.5) rectangle (1.4,0.78);

    % partition label
    \node[lab] at (0.00,-0.82) {$\pi_1=\{v_1,v_2,v_3\}$};
\end{scope}

% arrow
\draw[->, line width=0.5pt] (3.0,0.88) -- (4.45,0.88);

% =========================================================
% Right: contracted graph after gluing v1,v2,v3
% =========================================================
\begin{scope}[xshift=4.95cm,yshift=-0.22cm]
    % vertices
    \node[vertex] (a)  at (0,0.58) {};
    \node[vertex] (rr) at (0.82,1.62) {};
    \node[vertex] (uu) at (1.72,0.92) {};

    % edges
    \draw[bedge] (rr) -- (a);
    \draw[bedge] (rr) -- (uu);
    \draw[dbedge] (a) -- (uu);

    % labels
    \node[lab] at (0.83,1.92) {$\vroot$};
    \node[lab] at (1.96,1.06) {$u$};
    \node[lab] at (0.5,0) {$\pi_1=\{v_1,v_2,v_3\}$};
\end{scope}

\end{tikzpicture}
\caption{Illustration of the summation for $T_a \vee T_b$ when $v_1$, $v_2$, and $v_3$ belong to the same part $\pi_1=\{v_1,v_2,v_3\}$. Leaves in the same part are merged into a single vertex, and any resulting parallel edges are retained as multiple edges.}
\label{fig:fixed-pi-gluing}
\end{figure}
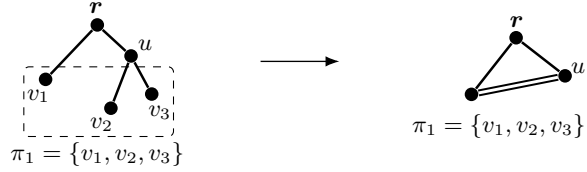

\begin{example}\label{EX.tree representation}
     Assuming that the random variables $\noise(t,x)$ are i.i.d., we illustrate the computation of the expectation of the action associated with the tree \(T_d= T_b\vee T_b\); see Figure \ref{F.Td-example46}. Among the partitions of $L(T_d)=\{v_1,v_2,v_3,v_4\}$, the only proper ones are

$$
    \pi_4:=\{\{v_1,v_2,v_3,v_4\}\},\qquad \pi_{13|24}:=\{\{v_1,v_3\},\{v_2,v_4\}\},\qquad  \pi_{14|23}:=\{\{v_1,v_4\},\{v_2,v_3\}\}.
$$
The proof of Proposition \ref{P.restrict to proper partitions} gives
$$
    \mathbb{E} [\action{t,x}{T_d}]=N^{-1-4\rate}\Bigl((\mathbb{E}[\truncatednoise(1,0)^4]-(\mathbb{E}[\truncatednoise(1,0)]^2)^2) \mathcal I_4(t,x)+2(\mathbb{E}[\truncatednoise(1,0)^2] )^2\,\mathcal I_{2,2}(t,x)\Bigr),
$$
where

\begin{align*}
    \mathcal I_4(t,x):=\sum_{t_1,t_2=1}^{t-1}\sum_{s=1}^{t_1\wedge t_2-1}\sum_{z_1,z_2,y\in\mathbb Z}&\cha(t-t_1,x-z_1)\cha(t-t_2,x-z_2)\\
&\times\cha(t_1-s,z_1-y)^2\cha(t_2-s,z_2-y)^2,   
\end{align*}
and
\begin{align*}
    \mathcal I_{2,2}(t,x):=&\sum_{t_1,t_2=1}^{t-1}   \sum_{s_1,s_2=1}^{t_1\wedge t_2-1}\sum_{z_1,z_2\in\mathbb Z}\sum_{\substack{y_1,y_2\in\mathbb Z\\(s_1,y_1)\neq(s_2,y_2)}}\cha(t-t_1,x-z_1)\cha(t-t_2,x-z_2)\\
    &\times\cha(t_1-s_1,z_1-y_1)\cha(t_2-s_1,z_2-y_1)\cha(t_1-s_2,z_1-y_2)\cha(t_2-s_2,z_2-y_2).
\end{align*}
Here $\mathcal I_4$ corresponds to $\pi_4$, while $\mathcal I_{2,2}$  corresponds to $\pi_{13|24}$ and $\pi_{14|23}$. The latter two partitions yield isomorphic glued graphs. $\pi_4$ and  $\pi_{13|24}$ (or $\pi_{14|23}$) correspond to the two graphs displayed together with $T_d$ in Figure~\ref{F.Td-example46}. 

The graph-lifting method introduced in Section \ref{S.Graph Lifting} yields
$$\mathbb{E} [\action{t,x}{T_d}]\loglesssim N^{-1-4\rate}.$$
\end{example}

\begin{figure}[htbp]
\centering
\begin{tikzpicture}[
    x=0.72cm,y=0.72cm,
    vertex/.style={circle, fill=black, inner sep=1.8pt},
     bedge/.style={draw=black,line width=0.95pt},
  dbedge/.style={draw=black,line width=0.75pt,double,double distance=1.1pt},
    lab/.style={font=\small},
    slab/.style={font=\scriptsize}
]

% =========================================================
% (a) original tree T_d
% =========================================================
\begin{scope}[xshift=-3.95cm,yshift=0cm]

\node[vertex] (r)  at (0,0.88) {};
\node[vertex] (u1) at (-0.78,0.20) {};
\node[vertex] (u2) at ( 0.78,0.20) {};
\node[vertex] (v1) at (-1.36,-0.44) {};
\node[vertex] (v2) at (-0.20,-0.44) {};
\node[vertex] (v3) at ( 0.20,-0.44) {};
\node[vertex] (v4) at ( 1.36,-0.44) {};

\draw[bedge] (r)  -- (u1);
\draw[bedge] (r)  -- (u2);
\draw[bedge] (u1) -- (v1);
\draw[bedge] (u1) -- (v2);
\draw[bedge] (u2) -- (v3);
\draw[bedge] (u2) -- (v4);

\node[lab]  at (0,1.56) {$T_d$};
\node[slab] at (0,-1.10) {$T_d=T_b\vee T_b$};
\end{scope}

% =========================================================
% (b) glued graph for pi_4
% =========================================================
\begin{scope}[xshift=0cm,yshift=0cm]

\node[vertex] (r)  at (0,0.88) {};
\node[vertex] (u1) at (-0.78,0.20) {};
\node[vertex] (u2) at ( 0.78,0.20) {};
\node[vertex] (w)  at ( 0.00,-0.44) {};

\draw[bedge] (r)  -- (u1);
\draw[bedge] (r)  -- (u2);

% parallel multiple edges from u1 to w
\draw (-0.82,0.16) -- (-0.04,-0.48);
\draw (-0.74,0.24) -- ( 0.04,-0.40);

% parallel multiple edges from u2 to w
\draw ( 0.74,0.24) -- (-0.04,-0.40);
\draw ( 0.82,0.16) -- ( 0.04,-0.48);

\node[lab]  at (0,1.56) {$\pi_4$};
\node[slab] at (0,-1.10) {$\{v_1,v_2,v_3,v_4\}$};
\end{scope}

% =========================================================
% (c) glued graph for pi_{13|24}
% =========================================================
\begin{scope}[xshift=3.95cm,yshift=0cm]

\node[vertex] (r)   at (0,0.88) {};
\node[vertex] (u1)  at (-0.78,0.20) {};
\node[vertex] (u2)  at ( 0.78,0.20) {};
\node[vertex] (w13) at (-0.58,-0.44) {};
\node[vertex] (w24) at ( 0.58,-0.44) {};

\draw[bedge] (r)  -- (u1);
\draw[bedge] (r)  -- (u2);

\draw[bedge] (u1) -- (w13);
\draw[bedge] (u1) -- (w24);
\draw[bedge] (u2) -- (w13);
\draw[bedge] (u2) -- (w24);

\node[lab]  at (0,1.56) {$\pi_{13|24}$};
\node[slab] at (-0.86,-1.10) {$\{v_1,v_3\}$};
\node[slab] at ( 0.86,-1.10) {$\{v_2,v_4\}$};
\end{scope}

\end{tikzpicture}

\par\vspace{3pt}
\refstepcounter{figure}\small Figure~\thefigure: The tree $T_d=T_b\vee T_b$ and the two graphs obtained from the proper partitions $\pi_4$ and $\pi_{13|24}$.
\label{F.Td-example46}
\end{figure}

\subsection{Reduction to full binary trees} 
After gluing, $\tree^\pi$ is in general a rooted multigraph rather than a tree. Each edge of $\tree^\pi$ is identified with its preimage in $\tree$, and therefore inherits the parent–child orientation from the original rooted tree.

Note that for each set of leaves that were glued together when constructing $\tree^\pi$, $\embed$ takes the same value on this set. Every embedding map $\embed\in\embedset{t,x}^\pi(\tree)$ naturally lifts to a map $\tilde\embed$ on $V(\tree^\pi)$. Let $\embedset{t,x}(\tree^\pi)$ denote the set of all $\tilde\embed$.
We may identify each edge of $\tree^\pi$ with the corresponding edge of $\tree$. We henceforth identify lifted maps with embeddings of $\tree^\pi$. Here, we slightly modify the weight function for technical reasons (as will be seen in the proof of the next proposition). Define
$$\cha'(t,x):=\frac{C_{\cha}}{1+t+x^2}, \quad C_{\cha}:=1+\sup_{s\ge 1, z\in\mathbb{Z}} (1+s+z^2)\absolute{\cha(s,z)}<\infty. $$
Given an edge $e=[u,v]$ and an embedding map $\tilde\embed=(\tilde\embed^{(1)},\tilde\embed^{(2)})$, the weight of $e$ associated with $\tilde\embed$ is defined as
$$\weight{\tilde\embed}(e):=\cha'\left(\tilde\embed^{(1)}(u)-\tilde\embed^{(1)}(v),\tilde\embed^{(2)}(u)-\tilde\embed^{(2)}(v)\right).$$
The resulting action of $\tree^\pi$ is defined by
\begin{equation}\label{E.def of action of T pi}
    \action{t,x}{\tree^\pi}:=\sum_{\tilde\embed\in\embedset{t,x}(\tree^\pi)}\prod_{e\in E(\tree^\pi)}\weight{\tilde\embed}(e),
\end{equation}
and hence

$$ \sum_{\embed\in\embedset{t,x}^\pi(\tree)}\prod_{e\in E(\tree)}\absolute{\weight{\embed}(e)}\le \action{t,x}{\tree^\pi}$$
since $|\cha|\le\cha’$ for all $(t,x)$.
We call a tree a \emph{full binary tree} if every vertex other than the root and the leaves has exactly two children. The following proposition allows us to reduce our discussion to the class of full binary trees. We defer the proof to Appendix~\ref{S.Proof of Proposition P.binary tree}.

\begin{proposition}\label{P.binary tree}
Let $\tree$ be a rooted tree such that every non-root, non-leaf vertex has at least two children. Fix a proper partition $\pi$ of $L(\tree)$. Then, there exists a full binary tree $\tilde\tree$ with the same number of leaves as
$\tree$, and a constant $C=C(\tree,\pi)$, such that $\pi$ is also a proper partition of $L(\tilde\tree)$ and, for all $x$ and $t\le aN$,
    $$\action{t,x}{\tree^\pi}\lesssim\action{t+C,x}{\tilde{\tree}^\pi}.$$
\end{proposition}

Figure~\ref{F.prop4.7-case} illustrates the operation in Proposition~\ref{P.binary tree} in the case where $\pi$ has only one part, namely, the set of all leaves of $T$.

\begin{figure}[htbp]
\centering
\begin{tikzpicture}[
    x=0.72cm,y=0.72cm,
    vertex/.style={circle, fill=black, inner sep=1.8pt},
     bedge/.style={draw=black,line width=0.95pt},
  dbedge/.style={draw=black,line width=0.75pt,double,double distance=1.1pt},
    lab/.style={font=\small},
    >=Latex
]

% =========================================================
% Figure 5: schematic case for Proposition 4.7
% =========================================================

% left tree T

\node[vertex] (a1) at (0.00,1.48) {};
\node[vertex] (a2) at (-1.00,0.74) {};
\node[vertex] (a3) at (-0.34,0.48) {};
\node[vertex] (a4) at ( 0.34,0.48) {};
\node[vertex] (a5) at ( 1.00,0.74) {};

\draw[bedge] (a1) -- (a2);
\draw[bedge] (a1) -- (a3);
\draw[bedge] (a1) -- (a4);
\draw[bedge] (a1) -- (a5);

\node[lab] at (0.00,2) {$T$};

% arrow
\draw[->] (1.5,1.10) -- (2.5,1.10);

% right tree \widetilde T
\node[vertex] (b0) at (4.10,1.48) {};
\node[vertex] (b1) at (3.30,0.82) {};
\node[vertex] (b2) at (4.90,0.82) {};
\node[vertex] (b3) at (2.68,0.08) {};
\node[vertex] (b4) at (3.92,0.08) {};
\node[vertex] (b5) at (4.28,0.08) {};
\node[vertex] (b6) at (5.52,0.08) {};

\draw[bedge] (b0) -- (b1);
\draw[bedge] (b0) -- (b2);
\draw[bedge] (b1) -- (b3);
\draw[bedge] (b1) -- (b4);
\draw[bedge] (b2) -- (b5);
\draw[bedge] (b2) -- (b6);

\node[lab] at (4.10,2) {$\widetilde{T}$};

\end{tikzpicture}

\par\vspace{2pt}
\refstepcounter{figure}\small Figure~\thefigure: An example of Proposition \ref{P.binary tree}.
\label{F.prop4.7-case}
\end{figure}

\section{Graph lifting}\label{S.Graph Lifting}
Building on Section \ref{S.Tree representation and reductions}, we introduce a technique called \emph{graph lifting} and prove Theorem \ref{T.moments upper bound on K 1}. Here we consider $\kterm{l}$ only for $l\ge 2$, which also means that the root $\vroot$ has no leaf children. We first introduce the necessary definitions.

\paragraph{Graphs} For a fixed tree $\tree$, we say a graph $G$ is a \emph{two-colored backbone graph} with backbone $\tree$ if it satisfies
\begin{itemize}
    \item All edges of $G$ are undirected, and each edge is colored black or red. Multiple edges are allowed.
    \item The black edges can be divided into two parts. The first part forms a copy of $\tree$ with simple edges, which we will identify with $\tree$ if no ambiguity arises. All black edges in the second part connect a vertex in $\tree$ to a vertex not in $\tree$.
\end{itemize}
We still call the root $\vroot$ of $\tree$ in $G$ the root of $G$, and denote it by $\vroot$, although $G$ is generally not a tree.
For an edge $e$ connecting vertices $u$ and $v$, we write $e=\{u,v\}$ if $e$ is red. If $e$ is black and belongs to the first part, we write $e=[u,v]$ if $u$ is closer to the root than $v$ is in the tree $\tree$. If $e$ is black and belongs to the second part, we write $e=[u,v]$ if $u$ is a vertex of $\tree$ but $v$ is not.
For any set of vertices $S$, denote by $R(S)$ the number of red edges between vertices in $S$ (a multiple edge is counted repeatedly), and $B(S)$ the number of black edges between vertices in $S$. Also, for disjoint vertex sets $S,S'$, let $B(S;S')$ be the number of black edges in the form $e=[u,v]$, $u\in S'$, $v\in S$. 
\begin{definition}The \emph{$D$-index} of $S$ is defined by $$D(S):=R(S)-B(S)-2B(S;S^c).$$
\end{definition}

\begin{definition}
We say that a \emph{two-colored backbone graph} $G$ is \emph{amenable} if it also satisfies:
\begin{itemize}
    \item The second part in the above definition of black edges is empty, i.e., all the black edges of $G$ exactly form a copy of $\tree$ with simple edges. Also, $\tree$ contains all vertices of $G$. That is, $\tree$ is a spanning tree.
    \item For every vertex set $S$ not containing $\vroot$, if $S$ is not empty then $D(S)\le -2$.
\end{itemize}
\end{definition}

We may also define the \emph{embedding} map for a \emph{two-colored backbone graph} $G$ in a similar way. We say that a map $\embed$ is an \emph{embedding} if 
$$\embed=\left(\embed^{(1)},\embed^{(2)}\right):V(G)\to \bbZ_+\times\bbZ$$
and for every black edge $e=[u,v]$, $\embed^{(1)}(u)>\embed^{(1)}(v)$. Weights are assigned differently to edges of different colors. For a black edge $e=[u,v]$,
$$\weight{\embed}(e):=\cha'\left(\embed^{(1)}(u)-\embed^{(1)}(v),\embed^{(2)}(u)-\embed^{(2)}(v)\right).$$
For a red edge $e=\{u,v\}$,
$$c_\embed(e):=\frac{1}{\sqrt{\absolute{\embed^{(1)}(u)-\embed^{(1)}(v)}+1}}.$$
Let $\embedset{t,x}(G)$ denote the set of all embedding maps of $G$ such that $\embed(\vroot)=(t,x)$. The \emph{action} of $G$ at $(t,x)$ is defined as
\begin{equation}\label{E.def of action of G}
    \action{t,x}{G}:=\sum_{\embed\in\embedset{t,x}(G)}\prod_{e\in E(G)}\weight{\embed}(e).
\end{equation}
In Sections \ref{Subs.Leaf lifting}--\ref{Subs.Lifting 3}, we develop a combinatorial reduction procedure called \emph{graph lifting}, which, combined with the kernel estimates in Appendix \ref{S.Some Lemmas}, proves Theorem \ref{T.moments upper bound on K 1}. 
\subsection{Leaf lifting}\label{Subs.Leaf lifting} 
In this section, we reduce the glued leaves of $\tilde{\tree}^\pi$.
Given a tree $\tilde{\tree}$, its \emph{naked part} $\tilde{\tree}^\circ$ is the subgraph of $\tilde{\tree}$ obtained by removing all the leaves.
\begin{proposition}[\textbf{Leaf lifting}]\label{P.lifting 1}
    Suppose $\tree$, $\pi$ and $\tilde{\tree}$ are as in Proposition \ref{P.binary tree}. Then, there exists an amenable graph $G$ with backbone $\tilde{\tree}^\circ$ and at least 
    $$\frac{\setsize{L(\tree)}}{2}+\frac{1}{2}\sum_{\pi_i\in\pi}(\setsize{\pi_i}-4)^+$$
     red edges. Also, for any $x$ and $t\le aN$,
    $$\action{t,x}{\tilde{\tree}^\pi}\loglesssim\action{t,x}{G}.$$
\end{proposition}
\begin{proof}  
    We construct a sequence of two-colored backbone graphs $G_0,\ldots,G_r$ that interpolate between $\tilde{\tree}^\pi$ and an amenable $G$. We set $G_0=\tilde{\tree}^\pi, G_r=G$, and $r=\setsize{\pi}$, the number of parts in $\pi$. Proceeding as in the proof of Proposition \ref{P.binary tree}, we construct a sequence $G_i,0\le i\le r$, such that
    \begin{equation}\label{E.Lifting 1}
        \action{t,x}{G_i}\loglesssim\action{t,x}{G_{i+1}}
    \end{equation}
    for every $0\le i\le r-1$. Since $r=\setsize{\pi}$ is finite, the proof is complete after verifying that $G=G_r$ is indeed amenable.
    
    First, we assign an integer-valued label to each vertex of $\tilde{\tree}^\circ$, with all labels distinct. For any vertex $v$ of $\tilde{\tree}^\circ$:
    \begin{itemize}
        \item If $v$ has one child in $\tilde{\tree}^\circ$, then the label of $v$ is smaller than the label of that child.
        \item If $v$ has two children in $\tilde{\tree}^\circ$, then its label lies between the labels of its two children.
    \end{itemize}  
    The existence of such labels follows by induction. These labeling requirements ensure that the graph $G$ obtained below is amenable. Assume that the parts of $\pi$ are $\pi_1,\ldots,\pi_r$.
    Each part of $\pi$ corresponds to a vertex of $\tilde{\tree}^\pi\setminus\tilde{\tree}^\circ$. Under this identification, we delete $\pi_1, \ldots, \pi_r$ one by one in order, and at each step add corresponding red edges. We inductively construct $G_{i}$, whose vertices are the vertices of $\tilde{\tree}^\circ$ and $\pi_{i+1},\ldots,\pi_r$. Assume that we already have $G_{i}$, $i\ge 0$, and next we construct $G_{i+1}$.
    Suppose $\pi_{i+1}=\{v_1,\ldots,v_k\}$ and $u_1,\ldots,u_m$ are all the vertices connected to $\pi_{i+1}$ in $\tilde{\tree}^\pi$. Here we also assume that the labels of $u_1,\ldots,u_m$ are in increasing order. Since $\tilde{\tree}$ is a full binary tree and $\pi$ is proper, we must have $m\ge 2$ and each $u_j$ is connected to $\pi_{i+1}$ by one or two edges in $\tilde{\tree}^\pi$. There are three cases.
    
    \textbf{Case 1:} If $m=k$, then every $u_j$ is connected to $\pi_{i+1}$ by a single edge. We delete $\pi_{i+1}$ and add red edges $\{u_{j},u_{j+1}\}, 1\le j\le m-1$. Hence, $k-1$ red edges are added. Let $G_{i+1}$ be the resulting graph. Then \eqref{E.Lifting 1} follows from Lemmas \ref{NL.lifting 1-1} and \ref{NL.lifting 1-2}.
    
    \textbf{Case 2:} If $m=2,k=3$, then one of $u_1, u_2$ is connected to $\pi_{i+1}$ by a double edge.  We delete $\pi_{i+1}$ and add two red edges $\{u_1,u_2\}$. Let $G_{i+1}$ be the resulting graph. Then \eqref{E.Lifting 1} follows from Lemma \ref{NL.lifting 1-3}. In this case, we say \textbf{Case 2 occurred at $u_1$} (resp. $u_2$) if $\pi_{i+1}$ is connected to $u_1$ (resp. $u_2$) by \textbf{one edge}.
     
    \textbf{Case 3:} If $m<k$ and $k\ge 4$, then exactly $k-m$ of the $u_1,\dots,u_m$ are connected to $\pi_{i+1}$ by a double edge. Suppose they are $u_{i_{1}},\ldots,u_{i_{k-m}}$, $1\le i_1<\dots<i_{k-m}\le m$. We delete $\pi_{i+1}$ and add red edges $\{u_{j},u_{j+1}\},1\le j\le m-1$ and $\{u_{i_{j}},u_{i_{j+1}}\},1\le j\le k-m-1$. This gives $k-2$ edges in total. Let $G_{i+1}$ be the resulting graph. Then \eqref{E.Lifting 1} follows from Lemmas \ref{NL.lifting 1-1} and \ref{NL.lifting 1-2}.

    \begin{figure}[htbp]
\centering
\begin{tikzpicture}[
  x=1cm,y=1cm,
  line cap=round,line join=round,
  vtx/.style={circle,fill=black,inner sep=1.7pt},
  bedge/.style={draw=black,line width=0.95pt},
  dbedge/.style={draw=black,line width=0.75pt,double,double distance=1.1pt},
  redge/.style={draw=red,line width=1.0pt},
  dredge/.style={draw=red,line width=0.8pt,double,double distance=1.1pt},
  downarrow/.style={-{Stealth[length=3mm,width=2mm]},line width=0.95pt},
  symlab/.style={font=\Large},
  casehead/.style={font=\large}
]

% frame
\draw[line width=1.2pt] (0,0) rectangle (15.9,6.75);
\draw[line width=1.2pt] (0,5.9) -- (15.9,5.9);
\draw[line width=1.2pt] (5.3,0) -- (5.3,6.75);
\draw[line width=1.2pt] (10.6,0) -- (10.6,6.75);

% headers
\node[casehead] at (2.65,6.18) {\textbf{Case 1}};
\node[casehead] at (7.95,6.18) {\textbf{Case 2}};
\node[casehead] at (13.25,6.18) {\textbf{Case 3}};

% =====================
% Case 1
% =====================
\coordinate (c1r) at (2.65,4.40);
\coordinate (c1u1) at (0.95,5.1);
\coordinate (c1u2) at (1.72,5.1);
\coordinate (c1u3) at (2.38,5.10);
\coordinate (c1um) at (4.35,5.1);

\draw[bedge] (c1r) -- (c1u1);
\draw[bedge] (c1r) -- (c1u2);
\draw[bedge] (c1r) -- (c1u3);
\draw[bedge] (c1r) -- (c1um);

\node[vtx] at (c1r) {};
\node[vtx] at (c1u1) {};
\node[vtx] at (c1u2) {};
\node[vtx] at (c1u3) {};
\node[vtx] at (c1um) {};

\node at (0.95,5.46) {$u_1$};
\node at (1.67,5.46) {$u_2$};
\node at (2.39,5.46) {$u_3$};
\node at (3.32,5.14) {$\cdots$};
\node at (4.36,5.46) {$u_m$};
\node at (2.88,4.0) {$\pi_{i+1}$};

\node[symlab] at (2.65,3) {$\Downarrow$};

\coordinate (c1b1) at (0.75,1.65);
\coordinate (c1b2) at (1.5,1.65);
\coordinate (c1b3) at (2.25,1.65);
\coordinate (c1bm1) at (3.95,1.65);
\coordinate (c1bm) at (4.55,1.65);

\draw[redge] (c1b1) -- (c1b2) -- (c1b3);
\draw[redge] (c1bm1) -- (c1bm);

\foreach \p in {c1b1,c1b2,c1b3,c1bm1,c1bm}
  \node[vtx] at (\p) {};

\node at (0.65,2.08) {$u_1$};
\node at (1.50,2.08) {$u_2$};
\node at (2.32,2.08) {$u_3$};
\node at (3.22,1.85) {$\cdots$};
\node at (4.82,2.08) {$u_m$};

% =====================
% Case 2
% =====================
% left top
\coordinate (c2lr) at (6.70,4.38);
\coordinate (c2lu1) at (6.25,5.27);
\coordinate (c2lu2) at (7.1,5.27);

\draw[bedge]  (c2lr) -- (c2lu1);
\draw[dbedge] (c2lr) -- (c2lu2);

\node[vtx] at (c2lr) {};
\node[vtx] at (c2lu1) {};
\node[vtx] at (c2lu2) {};

\node at (6.22,5.58) {$u_1$};
\node at (7.1,5.58) {$u_2$};
\node at (6.6,4.00) {$\pi_{i+1}$};

% right top
\coordinate (c2rr) at (9.10,4.38);
\coordinate (c2ru1) at (8.70,5.27);
\coordinate (c2ru2) at (9.48,5.27);

\draw[dbedge] (c2rr) -- (c2ru1);
\draw[bedge]  (c2rr) -- (c2ru2);

\node[vtx] at (c2rr) {};
\node[vtx] at (c2ru1) {};
\node[vtx] at (c2ru2) {};

\node at (8.68,5.58) {$u_1$};
\node at (9.50,5.58) {$u_2$};
\node at (9.1,4.00) {$\pi_{i+1}$};

\node[font=\normalsize] at (7.95,4.82) {or};

\node[symlab] at (7.95,3) {$\Downarrow$};

\coordinate (c2b1) at (7.08,1.70);
\coordinate (c2b2) at (8.82,1.70);

\draw[dredge] (c2b1) -- (c2b2);

\node[vtx] at (c2b1) {};
\node[vtx] at (c2b2) {};

\node at (7.08,2.10) {$u_1$};
\node at (8.82,2.10) {$u_2$};

% =====================
% Case 3
% =====================
\coordinate (c3r)   at (13.12,3.7);
\coordinate (c3t1)  at (11.02,4.7);
\coordinate (c3t2)  at (11.70,4.7);
\coordinate (c3t3)  at (12.30,4.7);
\coordinate (c3tk)  at (13.40,4.7);
\coordinate (c3s1)  at (14.48,4.7);
\coordinate (c3s2)  at (15.55,4.7);

\draw[dbedge] (c3r) -- (c3t1);
\draw[dbedge] (c3r) -- (c3t2);
\draw[dbedge] (c3r) -- (c3t3);
\draw[dbedge] (c3r) -- (c3tk);
\draw[bedge]  (c3r) -- (c3s1);
\draw[bedge]  (c3r) -- (c3s2);

\node[vtx] at (c3r) {};
\foreach \p in {c3t1,c3t2,c3t3,c3tk,c3s1,c3s2}
  \node[vtx] at (\p) {};

\node at (13.18,5.38) {$\overbrace{\qquad\qquad\qquad\qquad\qquad\quad}^{u_1\sim u_m}$};

\node at (11.14,4.94) {$u_{i_1}$};
\node at (11.84,4.94) {$u_{i_2}$};
\node at (12.46,4.94) {$u_{i_3}$};
\node at (12.92,4.72) {$\cdots$};
\node at (13.72,4.94) {$u_{i_{k-m}}$};
\node at (14.96,4.72) {$\cdots$};
\node at (13.12,3.40) {$\pi_{i+1}$};

\node[symlab] at (13.32,2.95) {$\Downarrow$};

% first red chain
\coordinate (c3u1) at (11.16,2.12);
\coordinate (c3u2) at (11.78,2.12);
\coordinate (c3u3) at (12.44,2.12);
\coordinate (c3um1) at (14.32,2.12);
\coordinate (c3um) at (14.90,2.12);

\draw[redge] (c3u1) -- (c3u2) -- (c3u3);
\draw[redge] (c3um1) -- (c3um);

\foreach \p in {c3u1,c3u2,c3u3,c3um1,c3um}
  \node[vtx] at (\p) {};

\node at (11.26,2.52) {$u_1$};
\node at (11.9,2.52) {$u_2$};
\node at (12.55,2.52) {$u_3$};
\node at (13.36,2.25) {$\cdots$};
\node at (15.20,2.52) {$u_m$};

\node[font=\Large] at (13.32,1.55) {$+$};

% second red chain
\coordinate (c3v1) at (11.22,0.72);
\coordinate (c3v2) at (11.90,0.72);
\coordinate (c3vk1) at (14.32,0.72);
\coordinate (c3vk) at (14.90,0.72);

\draw[redge] (c3v1) -- (c3v2);
\draw[redge] (c3vk1) -- (c3vk);

\foreach \p in {c3v1,c3v2,c3vk1,c3vk}
  \node[vtx] at (\p) {};

\node at (11.42,1.10) {$u_{i_1}$};
\node at (12.14,1.10) {$u_{i_2}$};
\node at (13.34,0.88) {$\cdots$};
\node at (15.28,1.10) {$u_{i_{k-m}}$};

\end{tikzpicture}
\caption{Three cases of Leaf lifting.}
\label{F.lifting 1}
\end{figure}
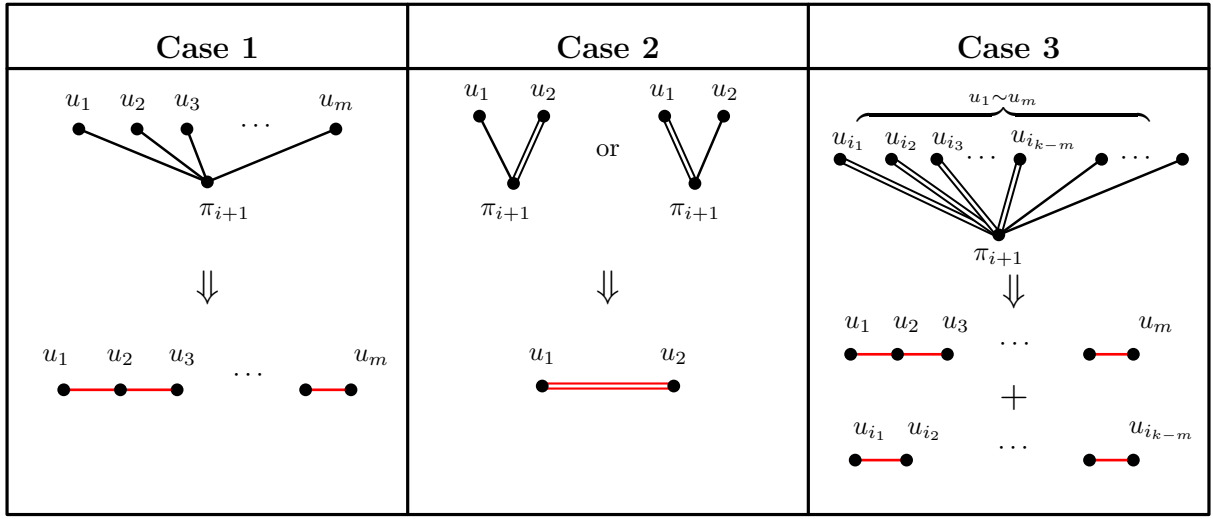

    The procedures for the above cases are illustrated in Figure \ref{F.lifting 1}. In any of these cases, $G_{i+1}$ has $\tilde{\tree}^\circ$ as backbone. Thus we may construct $\tilde{\tree}^\pi=G_0,G_1,\ldots,G_r$ by induction. Finally, the graph $G_r$ has $\tilde{\tree}^\circ$ as a spanning tree. We let $G=G_r$. Note that in any of the above three cases, at least $k-\min\{2,k/2\}$ red edges are added to $G$, where $k=\setsize{\pi_i}$. Therefore, after summing over all $1\le i\le n$, we obtain that the number of red edges in $G$ is at least:    
    $$\sum_{i=1}^n\left(\setsize{\pi_i}-\min\{2,\frac{1}{2}\setsize{\pi_i}\}\right)=\frac{1}{2}\sum_{i=1}^n\left(\setsize{\pi_i}+(\setsize{\pi_i}-4)^+\right)=\frac{\setsize{L(\tree)}}{2}+\frac{1}{2}\sum_{\pi_i\in\pi}(\setsize{\pi_i}-4)^+.$$

    Now we proceed to verify that $G$ is amenable. By the above edge-connection rule, for a non-root vertex $u$ in $\tilde{T}^\circ$, the number of red and black edges connected to $u$ has the following properties:
    \begin{itemize}
        \item If $u$ is connected to \textbf{no leaves} in $\tilde{\tree}$ (that is, it is connected to \textbf{three black edges} in $\tilde{\tree}^\circ$), then $u$ is connected to \textbf{no red edge} in $G$.
        \item If $u$ is connected to \textbf{one leaf} in $\tilde{\tree}$ (that is, it is connected to \textbf{two black edges} in $\tilde{\tree}^\circ$),   then $u$ is connected to \textbf{at most two red edges} in $G$.
        \item If $u$ is connected to \textbf{two leaves} in $\tilde{\tree}$ (that is, it is connected to \textbf{one black edge} in $\tilde{\tree}^\circ$), then $u$ is connected to \textbf{at most four red edges} in $G$.
    \end{itemize}
    For a set of non-root vertices $S$ and a vertex $u\in S$, denote by $\reddeg^{S}(u)$ the number of red edges in $S$ that $u$ is connected to, and $\blackdeg^{S}(u)$ the number of black edges in $S$ of the form $[u,v]$. The \emph{freedom} of $u$ in $S$ is defined as $\freedom^{S}(u):=\reddeg^{S}(u)+2\blackdeg^{S}(u)-4$. If $S$ consists of all non-root vertices, then we omit $S$ and simply write $\reddeg$, $\blackdeg$, and $\freedom$.
    With this notation, the above observations are equivalent to saying that $\freedom(u)$ is non-positive for every non-root vertex $u$. The following lemma explains the relationship between \emph{freedom} and the \emph{D-index}.
    \begin{lemma}  
         For a set of non-root vertices $S$,
         $$D(S)=\frac{1}{2}\sum_{u\in S}\freedom^S(u)$$
    \end{lemma}
    \begin{proof}
        Since all the black edges in $G$ form a spanning tree of $G$, the black edges in $S$ form a spanning forest of $S$. Suppose the black edges in $S$ form $k$ connected components; then the number of black edges in $S$ is exactly $\setsize{S}-k$, and $B(S;S^c)=k$. Hence,
        \begin{align*}
            2D(S)&=2R(S)-2B(S)-4k=\sum_{u\in S}\reddeg^S(u)-2\sum_{u\in S}\blackdeg^S(u)-4k\\
            &=\sum_{u\in S}\left(\reddeg^S(u)+2\blackdeg^S(u)-4\right)-4\sum_{u\in S}\blackdeg^S(u)+4\setsize{S}-4k=\sum_{u\in S}\freedom^S(u).
        \end{align*}
    \end{proof}
    Now we continue the proof of Proposition \ref{P.lifting 1}. By the lemma above, to prove the desired bound $D(S)\le -2$ for a set $S$ of non-root vertices, it suffices to show
    $$\sum_{u\in S}\freedom^S(u)\le -3.$$
    If $\setsize{S}=1$, then $R(S)=B(S)=0$ and $B(S;S^c)=1$ since the only vertex in $S$ is not the root; thus, by definition, $D(S)=-2$. Now suppose $\setsize{S}\ge 2$.  
    
    If \textbf{Case 2} does not occur on $u$, then the number of red edges from $u$ to vertices with labels larger than that of $u$ is at most $2-\blackdeg u$, and the same holds for vertices with labels smaller than that of $u$.
    For a vertex $u$, we say that $u$ is of type 1 (resp. type 2, type 3) if $\blackdeg u = 0$ (resp. 1,2). Let $u_1$ and $u_2$ be the vertices in $S$ with the smallest and largest labels, respectively. We first assume \textbf{Case 2} did not occur at $u_1$ and $u_2$. For $u=u_1$ or $u_2$,
    \begin{itemize}
        \item If $u$ is of type 1, then $\reddeg^S(u)\le 2$ and therefore $\freedom^S(u)=\reddeg^S(u)-4\le-2$.
        \item If $u$ is of type 2, then $\reddeg^S(u)\le 1$ and therefore $\freedom^S(u)\le\reddeg^S(u)+2\blackdeg(u)-4\le-1$.
        \item If $u$ is of type 3, then $\blackdeg^S(u)\le 1$ because the two children of $u$ in $\tilde{\tree}^\circ$ have labels such that one is larger than that of $u$ and the other is smaller.
        Therefore $\freedom^S(u)=2\blackdeg^S(u)-4\le-2$.
    \end{itemize}
    
    Hence $\freedom^S(u_1)+\freedom^S(u_2)\le-3$ if $u_1$ and $u_2$ are not both of type 2. If they are both of type 2, note that the child of $u_2$ in $\tilde{\tree}^\circ$ has a label larger than the label of $u_2$. This means $\blackdeg^S(u_2)=0$ and then  $\freedom^S(u_2)=\reddeg^S(u_2)-4\le -3$.
    In conclusion, if \textbf{Case 2} did not occur at $u_1$ and $u_2$, we have $\freedom^S(u_1)+\freedom^S(u_2)\le -3$ and consequently $D(S)\le-2$. 

    Now, we assume \textbf{Case 2} occurred at one or both of the vertices $u_1,u_2$. For simplicity, we use a proof by contradiction and assume that $D(S)\ge -1$.
    If \textbf{Case 2} occurred at $u_1$, then there exists another vertex $v$ such that $v$ is of type 1, and the only red edges connected to $v$ are the double red edges between $u_1$ and $v$.
    If $v\notin S$, then we still have $\reddeg^S(u_1)\le 2$ and therefore $\freedom^S(u_1)\le -2$. Together with $\freedom^S(u_2)\le -1$, this leads to a contradiction.
    If $v\in S$ but $v\neq u_2$, then, since $v$ is of type 1 and connected only to two red edges, we have $\freedom^S(v)=2+0-4=-2$. Then $\freedom^S(v)+\freedom^S(u_2)\le -3$, a contradiction.
    
    Now suppose $v=u_2$. For the same reason, we have $\freedom^S(u_2)\le -2$. Then $D(S)\ge -1$ if and only if $\freedom^S(\cdot)$ is zero for every vertex in $S$ other than $u_2$. We have $\freedom^S(u_1)=0$ and $u_1$ must be of type 1 or type 2 since it is connected to red edges. 
    
    If $u_1$ is of type 1 and \textbf{Case 2} occurred only once at $u_1$, then among the remaining red edges connected to $u_1$, at most one is connected to a vertex with a label larger than that of $u_1$. Hence $\freedom^S(u_1)\le 3+0-4=-1$, a contradiction. So \textbf{Case 2} must occur again at $u_1$, and all the red edges connected to $u_1$ are between vertices in $S$. This means that there exists another $v'\in S$ of type 1 and the only red edges connected to $v$ are the double red edges between $u_1$ and $v'$. Then, as in the discussion above, we have $\freedom^S(v')=-2$, a contradiction.
    
    So $u_1$ is of type 2, and the child of $u_1$ is in $S$ since $\freedom^S(u_1)=0$. Let $u_3$ be the vertex in $S$ with the second smallest label. If $\setsize{S}=2$, then $u_3=u_2=v$ and $u_3$ is the child of $u_1$. Therefore, the leaves connected to $u_1$ and $u_3$ form a part of $\pi$ and hence $\pi$ is not proper, a contradiction. So $\setsize{S}\ge 3$ and $u_3\neq u_2$. Note that there is no red edge between $u_3$ and $u_1$. If $u_3$ is of type $1$ or type $2$, then $\freedom^S(u_3)\le -1$ because the red edges between $u_3$ and the vertices with smaller labels are not in $S$, a contradiction.
    
    So $u_3$ is of type 3, and the children of $u_3$ are all in $S$. The one with a smaller label must be $u_1$. Now $S$ contains at least four elements: $u_3$, $u_1$, the child of $u_3$ other than $u_1$, and the child of $u_1$. 
    Let $u_4$ be the vertex in $S$ with the third smallest label. $u_4\neq u_1, u_3$. Note that $u_3$ is connected to no red edge. Following the same argument, we obtain that $u_4$ is also of type 3 and $u_3$ must be the child of $u_4$, which has a smaller label. Now $S$ contains at least six elements: $u_1$, $u_3$, $u_4$, the child of $u_1$, the child of $u_3$ other than $u_1$, and the child of $u_4$ other than $u_3$. The above argument can be repeated, but since the number of elements in $S$ is finite, the argument must eventually lead to a contradiction.
    
    If \textbf{Case 2} occurred at $u_2$, the same argument applies. The only difference is that we choose $u_3$ to be the vertex in $S$ with the second largest label instead of the second smallest one. Moreover, $u_2$ must be the child of $u_3$ that has the larger label. Continuing with $u_4,u_5,\ldots$ yields a contradiction.
\end{proof}
\subsection{Backbone pruning} In this subsection, we derive a reduction for amenable graphs. At each step, we remove one of the leaves of the backbone of $G$. 
We therefore call this step \emph{Backbone pruning}.
This will allow us to reduce a complicated amenable graph to the simplest type, in which all the non-root vertices are connected to the root.

\begin{figure}[htbp]
\centering
\resizebox{1\textwidth}{!}{%
\begin{tikzpicture}[
    x=1cm,y=1cm,
    line cap=round,line join=round,
    vtx/.style={circle,fill=black,inner sep=1.8pt},
    bedge/.style={draw=black,line width=0.95pt},
    redge/.style={draw=red,line width=1.05pt},
    casehead/.style={font=\large\bfseries},
    symlab/.style={font=\Large},
    lab/.style={font=\small}
]

% =========================================================
% frame (further enlarged to avoid all overlaps)
% =========================================================
\draw[line width=1.2pt] (0,0.85) rectangle (16.3,8.8);

% horizontal separators
\draw[line width=1.2pt] (0,7.95) -- (16.3,7.95); % below Case 1 header
\draw[line width=1.2pt] (0,5.00) -- (16.3,5.00); % below Case 1 content
\draw[line width=1.2pt] (0,4.15) -- (16.3,4.15); % below Case 2 header

% headers
\node[anchor=west,casehead] at (0.70,8.38) {Case 1:};
\node[anchor=west,casehead] at (0.70,4.58) {Case 2:};

% =========================================================
% Case 1: text and symbols
% =========================================================
\node[symlab] at (1.75,6.43) {$\mathcal{L}_{t,x}$};
\node[symlab] at (2.50,6.43) {$\Bigg ( $};
\node[symlab] at (6.30,6.43) {$\Bigg  ) $};
\node[symlab] at (7.20,6.43) {$\lesssim$};
\node[symlab] at (8.65,6.43) {$N^{1/2}\ \mathcal{L}_{t,x}$};
\node[symlab] at (9.90,6.43) {$\Bigg  ( $};
\node[symlab] at (14.05,6.43) {$\Bigg  ) $};

% =========================================================
% Case 1: left graph
% =========================================================
\begin{scope}[xshift=4.45cm,yshift=6.30cm]
    \coordinate (L)  at (-1.35, 0.52);
    \coordinate (A)  at ( 0.00, 1.42);
    \coordinate (R)  at ( 1.35, 0.52);
    \coordinate (ML) at (-0.58, 0.52);
    \coordinate (MR) at ( 0.48, 0.52);
    \coordinate (BL) at (-1.35,-0.75);
    \coordinate (B1) at (-0.90,-0.75);
    \coordinate (B2) at (-0.25,-0.75);
    \coordinate (V)  at ( 0.30,-0.75);
    \coordinate (B3) at ( 0.92,-0.75);
    \coordinate (BR) at ( 1.35,-0.75);

    \draw[bedge] (L)--(A)--(R);
    \draw[bedge] (L)--(BL);
    \draw[bedge] (A)--(ML) (ML)--(B1) (ML)--(B2);
    \draw[bedge] (A)--(MR) (MR)--(V)  (MR)--(B3);
    \draw[bedge] (R)--(BR);

    \draw[redge] (L)--(B1);
    \draw[redge] (L)--(V);
    \draw[redge] (V)--(R);
    \draw[redge] (B2)--(V);
    \draw[redge] (B3)--(BR);

    \foreach \p in {L,A,R,ML,MR,BL,B1,B2,V,B3,BR}
        \node[vtx] at (\p) {};

    \node[lab] at (-1.32,-1.05) {$v$};
\end{scope}

% =========================================================
% Case 1: right graph
% =========================================================
\begin{scope}[xshift=12.00cm,yshift=6.30cm]
    \coordinate (L)  at (-1.35, 0.52);
    \coordinate (A)  at ( 0.00, 1.42);
    \coordinate (R)  at ( 1.35, 0.52);
    \coordinate (ML) at (-0.58, 0.52);
    \coordinate (MR) at ( 0.48, 0.52);
    \coordinate (B1) at (-0.90,-0.75);
    \coordinate (B2) at (-0.25,-0.75);
    \coordinate (V)  at ( 0.30,-0.75);
    \coordinate (B3) at ( 0.92,-0.75);
    \coordinate (BR) at ( 1.35,-0.75);

    \draw[bedge] (L)--(A)--(R);
    \draw[bedge] (A)--(ML) (ML)--(B1) (ML)--(B2);
    \draw[bedge] (A)--(MR) (MR)--(V)  (MR)--(B3);
    \draw[bedge] (R)--(BR);

    \draw[redge] (L)--(B1);
    \draw[redge] (L)--(V);
    \draw[redge] (V)--(R);
    \draw[redge] (B2)--(V);
    \draw[redge] (B3)--(BR);

    \foreach \p in {L,A,R,ML,MR,B1,B2,V,B3,BR}
        \node[vtx] at (\p) {};
\end{scope}

% =========================================================
% Case 2: symbols
% =========================================================
\node[symlab] at (4.25,2.20) {$\Rightarrow$};
\node[symlab] at (8.50,2.20) {$+$};
\node[symlab] at (12.9,2.20) {$+$};

% =========================================================
% Case 2: initial graph
% =========================================================
\begin{scope}[xshift=2.10cm,yshift=2.15cm]
    \coordinate (L)  at (-1.35, 0.52);
    \coordinate (A)  at ( 0.00, 1.42);
    \coordinate (R)  at ( 1.35, 0.52);
    \coordinate (ML) at (-0.58, 0.52);
    \coordinate (MR) at ( 0.48, 0.52);
    \coordinate (BL) at (-1.35,-0.75);
    \coordinate (B1) at (-0.90,-0.75);
    \coordinate (B2) at (-0.25,-0.75);
    \coordinate (V)  at ( 0.30,-0.75);
    \coordinate (B3) at ( 0.92,-0.75);
    \coordinate (BR) at ( 1.35,-0.75);

    \draw[bedge] (L)--(A)--(R);
    \draw[bedge] (L)--(BL);
    \draw[bedge] (A)--(ML) (ML)--(B1) (ML)--(B2);
    \draw[bedge] (A)--(MR) (MR)--(V)  (MR)--(B3);
    \draw[bedge] (R)--(BR);

    \draw[redge] (L)--(B1);
    \draw[redge] (L)--(V);
    \draw[redge] (V)--(R);
    \draw[redge] (B2)--(V);
    \draw[redge] (B3)--(BR);

    \foreach \p in {L,A,R,ML,MR,BL,B1,B2,V,B3,BR}
        \node[vtx] at (\p) {};

    \node[lab] at (0.30,-1.05) {$v$};
\end{scope}

% =========================================================
% Case 2: first resulting graph
% =========================================================
\begin{scope}[xshift=6.45cm,yshift=2.15cm]
    \coordinate (L)  at (-1.35, 0.52);
    \coordinate (A)  at ( 0.00, 1.42);
    \coordinate (R)  at ( 1.35, 0.52);
    \coordinate (ML) at (-0.58, 0.52);
    \coordinate (MR) at ( 0.48, 0.52);
    \coordinate (BL) at (-1.35,-0.75);
    \coordinate (B1) at (-0.90,-0.75);
    \coordinate (B2) at (-0.25,-0.75);
    \coordinate (B3) at ( 0.92,-0.75);
    \coordinate (BR) at ( 1.35,-0.75);

    \draw[bedge] (L)--(A)--(R);
    \draw[bedge] (L)--(BL);
    \draw[bedge] (A)--(ML) (ML)--(B1) (ML)--(B2);
    \draw[bedge] (A)--(MR) (MR)--(B3);
    \draw[bedge] (R)--(BR);

    \draw[redge] (L)--(B1);
    \draw[redge] (L) to[bend left=12] (R);
    \draw[redge] (B2)--(MR);
    \draw[redge] (B3)--(BR);

    \foreach \p in {L,A,R,ML,MR,BL,B1,B2,B3,BR}
        \node[vtx] at (\p) {};
\end{scope}

% =========================================================
% Case 2: second resulting graph
% =========================================================
\begin{scope}[xshift=10.75cm,yshift=2.15cm]
    \coordinate (L)  at (-1.35, 0.52);
    \coordinate (A)  at ( 0.00, 1.42);
    \coordinate (R)  at ( 1.35, 0.52);
    \coordinate (ML) at (-0.58, 0.52);
    \coordinate (MR) at ( 0.48, 0.52);
    \coordinate (BL) at (-1.35,-0.75);
    \coordinate (B1) at (-0.90,-0.75);
    \coordinate (B2) at (-0.25,-0.75);
    \coordinate (B3) at ( 0.92,-0.75);
    \coordinate (BR) at ( 1.35,-0.75);

    \draw[bedge] (L)--(A)--(R);
    \draw[bedge] (L)--(BL);
    \draw[bedge] (A)--(ML) (ML)--(B1) (ML)--(B2);
    \draw[bedge] (A)--(MR) (MR)--(B3);
    \draw[bedge] (R)--(BR);

    \draw[redge] (L)--(B1);
    \draw[redge] (L) to[bend left=18] (MR);
    \draw[redge] (B2)--(R);
    \draw[redge] (B3)--(BR);

    \foreach \p in {L,A,R,ML,MR,BL,B1,B2,B3,BR}
        \node[vtx] at (\p) {};
\end{scope}

% =========================================================
% Case 2: third resulting graph
% =========================================================
\begin{scope}[xshift=14.65cm,yshift=2.15cm]
    \coordinate (L)  at (-1.25, 0.52);
    \coordinate (A)  at ( 0.00, 1.42);
    \coordinate (R)  at ( 1.25, 0.52);
    \coordinate (ML) at (-0.48, 0.52);
    \coordinate (MR) at ( 0.40, 0.52);
    \coordinate (B1) at (-0.82,-0.75);
    \coordinate (B2) at (-0.18,-0.75);
    \coordinate (B3) at ( 0.88,-0.75);
    \coordinate (BR) at ( 1.25,-0.75);

    \draw[bedge] (L)--(A)--(R);
    \draw[bedge] (A)--(ML) (ML)--(B1) (ML)--(B2);
    \draw[bedge] (A)--(MR) (MR)--(B3);
    \draw[bedge] (R)--(BR);

    \draw[redge] (L)--(B1);
    \draw[redge] (L)--(B2);
    \draw[redge] (MR)--(R);
    \draw[redge] (B3)--(BR);

    \foreach \p in {L,A,R,ML,MR,B1,B2,B3,BR}
        \node[vtx] at (\p) {};
\end{scope}

\end{tikzpicture}
}
\caption{An example of Backbone pruning.}
\label{F.lifting 2}
\end{figure}
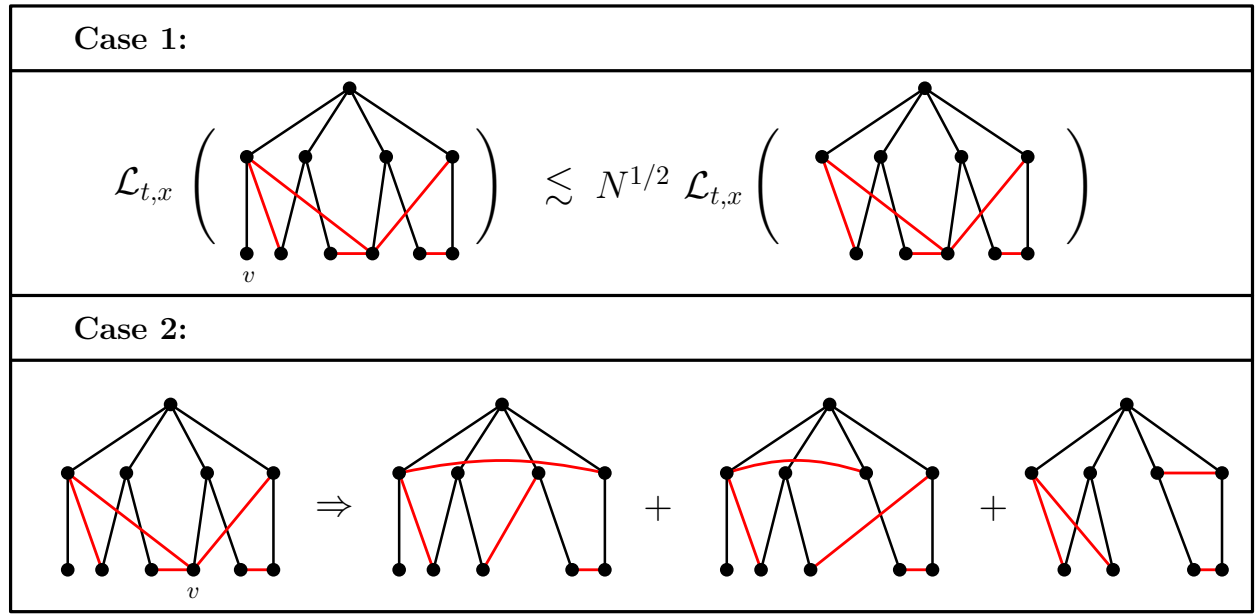

\begin{proposition}[\textbf{Backbone pruning}]\label{P.lifting 2}
    For any full binary tree $\tilde{\tree}$ and an amenable graph $G$ with backbone $\tilde{\tree}^\circ$, if $v$ is a leaf of $\tilde{\tree}^\circ$ which is not connected to the root $\vroot$, let $\tree'$ be the tree obtained by deleting $v$ from $\tilde{\tree}^\circ$.
    \begin{itemize}
    \item If $v$ is not connected to any red edge, let $G'$ be the graph obtained by deleting $v$ from $G$. It is an amenable graph with backbone $\tree'$, and for any $x$ and $t\le aN$,
            \begin{equation}\label{E.lifting 2-2}
                \action{t,x}{G}\lesssim N^{\frac{1}{2}}\action{t,x}{G'}.
            \end{equation}
        \item If $v$ is connected to at least one red edge in $G$, then there exists a family of amenable graphs $\{G'_{w}\}_{w\in N(v)}$ (here $N(v)$ denotes the set of neighbors of $v$ in $G$) all having $\tree'$ as backbone and one fewer red edge than $G$,  such that  for any $x$ and $t\le aN$,
        \begin{equation}\label{E.lifting 2-1}
            \action{t,x}{G}\loglesssim\sum_{w\in N(v)}\action{t,x}{G'_w}.
        \end{equation}
    \end{itemize} 
\end{proposition}
Figure \ref{F.lifting 2}  illustrates the two situations arising in Backbone pruning. Before the proof, we derive several crucial properties of the \emph{D-index}.
\begin{fact}\label{Fact.property of D}
    For any three disjoint vertex sets $X,Y,Z$,
    $$ D(X\cup Y)+D(Y\cup Z)\le D(Y)+D(X\cup Y\cup Z).$$
\end{fact}
\begin{proof}
    For disjoint vertex sets $S,S’$, let $R(S,S’)$ be the number of red edges between $S$ and $S’$. Recall that the number of black edges between $S$ and $S’$ is $B(S;S’)+B(S’;S)$. By definition,
    \begin{align*}
        D(X\cup Y)=&R(X\cup Y)-B(X\cup Y)-2B\left(X\cup Y; (X\cup Y)^c\right)\\
        =&R(X)+R(Y)+R(X,Y)-B(X)-B(Y)-B(X;Y)-B(Y;X)\\&-2B\left(X;(X\cup Y)^c\right)-2B\left(Y;(X\cup Y)^c\right)\\
        =&\left(R(X)-B(X)-2B(X;X^c)\right)+\left(R(Y)-B(Y)-2B(Y;Y^c)\right)+R(X,Y)\\
        &+B(X;Y)+B(Y;X)\\
        =&D(X)+D(Y)+R(X,Y)+B(X;Y)+B(Y;X).
    \end{align*}
    Here we use $B(X;X^c)=B\left(X;(X\cup Y)^c\right)+B(X;Y)$ and vice versa.
    Similarly, $D(X\cup Y\cup Z)$ can be expanded as
    $$D(X\cup Y\cup Z)=D(X)+D(Y\cup Z)+R(X,Y\cup Z)+B(X;Y\cup Z)+B(Y\cup Z;X).$$
    Adding $D(Y)$ to both sides  and comparing the result with the expansion of $D(X\cup Y)$, we obtain
    $$D(Y)+D(X\cup Y\cup Z)=D(X\cup Y)+D(Y\cup Z)+R(X,Z)+B(X;Z)+B(Z;X).$$
    This proves the claim, since the last three terms above are nonnegative.
\end{proof}
\begin{proof}[Proof of Proposition \ref{P.lifting 2}]
    Suppose first that $v$ is not connected to any red edge. For any vertex set $S$ not containing $v$, $D(S)$ is unchanged after deleting $v$ from $G$. Therefore, $G'$ remains amenable and has backbone $T'$. The inequality \eqref{E.lifting 2-2} follows from Lemma \ref{NL.sum of cha}.

    Suppose now that $v$ is connected to at least one red edge in $G$. We need several further properties of the \emph{D-index}. As a corollary of Lemma \ref{Fact.property of D}, for any pair of vertex sets $A$ and $B$, taking $X=A\setminus B$, $Y=A\cap B$, and $Z=B\setminus A$, we obtain
    \begin{equation}\label{E.D-index inequality}
            D(A)+D(B)\le D(A\cap B)+D(A\cup B).
    \end{equation}
    Suppose $w_0$ is the parent of $v$ and $\reddeg(v)=d$. $v$ is connected by a red edge to each of $w_1, \ldots, w_d$. They need not be distinct. Let
    $N(v)=\{w_0,w_1,\ldots,w_d\}$ denote the multiset of neighbors of $v$ in $G$. 
    For any $w\in N(v)$, if $w=w_0$, then $D(\{v,w\})=R(\{v,w\})-3$ and therefore $R(\{v,w\})\le 1$, i.e., $w$ appears at most once in $\{w_1,\ldots,w_d\}$. If $w\neq w_0$, then $D(\{v,w\})=R(\{v,w\})-4$, and $w$  appears at most twice in $N(v)$. Therefore, we have just proved that:
    \begin{fact}
        Any vertex appears in $N(v)$ at most twice.
    \end{fact}
    To discuss subsets of $N(v)$, we fix some conventions here. By $A \subseteq N(v)$, we mean that $A$ is a submultiset of $N(v)$ in the restricted sense that, for every vertex $w$, its multiplicity in $A$ is either $0$ or exactly its multiplicity in $N(v)$. Thus, partial inclusion of repeated elements is excluded; for example, if $N(v)$ contains two copies of a vertex $w$, then a multiset $A$ containing only one copy of $w$ will not be discussed. Moreover, $\setsize{A}$ is defined to be the sum of the multiplicities of its vertices in $N(v)$. If $S$ is a set of vertices, we interpret $S \cap N(v)$ as a submultiset of $N(v)$, where the multiplicity of each vertex is prescribed by the convention stated above.
    
    A second fact follows from a similar argument. Let $S$ be a set of vertices containing neither $v$ nor the root $\vroot$. By distinguishing the cases according to whether $w_0 \in S$, we obtain

    $$D(S\cup\{v\})=D(S)+\setsize{S\cap N(v)}-2.$$
    The left-hand side is $\le -2$ since $G$ is amenable. Therefore,
    \begin{fact}\label{Fact.constraints on D}
        If $\vroot, v\notin S$,
        $$D(S)\le-\setsize{S\cap N(v)}.$$
    \end{fact}
    We next consider the conditions under which the above inequality becomes an equality. To this end, we introduce the following family of submultisets of $N(v)$.
    $$\mathcal{F}:=\left\{A\subseteq N(v):there~is~a~vertex~set~ S_A~s.t.~\vroot, v\notin S_A,~S_A\cap N(v)=A~and~D(S_A)=-\setsize{A}\right\}.$$
    For any pair of sets $A,B\in \mathcal{F}$,  \eqref{E.D-index inequality} gives
    $$D(S_A\cap S_B)+D(S_A\cup S_B)\ge D(S_A)+D(S_B)=-\setsize{A}-\setsize{B}=-\setsize{A\cap B}-\setsize{A\cup B}.$$
    Note that $(S_A\cap S_B)\cap N(v)=A\cap B$ and $(S_A\cup S_B)\cap N(v)=A\cup B$. Together with Fact \ref{Fact.constraints on D}, these identities imply $D(S_A\cap S_B)=-\setsize{A\cap B}$ and $D(S_A\cup S_B)=-\setsize{A\cup B}$, which shows that both $A\cap B$ and $A\cup B$ belong to $\mathcal F$. Thus we obtain the following proposition.

    \begin{proposition}\label{P.property of F}
        The family of sets $\mathcal F$ defined above is closed under union and intersection.
    \end{proposition}
    Since $G$ is amenable, $\mathcal F$ does not contain any set with cardinality 1. By Proposition \ref{P.property of F}, for any $w\in N(v)$,
    $$\bigcap_{w\in A,A\in \mathcal F}A\neq\{w\}.$$
    It follows that there exists an element $w' \in N(v)$ such that, whenever $A \in \mathcal{F}$ and $w \in A$, one also has $w' \in A$. In the special case where $w$ appears twice in $N(v)$, we may take $w'$ to be the other copy of $w$ in $N(v)$. Let $G^\circ := G - v$ denote the graph obtained by deleting $v$ and all edges incident to it. Based on $G^\circ$, we now define a family of graphs $\{G'_w\}_{w \in N(v)}$ by adding red edges, in particular between vertices of $N(v)$.
    For any $w\in N(v)$, we attach red edges to $G^\circ$ as follows.
    \begin{itemize}
        \item\textbf{Case 1}: If the vertex $w$ appears only once in $N(v)$, then there exists another vertex $w'$ that 
        $$\{w,w'\}\subset\bigcap_{w\in A,A\in\mathcal F}A.$$
        The red edges we add are
        $$\{e=\{w,w_i\}:0\le i\le d,w_i\neq w,w'\}.$$
        \item\textbf{Case 2}: If the vertex $w$ appears twice in $N(v)$, then we add red edges:
        $$\{e=\{w,w_i\}:0\le i\le d,w_i\neq w\}.$$
    \end{itemize}
    Note that in either case we add $d-1$ red edges to $G^\circ$; therefore, any resulting $G'_{w}$ has one fewer red edge than $G$.
    We now verify that $G'_w$ is amenable for any $w$. Denote by $D'(\cdot)$ the \emph{D-index} of $G'_{w}$. By our rules for adding red edges, for any vertex set $S$ of $G'_{w}$, if $w$ belongs to Case 1, then
    \begin{align*}
        D'(S)=\begin{cases}
            D(S)+\left(\setsize{S\cap N(v)}-2\right)&\text{$w,w'\in S$.}\\
            D(S)+\left(\setsize{S\cap N(v)}-1\right)&\text{$w\in S$ but $w'\notin S$.}\\
            D(S)&\text{$w\notin S$.}
        \end{cases}
    \end{align*}
    If $w\in S$ but $w'\notin S$, then $S\cap N(v)\notin\mathcal F$ and hence $D(S)<-\setsize{S\cap N(v)}$. Therefore, $D'(S)\le -2$. In the other two cases, we have $D'(S)\le -2$ by the amenability of $G$ and Fact \ref{Fact.constraints on D}. When $w$ belongs to Case 2, then
    \begin{align*}
        D'(S)=\begin{cases}
            D(S)+\left(\setsize{S\cap N(v)}-2\right)&\text{$w\in S$.}\\
            D(S)&\text{$w\notin S$.}
        \end{cases}
    \end{align*}
    So $D'(S)\le -2$ follows from the amenability of $G$ and Fact \ref{Fact.constraints on D}. Thus, we have verified that $G_w'$ is amenable for any $w\in N(v)$. It remains to prove \eqref{E.lifting 2-1}. Each embedding map $\embed\in\embedset{t,x}(G)$ can be decomposed into its restriction to $G^\circ$, which is an embedding map $\embed^\circ\in\embedset{t,x}(G^\circ)$, and the value of $\embed$ at $v$. This gives an expansion
    \begin{align*}
        \action{t,x}{G}=&\sum_{\embed\in\embedset{t,x}(G)}\prod_{e\in E(G)}\weight{\embed}(e)\\
        =&\sum_{\embed^\circ\in\embedset{t,x}(G^\circ)}\prod_{e\in E(G^\circ)}\weight{\embed^\circ}(e)\cdot\\
        &\left(\sum_{s=1}^{\embed^{\circ(1)}(w_0)-1}\sum_{z\in\bbZ}\cha'\left(\embed^{\circ(1)}(w_0)-s,\embed^{\circ(2)}(w_0)-z\right)\prod_{i=1}^{d}\frac{1}{\sqrt{\absolute{\embed^{\circ(1)}(w_i)-s}}+1}\right)\\
        \lesssim&\sum_{\embed^\circ\in\embedset{t,x}(G^\circ)}\prod_{e\in E(G^\circ)}\weight{\embed^\circ}(e)\cdot\sum_{s=1}^{aN}\prod_{i=0}^{d}\frac{1}{\sqrt{\absolute{\embed^{\circ(1)}(w_i)-s}}+1}\loglesssim\sum_{w\in N(v)}\action{t,x}{G_w'}.
    \end{align*}
    Here we use Lemma \ref{NL.sum of cha} in the penultimate inequality and Lemma \ref{NL.lifting 2} in the last inequality. 
\end{proof}
Proposition \ref{P.lifting 2} allows us to integrate out all the vertices of $G$ which are not connected to the root $\vroot$. Recall that $S_n$ is the tree with $n$ non-root vertices, all connected directly to the root. Suppose the root of $\tilde{\tree}^\circ$ has $n$ children, and the number of edges in $\tilde{\tree}^\circ$ is $n+m$.  By iteratively applying Proposition \ref{P.lifting 2}, we obtain
\begin{equation}\label{E.lifting 2-3}
    \action{t,x}{G}\loglesssim \sum_{G'\in\mathcal G}N^{\frac{1}{2}(m-R(G)+R(G'))}\action{t,x}{G'},
\end{equation}
where $R(G)$ (resp. $R(G')$) is the number of red edges in $G$ (resp. $G'$) and $\mathcal G$ is the set of all amenable graphs with backbone $S_n$ and at least $R(G)-m$ red edges. 
\subsection{Star reduction}\label{Subs.Lifting 3} 
In this subsection, we reduce amenable graphs with star backbones.
\begin{proposition}[\textbf{Star reduction}]\label{P.lifting 3}
    Let $G'$ be an amenable graph with backbone $S_n$, and assume that no red edge is incident to the root $\vroot$. Then for any $x$ and $t\le aN$,
    $$\action{t,x}{G'}\loglesssim N^{\frac{1}{4}(2n-R(G'))}.$$
\end{proposition}

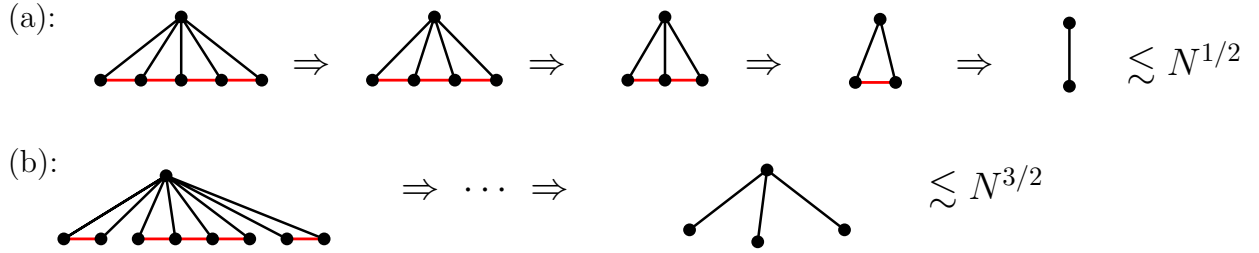
\begin{figure}[htbp]
\centering
\begin{tikzpicture}[
    x=0.82cm,y=0.88cm,
    line cap=round,line join=round,
    vtx/.style={circle,fill=black,inner sep=1.7pt},
    bedge/.style={draw=black,line width=0.95pt},
    redge/.style={draw=red,line width=1.05pt},
    lab/.style={font=\large},
    orlab/.style={font=\Large},
    note/.style={font=\Large}
]

% =========================================================
% row (a)
% =========================================================
\node[anchor=west,lab] at (0.10,3.40) {(a):};

% (a) first graph
\begin{scope}[xshift=2.50cm,yshift=2.18cm]
    \coordinate (t)  at (0,0.95);
    \coordinate (b1) at (-1.30,0.00);
    \coordinate (b2) at (-0.65,0.00);
    \coordinate (b3) at ( 0.00,0.00);
    \coordinate (b4) at ( 0.65,0.00);
    \coordinate (b5) at ( 1.30,0.00);

    \draw[bedge] (t)--(b1) (t)--(b2) (t)--(b3) (t)--(b4) (t)--(b5);
    \draw[redge] (b1)--(b2)--(b3)--(b4)--(b5);

    \node[vtx] at (t) {};
    \foreach \p in {b1,b2,b3,b4,b5} {\node[vtx] at (\p) {};}
\end{scope}

\node[orlab] at (5.15,2.67) {$\Rightarrow$};

% (a) second graph
\begin{scope}[xshift=5.85cm,yshift=2.18cm]
    \coordinate (t)  at (0,0.95);
    \coordinate (b1) at (-1.00,0.00);
    \coordinate (b2) at (-0.33,0.00);
    \coordinate (b3) at ( 0.33,0.00);
    \coordinate (b4) at ( 1.00,0.00);

    \draw[bedge] (t)--(b1) (t)--(b2) (t)--(b3) (t)--(b4);
    \draw[redge] (b1)--(b2)--(b3)--(b4);

    \node[vtx] at (t) {};
    \foreach \p in {b1,b2,b3,b4} {\node[vtx] at (\p) {};}
\end{scope}

\node[orlab] at (8.95,2.67) {$\Rightarrow$};

% (a) third graph
\begin{scope}[xshift=8.90cm,yshift=2.18cm]
    \coordinate (t)  at (0,0.95);
    \coordinate (b1) at (-0.60,0.00);
    \coordinate (b2) at ( 0.00,0.00);
    \coordinate (b3) at ( 0.60,0.00);

    \draw[bedge] (t)--(b1) (t)--(b2) (t)--(b3);
    \draw[redge] (b1)--(b2)--(b3);

    \node[vtx] at (t) {};
    \foreach \p in {b1,b2,b3} {\node[vtx] at (\p) {};}
\end{scope}

\node[orlab] at (12.45,2.67) {$\Rightarrow$};

% (a) fourth graph
\begin{scope}[xshift=11.50cm,yshift=2.15cm]
    \coordinate (t)  at (0.30,0.95);
    \coordinate (b1) at (-0.10,0.00);
    \coordinate (b2) at ( 0.55,0.00);

    \draw[bedge] (t)--(b1) (t)--(b2);
    \draw[redge] (b1)--(b2);

    \node[vtx] at (t) {};
    \node[vtx] at (b1) {};
    \node[vtx] at (b2) {};
\end{scope}

\node[orlab] at (15.85,2.67) {$\Rightarrow$};

% (a) final graph
\begin{scope}[xshift=14.25cm,yshift=2.10cm]
    \coordinate (t) at (0,0.95);
    \coordinate (b) at (0,0.00);
    \draw[bedge] (t)--(b);
    \node[vtx] at (t) {};
    \node[vtx] at (b) {};
\end{scope}

\node[note,anchor=west] at (18.15,2.74) {$\lesssim N^{1/2}$};

% =========================================================
% row (b)
% =========================================================
\node[anchor=west,lab] at (0.10,1.22) {(b):};

% (b) initial graph
\begin{scope}[xshift=2.30cm,yshift=0.08cm]
    \coordinate (t)  at (0,0.95);
    \coordinate (b0) at (-1.65,0.00);
    \coordinate (b1) at (-1.65,0.00);
    \coordinate (b2) at (-1.05,0.00);
    \coordinate (b3) at (-0.45,0.00);
    \coordinate (b4) at ( 0.15,0.00);
    \coordinate (b5) at ( 0.75,0.00);
    \coordinate (b6) at ( 1.35,0.00);
    \coordinate (b7) at ( 1.95,0.00);
    \coordinate (b8) at ( 2.55,0.00);

    \draw[bedge] (t)--(b0) (t)--(b1) (t)--(b2) (t)--(b3) (t)--(b4)
                 (t)--(b5) (t)--(b6) (t)--(b7) (t)--(b8);
    \draw[redge] (b0)--(b1)--(b2) (b3)--(b4)--(b5)--(b6) (b7)--(b8);

    \node[vtx] at (t) {};
    \foreach \p in {b1,b2,b3,b4,b5,b6,b7,b8} {\node[vtx] at (\p) {};}
\end{scope}

\node[orlab] at (7.95,0.82) {$\Rightarrow\ \cdots\ \Rightarrow$};

% (b) final graph
\begin{scope}[xshift=10.25cm,yshift=0.2cm]
    \coordinate (t)  at (0,0.90);
    \coordinate (b1) at (-1.25,0.00);
    \coordinate (b2) at (-0.15,-0.18);
    \coordinate (b3) at ( 1.25,0.00);

    \draw[bedge] (t)--(b1) (t)--(b2) (t)--(b3);

    \node[vtx] at (t) {};
    \foreach \p in {b1,b2,b3} {\node[vtx] at (\p) {};}
\end{scope}

\node[note,anchor=west] at (14.95,0.88) {$\lesssim N^{3/2}$};

\end{tikzpicture}
\caption{An example of Star reduction.}
\label{F.lifting 3}
\end{figure}

\begin{proof}
    We first prove that, if $G'\setminus{\vroot}$ is connected, then
    $$\action{t,x}{G'} \loglesssim N^{1/2}, \qquad \forall~x\text{ and } t\le aN.$$ This will be established by induction on $n$. When $n=1$, the claim follows directly from Lemma \ref{NL.sum of cha}. Assume that the statement holds for all $S_{n'}$, $n'\le n-1$. Let $n\ge 2$. Since $G'\setminus{\vroot}$ is connected, it has a spanning tree. Pick a leaf $v$ of this tree, and let $G''$ be the graph obtained by deleting $v$ from $G'$. Then, $G''\setminus{\vroot}$ remains connected. Suppose that $v$ is connected to a vertex $u$ by a red edge. Then, we have
    \begin{align*}
    \action{t,x}{G'}
    &= \sum_{\embed \in \embedset{t,x}(G')} \prod_{e\in E(G')} \weight{\embed}(e) \\
    &\le \sum_{\embed' \in \embedset{t,x}(G'')} \left[ \prod_{e\in E(G'')} \weight{\embed'}(e) \cdot\sum_{s=1}^{t-1}\sum_{z\in\bbZ} \cha'(t-s,x-z)\frac{1}{\sqrt{\lvert \embed'^{(1)}(u)-s\rvert}+1} \right] \loglesssim \action{t,x}{G''}.
    \end{align*}
    Here we use Lemma \ref{NL.1 red} and \ref{NL.sum of cha} in the last inequality. By the induction hypothesis, the desired bound follows. Figure \ref{F.lifting 3}(a) illustrates the reduction obtained by eliminating non-root vertices step by step, until reaching the case $n=1$, which gives $N^{1/2}$. In Figure \ref{F.lifting 3} (b), we apply this $N^{1/2}$ bound to each component.

    Now consider the general case where $G'\setminus{\vroot}$ has $k$ connected components, denoted $\mathcal C_1,\ldots,\mathcal C_k$. For each $i$, let $G'_i$ be the subgraph of $G'$ induced by $\vroot$ together with $\mathcal C_i$. Each $G'_i$ is an amenable graph with backbone $S_{\lvert\mathcal C_i\rvert}$. We obtain the factorization
    \begin{align*}
    \action{t,x}{G'}
    &= \sum_{\embed\in\embedset{t,x}(G')} \prod_{e\in E(G')} \weight{\embed}(e)= \sum_{\substack{\embed_i\in\embedset{t,x}(G'_i) \\ 1\le i\le k}}    \prod_{i=1}^k \prod_{e\in E(G'_i)} \weight{\embed_i}(e) \\
    &= \prod_{i=1}^k \left[ \sum_{\embed_i\in\embedset{t,x}(G'_i)}    \prod_{e\in E(G'_i)} \weight{\embed_i}(e) \right] = \prod_{i=1}^k \action{t,x}{G'_i}\loglesssim N^{\frac{k}{2}},
    \end{align*}
    where the last inequality uses the previously established bound $\action{t,x}{G'_i}\loglesssim N^{1/2}$ for each $i$.

    Finally, since $G'$ is amenable, each component $\mathcal C_i$ satisfies
    $$R(\mathcal C_i) - 2\setsize{\mathcal C_i}=D(\mathcal C_i)\le -2,$$
    which implies that the number of red edges in $\mathcal C_i$ is at most $2\lvert\mathcal C_i\rvert - 2$. Summing over $i$ gives
    $$ R(G') \le \sum_{i=1}^k \left(2\lvert\mathcal C_i\rvert - 2\right) = 2n - 2k.$$
    Thus $ k\le n-R(G')/2$ and the proof is complete.
\end{proof}

\subsection{Proof of Theorem \ref{T.moments upper bound on K 1}}  
\begin{proof}[Proof of Theorem \ref{T.moments upper bound on K 1}]
    The case $l=1$ can be proved by the same argument as in Theorem \ref{T.moments upper bound on K 2}, with $\truncatednoise$ in place of $\noise$. We now focus on $l\ge 2$. For a fixed $\tree\in\treefamily_l^n$ and a proper partition $\pi$ Proposition \ref{P.binary tree} yields a constant $C$ and a full binary tree $\tilde{\tree}$ whose root has exactly $n$ children, such that
    $$\sum_{\embed\in\embedset{t,x}^\pi(\tree)}\prod_{e\in E(\tree)}\absolute{\weight{\embed}(e)}\lesssim\action{t,x}{\tree^\pi}\lesssim\action{t+C,x}{\tilde\tree^\pi}.$$
    Let $\tree^\circ$ denote the naked part of $\tilde\tree$ and suppose it has $m+n$ edges. Then $\setsize{L(\tree)}=m+2n$. By applying Propositions \ref{P.lifting 1}, \ref{P.lifting 2} and \ref{P.lifting 3}, we obtain
    \begin{align*}
    \action{t,x}{\tilde\tree^\pi}&\loglesssim\action{t,x}{G}\loglesssim\sum_{G'\in\mathcal G}N^{\frac{1}{2}(m-R(G)+R(G'))}\action{t,x}{G'}\loglesssim\sum_{G'\in\mathcal G} N^{\frac{1}{4}(2m+2n+R(G')-2R(G))}\\
    &\lesssim N^{\frac{1}{2}(m+2n-1-R(G))}\lesssim N^{\frac{1}{4}(\setsize{L(\tree)}-2-\sum_{\pi_i}(\setsize{\pi_i}-4)^+)}.
    \end{align*}
    Here $\mathcal G$ is the set of all amenable graphs with backbone $S_n$ and at least $R(G)-m$ red edges. In the second-to-last inequality, we use the fact that $R(G’)\le 2n-2$, since $G’$ is amenable. In the last inequality, we use the fact that the $G$ we constructed in Proposition \ref{P.lifting 1} satisfies
    $$R(G)\ge\frac{\setsize{L(T)}}{2}+\frac{1}{2}\sum_{\pi_i\in\pi}(\setsize{\pi_i}-4)^+=\frac{m+2n}{2}+\frac{1}{2}\sum_{\pi_i\in\pi}(\setsize{\pi_i}-4)^+.$$
    Using the above inequality and \eqref{E.restrict to proper partitions corollary} in Corollary \ref{C.kterm represented by tree}, we may conclude
    \begin{align*}
        \absolute{\bbE\left[\left(\kterm{l}(t,x)\right)^n\right]}\lesssim &\sum_{\tree\in\treefamily_l^n}\absolute{\bbE[\action{t,x}{\tree}]}\\
        \loglesssim&\sum_{\tree\in\treefamily_l^n}N^{-(\frac{1}{4}+\rate)\setsize{L(\tree)}}\sum_{\substack{\pi\in\Pi(L(\tree))\\\pi~is~proper}}N^{\frac{1}{4}\sum_{\pi_i\in\pi}(\setsize{\pi_i}-8)^+}\sum_{\embed\in\embedset{t,x}^\pi(\tree)}\prod_{e\in E(\tree)}\absolute{\weight{\embed}(e)}\\
        \loglesssim&\sum_{\tree\in\treefamily_l^n}N^{-(\frac{1}{4}+\rate)\setsize{L(\tree)}}\sum_{\substack{\pi\in\Pi(L(\tree))\\\pi~is~proper}}N^{\frac{1}{4}\sum_{\pi_i\in\pi}(\setsize{\pi_i}-8)^+}\cdot N^{\frac{1}{4}\setsize{L(\tree)}-\frac{1}{4}\sum_{\pi_i\in\pi}(\setsize{\pi_i}-4)^+}\\
        \lesssim& N^{-\rate\setsize{L(\tree)}}\lesssim N^{-ln\rate}.
    \end{align*}
    The last inequality uses the fact that any tree in $\treefamily_l^n$ has at least $ln$ leaves. This completes the proof.
\end{proof}

\section{Expectation and concentration of  $\xterm{l}$}\label{S.expectation and concentration of X} 
In this section, we prove Theorems  \ref{T.expectation of Xl 1} and \ref{T.concentration on Xl 1} based on the same graph lifting technique developed in Section \ref{S.Graph Lifting}.

\subsection{Proof of Theorem \ref{T.expectation of Xl 1}}

The recursion formula \eqref{E.redefine Xl} shows that $\xterm{l}$ can also be expressed explicitly as a linear combination of lower-order products of $\kterm{m},m\le l-1$. Thus, the same tree representation works, but edges attached to the root are weighted by $p(\cdot,\cdot)$ instead of $\cha(\cdot,\cdot)$.
  
To apply our graph lifting technique, a few more definitions are needed.
\paragraph{Actions.} We may use the same family of trees to represent $\xterm{l}$, but we use a slightly larger class of embeddings since $p(0,\cdot)$ is well-defined.

Fix a tree $\tree$ with root $\vroot$ and $(t,x)\in\bbZ_+\times\bbZ$. Denote by $\embedsetnew{t,x}(\tree)$ the set of all maps $$\embed=\left(\embed^{(1)},\embed^{(2)}\right):V(\tree)\to \bbZ_+\times\bbZ$$
that satisfy $\embed(\vroot)=(t,x)$, $\embed^{(1)}(u)>\embed^{(1)}(v)$ for any $e=[u,v]\in E(\tree)$ such that $u\neq\vroot$ and $\embed^{(1)}(\vroot)\ge\embed^{(1)}(v)$ for any $e=[\vroot,v]\in E(\tree)$. 

For any $\embed\in\embedsetnew{t,x}(\tree)$, we assign another weight function to the edges $E(\tree)$ by
\begin{align*}
    \weightnew{\embed}(e)=
    \begin{cases}
        \cha\left(\embed^{(1)}(u)-\embed^{(1)}(v),\embed^{(2)}(u)-\embed^{(2)}(v)\right), &\text{if $e=[u,v]$ and $u\neq\vroot$;}\\
        p\left(t-\embed^{(1)}(v),x-\embed^{(2)}(v)\right), &\text{if $e=[\vroot,v]$.}
    \end{cases}
\end{align*}
The corresponding action $\actionnew{t,x}{\cdot}$ is defined as
$$\actionnew{t,x}{\tree}:=\sum_{\embed\in\embedsetnew{t,x}(\tree)}\prod_{e\in E(\tree)}\weightnew{\embed}(e)\left[\prod_{v\in L(\tree)}N^{-\frac{1}{4}-\rate}\truncatednoise\left(\embed(v)\right)\right].$$
If $\pi$ is a partition of the leaves of $\tree$, we denote by $\embedsetnew{t,x}^\pi(\tree)$ the subset of $\embedsetnew{t,x}(\tree)$ consisting of all  $\embed\in\embedsetnew{t,x}(\tree)$ that take the same value on two leaves of $\tree$ if and only if these two leaves belong to the same part of $\pi$.
Also, the set of maps $\embedsetnew{t,x}(\cdot)$ for \emph{leaf-glued graphs} and \emph{two-colored backbone graphs} is defined similarly.
For an embedding $\embed\in\embedsetnew{t,x}(\tree^\pi)$ or $\embed\in\embedsetnew{t,x}(G)$, 

\begin{align*}
    \weightnew{\embed}(e)=
    \begin{cases}
        \cha'\left(\embed^{(1)}(u)-\embed^{(1)}(v),\embed^{(2)}(u)-\embed^{(2)}(v)\right) &\text{if $e=[u,v]$ and $u\neq\vroot$.}\\
        p\left(t-\embed^{(1)}(v),x-\embed^{(2)}(v)\right) &\text{if $e=[\vroot,v]$.}\\
        \left(\absolute{\embed^{(1)}(u)-\embed^{(1)}(v)}+1\right)^{-1/2}&\text{if $e=\{u,v\}$ is a red edge.}
    \end{cases}
\end{align*}
The actions $\actionnew{t,x}{\cdot}$ are defined in the same way as the right-hand sides of \eqref{E.def of action of T pi} and \eqref{E.def of action of G}, where $\weight{\embed}(e)$ is replaced by $\weightnew{\embed}(e)$. In Propositions \ref{P.binary tree}, \ref{P.lifting 1}, and \ref{P.lifting 2}, if $\action{t,x}{\cdot}$ is replaced by $\actionnew{t,x}{\cdot}$, the arguments remain valid and yield the same conclusions for $\actionnew{t,x}{\cdot}$.

With this notation, $\xterm{l}$ can also be represented as an action of trees. The following proposition can be proved in the same way as Proposition \ref{P.kterm represented by tree}.

\begin{proposition}\label{P.xterm represented by tree}
    For a fixed $l\ge1$ and any $(t,x)\in\bbZ_+\times\bbZ$,
    \begin{equation}\label{E.xterm represented by tree}
        \xterm{l}(t,x)=\sum_{\tree\in\treefamily_l}c(\tree,l)\actionnew{t,x}{\tree}.
    \end{equation}
    Here, the constants $c(\tree,l)$ are the same as those  defined in Proposition \ref{P.kterm represented by tree}.
\end{proposition}
Then, we can reduce improper partitions to proper partitions as we did in Proposition \ref{P.restrict to proper partitions}. For an improper partition $\pi$, if a vertex $u$ is a bad point of $\pi$, then one may reduce $\pi$ to a collection of partitions, none of which has $u$ as a bad point. This reduction relies on the fact that the edge $e$ connecting $u$ to its parent $v$ is weighted by $\cha(\cdot,\cdot)$, and the following summation vanishes:
$$\sum_{z\in\bbZ}\cha\left(\embed_1^{(1)}(v)-s,\embed_1^{(2)}(v)-z\right)=0.$$
The only obstacle to deriving the same result for $\actionnew{t,x}{\cdot}$ is that if a bad point of $\pi$ is the child of the root, then that edge is weighted by $p(\cdot,\cdot)$, whose summation over space does not vanish. Fortunately, if $\tree\in\treefamily_l$, such a case occurs for only a few tree--partition pairs $(\tree,\pi)$. Indeed, a vertex $v$ is a bad point for a partition $\pi$ if and only if the subtree rooted at $v$ belongs to one of the three cases, and its leaves form a part of $\pi$. Therefore, for $\tree\in\treefamily_l$, the root has only one child, and that child is a bad point if and only if $(\tree,\pi)$ belongs to one of the three cases shown in Figure \ref{F.proper-counterexamples} (b) and (c).

\begin{proposition}\label{P.restrict to proper partitions new}
    If $\tree\in\treefamily_l$, $l\ge 2$ and $(\tree,\pi)$ does not belong to any of the three cases shown in Figure \ref{F.proper-counterexamples} (b) and (c), then for any $x$ and $t\le aN$,
    \begin{equation}\label{E.restrict to proper partitions new}
        \absolute{\sum_{\embed\in\embedsetnew{t,x}^{\pi}(\tree)}\prod_{e\in E(\tree)}\weightnew{\embed}(e)\moment{\embed}{\pi}}\loglesssim \sum_{\substack{\tilde{\pi}\in\Pi(L(\tree))\\ \tilde{\pi}~\text{is proper}}}N^{\frac{1}{4}\sum_{\tilde{\pi}_i\in\tilde{\pi}}(\setsize{\tilde{\pi}_i}-8)^+}\cdot\sum_{\embed\in\embedsetnew{t,x}^{\tilde{\pi}}(\tree)}\prod_{e\in E(\tree)}\absolute{\weightnew{\embed}(e)}.
    \end{equation}
\end{proposition}
The proof is identical to that of Proposition \ref{P.restrict to proper partitions}. We now prove Theorem \ref{T.expectation of Xl 1}.
\begin{proof}[Proof of Theorem \ref{T.expectation of Xl 1}]
   We show that, if $\tree\in\treefamily_l$, $l\ge 2$ and $(\tree,\pi)$ does not belong to any of the three cases shown in Figure \ref{F.proper-counterexamples} (b) and (c), then for any $x$ and $t\le aN$,
    \begin{equation}\label{E.bound on usual cases}
        \absolute{\sum_{\embed\in\embedsetnew{t,x}^{\pi}(\tree)}\prod_{e\in E(\tree)}\weightnew{\embed}(e)\moment{\embed}{\pi}}\loglesssim N^{\frac{1}{4}\setsize{L(\tree)}}.
    \end{equation}
    By Proposition \ref{P.restrict to proper partitions new}, it suffices to show 
    $$\absolute{\sum_{\embed\in\embedsetnew{t,x}^{\pi}(\tree)}\prod_{e\in E(\tree)}\weightnew{\embed}(e)}\loglesssim N^{\frac{1}{4}\setsize{L(\tree)}-\frac{1}{4}\sum_{\pi_i\in\pi}(\setsize{\pi_i}-8)^+}$$
    for any proper partition $\pi$. Note that the root of $\tree$ has exactly one child. By Proposition \ref{P.binary tree}, there exists a constant $C$ and a full binary tree $\tilde{\tree}$ whose root has exactly one child such that 
    $$\absolute{\sum_{\embed\in\embedsetnew{t,x}^{\pi}(\tree)}\prod_{e\in E(\tree)}\weightnew{\embed}(e)}\lesssim\actionnew{t,x}{\tree^\pi}\lesssim \actionnew{t+C,x}{\tilde{\tree}^\pi}$$
    for all $x$ and $t\le aN$. Denote the naked part of $\tilde{\tree}$ by $\tree^\circ$ and suppose it has $m+1$ edges. Then $\setsize{L(\tree)}=m+2$. By applying Propositions \ref{P.lifting 1} and \ref{P.lifting 2} followed by inequality \eqref{E.lifting 2-3}, we obtain

    $$\actionnew{t,x}{\tilde{\tree}^\pi}\loglesssim N^{\frac{1}{2}(m-R(G))}\actionnew{t,x}{G'},$$
    where $G'$ is exactly $\one$, the graph with only one black edge. Hence $\actionnew{t,x}{G'}=t\loglesssim N$ and so
    \begin{align*}
        \actionnew{t,x}{\tilde{\tree}^\pi}\loglesssim N^{\frac{1}{2}(m+2-R(G))}\le N^{\frac{1}{4}(m+2)-\frac{1}{4}\sum_{\pi_i\in\pi}(\setsize{\pi_i}-4)^+}.
    \end{align*}
    Here, in the last inequality, we use $R(G)\ge\setsize{L(\tree)}/2+\sum_{\pi_i\in\pi}(\setsize{\pi_i}-4)^+/2$ and $\setsize{L(\tree)}=m+2$. This finishes the proof of \eqref{E.bound on usual cases}. 
    
    Taking expectations in \eqref{E.xterm represented by tree}, we obtain
    \begin{align*}
        \bbE[\xterm{l}(t,x)]&=\sum_{\tree\in\treefamily_l}c(\tree,l)\bbE[\actionnew{t,x}{\tree}]\\
        &=\sum_{\tree\in\treefamily_l}\sum_{\pi\in\Pi(L(\tree))}c(\tree,l)N^{-(\frac{1}{4}+\rate)\setsize{L(\tree)}}\sum_{\embed\in\embedsetnew{t,x}^{\pi}(\tree)}\prod_{e\in E(\tree)}\weightnew{\embed}(e)\moment{\embed}{\pi}.
    \end{align*}
    If $(\tree,\pi)$ does not belong to any of the three cases shown in Figure \ref{F.proper-counterexamples} (b) and (c), then 
    $$\setsize{N^{-(\frac{1}{4}+\rate)\setsize{L(\tree)}}\sum_{\embed\in\embedsetnew{t,x}^{\pi}(\tree)}\prod_{e\in E(\tree)}\weightnew{\embed}(e)\moment{\embed}{\pi}}\loglesssim N^{-\rate\setsize{L(\tree)}},$$
    which is bounded by $N^{-2\rate}$ since any $\tree\in\treefamily_l,l\ge 2$ has at least two leaves. So we only need to consider the exceptional cases.
    Denote the three cases in Figure \ref{F.proper-counterexamples} (b) and (c) by $(\tree_1,\pi_1), (\tree_2,\pi_2)$ and $(\tree_3,\pi_3)$. We have $\moment{\embed}{\pi_1}=\mu_2+o(N^{-1})$ for any $\embed\in\embedsetnew{t,x}(\tree_1)$, and the same holds for $(\tree_2,\pi_2)$ or $(\tree_3,\pi_3)$. Thus, we have proved
    \begin{align}
        \nonumber
       \bbE[\xterm{2}(t,x)]=&c(\tree_1,2)N^{-\frac{1}{2}-2\rate}\sum_{\embed\in\embedsetnew{t,x}^{\pi_1}(\tree_1)}\prod_{e\in E(\tree_1)}\weightnew{\embed}(e)(\mu_2+o(N^{-1}))\\
        \label{E.expectation of X2 midstep}
        &+c(\tree_2,2)N^{-\frac{3}{4}-3\rate}\sum_{\embed\in\embedsetnew{t,x}^{\pi_2}(\tree_2)}\prod_{e\in E(\tree_2)}\weightnew{\embed}(e)(\mu_3+o(N^{-1}))+o(N^{-\rate}).\\
        \label{E.expectation of X3 midstep}
        \bbE[\xterm{3}(t,x)]=&c(\tree_3,3)N^{-\frac{3}{4}-3\rate}\sum_{\embed\in\embedsetnew{t,x}^{\pi_3}(\tree_3)}\prod_{e\in E(\tree_3)}\weightnew{\embed}(e)(\mu_3+o(N^{-1}))+o(N^{-\rate}).\\
        \nonumber
        \bbE[\xterm{l}(t,x)]=&o(N^{-\rate}),\forall~l\ge 4.
    \end{align}
    Therefore, we only need to compute those three non-vanishing terms. First, $\tree_1$ and $\tree_2$ are unique in $\treefamily_2$ up to isomorphism. There are two trees isomorphic to $\tree_3$ in $\treefamily_3$. We have $c(\tree_1,2)=\beta/8$, $c(\tree_2,2)=0$ (since $\driving$ is an even function) and $c(\tree_3,3)=\beta^2/32$. 
    The main terms in \eqref{E.expectation of X2 midstep} and \eqref{E.expectation of X3 midstep} can be computed explicitly. For the main term in $\bbE[\xterm{2}(t,x)]$,

    \begin{align*}
        \sum_{\embed\in\embedsetnew{t,x}^{\pi_1}(\tree_1)}\prod_{e\in E(\tree)}\weightnew{\embed}(e)&=\sum_{s=1}^{t}\sum_{z\in\bbZ}p(t-s,x-z)\cdot\sum_{s_1=1}^{s-1}\sum_{z_1\in\bbZ}\cha(s-s_1,z-z_1)^2\\
        &=\sum_{s=1}^{t}\sum_{s_1=1}^{s-1}\sum_{z_1\in\bbZ}\cha(s-s_1,z_1)^2=\sum_{s=1}^{t}4(1-p(2s-2,0))\\
        &=4t+O(t^{\frac{1}{2}}).
    \end{align*}
    The second equality follows from Lemma \ref{NL.reflection identity}, and the third from Stirling's formula.

    For the main term in $\bbE[\xterm{3}(t,x)]$ we have
    \begin{align*}
        &\sum_{\embed\in\embedsetnew{t,x}^{\pi_3}(\tree_3)}\prod_{e\in E(\tree)}\weightnew{\embed}(e)\\
        =&\sum_{\substack{1\le s\le t\\z\in\bbZ}}p(t-s,x-z)\sum_{\substack{1\le s_1\le s-1\\z_1\in\bbZ}}\cha(s-s_1,z-z_1)\sum_{\substack{1\le s_2\le s_1-1\\z_2\in\bbZ}}\cha(s-s_2,z-z_2)\cha(s_1-s_2,z_1-z_2)^2\\
        =&\sum_{s=1}^{t}\sum_{\substack{1\le s_1\le s-1\\z_1\in\bbZ}}\cha(s-s_1,-z_1)\sum_{\substack{1\le s_2\le s_1-1\\z_2\in\bbZ}}\cha(s-s_2,-z_2)\cha(s_1-s_2,z_1-z_2)^2.
     \end{align*}
     For any fixed $s$, changing variables $s^*:=s_1-s_2$ and $z^*:=z_1-z_2$ gives
     \begin{align*}
         &\sum_{\substack{1\le s_1\le s-1\\z_1\in\bbZ}}\cha(s-s_1,-z_1)\sum_{\substack{1\le s_2\le s_1-1\\z_2\in\bbZ}}\cha(s-s_2,-z_2)\cha(s_1-s_2,z_1-z_2)^2\\
         =&\sum_{\substack{1\le s^*\le s-1\\z^*\in\bbZ}}\cha(s^*,z^*)^2\sum_{\substack{1\le s_2\le s-s^*-1\\z_2\in\bbZ}}\cha(s-s^*-s_2,-z^*-z_2)\cha(s-s_2,-z_2)\\
         =4&\sum_{\substack{1\le s^*\le s-1\\z^*\in\bbZ}}\cha(s^*,z^*)^2\left(p(s^*,z^*)-p(2s-s^*-2,z^*)\right)
     \end{align*}
    The last equality follows from Lemma \ref{NL.reflection identity}.
    Define
	$$A_t:=\sum_{z\in \Z}p^3(t,z)=O((t+1)^{-1}),~\forall~t\ge 0.$$
	It satisfies the recursive formula
	\begin{equation}\label{E.recursive of At}
		A_{t+1}=\sum_{z\in\Z}\left(\frac{p(t,z+1)+p(t,z-1)}{2}\right)^3=\frac{1}{4}A_t+\frac{3}{4}\sum_{z\in\Z}p^2(t,z+1)p(t,z-1).
	\end{equation}
	Here we use that $p(t,\cdot)$ is an even function and therefore
    \begin{equation}\label{E.symmetry of p}
		\sum_{z\in\Z}p(s,z+1)^2p(s,z-1)=\sum_{z\in\Z}p(s,-z+1)^2p(s,-z-1)=\sum_{z\in\Z}p(s,z-1)^2p(s,z+1).
	\end{equation}
    Consequently,
    \begin{align*}
        \sum_{\substack{1\le s^*\le s-1\\z^*\in\bbZ}}\cha(s^*,z^*)^2p(s^*,z^*)
        =&\sum_{\substack{0\le u\le s-2\\z^*\in\bbZ}}\left[p(u,z^*+1)-p(u,z^*-1)\right]^2\frac{p(u,z^*-1)+p(u,z^*+1)}{2}\\
        =&\sum_{\substack{0\le u\le s-2\\z^*\in\bbZ}}\left[p(u,z^*+1)-p(u,z^*-1)\right]^2p(u,z^*+1)\\
		=&\sum_{\substack{0\le u\le s-2\\z^*\in\bbZ}}p(u,z^*+1)^3-p(u,z^*+1)^2p(u,z^*-1)\\
        =&\frac{4}{3}(A_0-A_{s-1})=\frac{4}{3}+O(s^{-1}),
    \end{align*}
	where the third equality uses \eqref{E.symmetry of p} and the fourth equality uses \eqref{E.recursive of At}. For the second term, Lemmas \ref{NL.upper bound on cha} and \ref{NL.reflection identity} give
    $$\sum_{\substack{1\le s^*\le s-1\\z^*\in\bbZ}}\cha(s^*,z^*)^2p(2s-s^*-2,z^*)=O(s^{-\frac{1}{2}}).$$
	Therefore,
	$$\sum_{\embed\in\embedsetnew{t,x}^{\pi_3}(\tree_3)}\prod_{e\in E(\tree)}\weightnew{\embed}(e)=\sum_{s=0}^{t-1}4\left(\frac{4}{3}+O(s^{-\frac{1}{2}})\right)=\frac{16}{3}t+O(t^{\frac{1}{2}}).$$
    Substituting these expressions into \eqref{E.expectation of X2 midstep} and \eqref{E.expectation of X3 midstep} completes the proof of Theorem \ref{T.expectation of Xl 1}.
\end{proof}
\subsection{Proof of Theorem \ref{T.concentration on Xl 1}}
As in Corollary \ref{C.kterm represented by tree}, one can deduce a representation of $\xterm{l}(t,x)^n$ from Proposition \ref{P.xterm represented by tree}.

\begin{corollary}\label{C.xterm represented by tree}
    For fixed $l,n\ge 1$ and any $(t,x)\in\bbZ_+\times\bbZ$,
    $$\xterm{l}(t,x)^n=\sum_{\tree\in\treefamily_l^n}c(\tree,l,n)\actionnew{t,x}{\tree}.$$
    Here, the constants $c(\tree,l,n)$ are defined in Corollary \ref{C.kterm represented by tree}.
\end{corollary}
Recall that \eqref{E.def of Xl} gives the following explicit expression for $\xterm{l}$ for any $t,x$,

\begin{equation*}
    \xterm{l}(t,x)=\sum_{s=1}^{t}\sum_{z\in\bbZ}p(t-s,x-z)\left[P\left(\sum_{m=1}^{l-1}\kterm{m}(s,z)\right)-P\left(\sum_{m=1}^{l-2}\kterm{m}(s,z)\right)\right],
\end{equation*}
or equivalently,
$$\xterm{l}(t,x)=\sum_{s=1}^{t}\sum_{z\in\bbZ}p(t-s,x-z)\ktermnew{l}(s,z),$$
where $\ktermnew{l}=P\left(\sum_{m=1}^{l-1}\kterm{m}\right)-P\left(\sum_{m=1}^{l-2}\kterm{m}\right)$.

Fix $n$ and a tree $\tree$ whose root has $n$ children $v_1,\ldots,v_n$. If $I$ is a subset of $[n]$, denote by $\sigma^I$ the partition of $L(\tree)$ in which for each $i\in I$, all the leaves attached to $v_i$ or to the descendants of $v_i$ form a part of $\sigma ^I$, while the remaining leaves form another part of $\sigma^I$. Thus, $\setsize{\sigma^I}=\min\{\setsize{I}+1,n\}$. For two partitions $\pi,\pi'$ of $L(\tree)$, their intersection is defined as

$$\pi\cap\pi':=\left\{A\cap B:A\in\pi,B\in\pi',A\cap B\neq\emptyset\right\},$$
which remains a partition of $L(\tree)$. Using Corollary \ref{C.xterm represented by tree} and the same argument as above, we can deduce
\begin{align*}
    &\bbE\left[\left(\xterm{l}(t,x)-\bbE[\xterm{l}(t,x)]\right)^n\right]\\
    =&\sum_{\substack{1\le s_1,\ldots,s_n\le t\\z_1,\ldots,z_n\in\bbZ}}\prod_{i=1}^n p(t-s_i,x-z_i)\cdot \bbE\left[\prod_{i=1}^{n}\left(\ktermnew{l}(s_i,z_i)-\bbE[\ktermnew{l}(s_i,z_i)]\right)\right]\\
    =&\sum_{\substack{1\le s_1,\ldots,s_n\le t\\z_1,\ldots,z_n\in\bbZ}}\prod_{i=1}^n p(t-s_i,x-z_i)\cdot\sum_{I\subseteq [n]}(-1)^{\setsize{I}}\prod_{i\in I}\bbE[\ktermnew{l}(s_i,z_i)]\cdot\bbE\left[\prod_{j\in [n]\setminus I}\ktermnew{l}(s_j,z_j)\right]\\
    =&\sum_{T\in\treefamily_l^n}c(\tree,l,n)N^{-(\frac{1}{4}+\rate)\setsize{L(\tree)}}\sum_{\pi\in\Pi(L(\tree))}\sum_{\embed\in\embedsetnew{t,x}^\pi(\tree)}\prod_{e\in E(\tree)}\weightnew{\embed}(e)\sum_{I\subseteq[n]}(-1)^{\setsize{I}}\moment{\embed}{\pi\cap\sigma^I}
\end{align*}
We say a partition $\pi$ is \emph{nonlocal} if $\pi\cap\sigma^{\{i\}}\neq\pi$ for every $1\le i\le n$, i.e., none of the sets of leaves attached to $v_i$ or to the descendants of $v_i$ is a union of parts of $\pi$. If $\pi$ is not \emph{nonlocal}, so that $\pi\cap\sigma^{\{i\}}=\pi$ for some $i$, then
$$\pi\cap\sigma^{I}=\pi\cap\sigma^{I\cup \{i\}},~\forall~I\subset[n]\setminus\{i\}, $$
and hence, for any $\embed$, we have $$\sum_{I\subseteq[n]}(-1)^{\setsize{I}}\moment{\embed}{\pi\cap\sigma^I}=0.$$
Thus, we can write
\begin{align}
    \nonumber&\bbE\left[\left(\xterm{l}(t,x)-\bbE[\xterm{l}(t,x)]\right)^n\right]\\
    \label{E.restrict to nonlocal partitions}=&\sum_{T\in\treefamily_l^n}c(\tree,l,n)N^{-(\frac{1}{4}+\rate)\setsize{L(\tree)}}\sum_{\substack{\pi\in\Pi(L(\tree))\\\pi~is~nonlocal}}\sum_{\embed\in\embedsetnew{t,x}^\pi(\tree)}\prod_{e\in E(\tree)}\weightnew{\embed}(e)\sum_{I\subseteq[n]}(-1)^{\setsize{I}}\moment{\embed}{\pi\cap\sigma^I}.
\end{align}
If an improper partition $\pi$ is nonlocal, then any partition coarser than $\pi$ remains nonlocal. Note that a bad point of a nonlocal partition cannot be the child of the root. Thus, the obstacles that appeared in the proof of Proposition \ref{P.restrict to proper partitions new} will not occur in the reduction procedure of a nonlocal improper partition.
\begin{proposition}\label{P.restrict to proper partitions new+}
If $\tree\in\treefamily_l^n,l\ge 2$ and $\pi$ is a nonlocal partition of $L(\tree)$, then for any $x$ and $t\le aN$,
    \begin{align}\label{E.restrict to proper partitions new+}
        \nonumber&\absolute{\sum_{\embed\in\embedsetnew{t,x}^{\pi}(\tree)}\prod_{e\in E(\tree)}\weightnew{\embed}(e)\sum_{I\subseteq[n]}(-1)^{\setsize{I}}\moment{\embed}{\pi\cap\sigma^I}}\\
        \loglesssim &\sum_{\substack{\tilde{\pi}\in\Pi(L(\tree))\\ \tilde{\pi}~\text{is proper}}}N^{\frac{1}{4}\sum_{\tilde{\pi}_i\in\tilde{\pi}}(\setsize{\tilde{\pi}_i}-8)^+}\cdot\sum_{\embed\in\embedsetnew{t,x}^{\tilde{\pi}}(\tree)}\prod_{e\in E(\tree)}\absolute{\weightnew{\embed}(e)}.
    \end{align}
\end{proposition}
For a proper partition, we can apply Propositions \ref{P.lifting 1} and \ref{P.lifting 2} again. However, an $\actionnew{t,x}{G}$-analogue of Proposition \ref{P.lifting 3} is needed.
\begin{proposition}\label{P.lifting 3 new}
    If $G$ is an amenable graph with backbone $S_n$, then for any $x$ and $t\le aN$,
    \begin{equation}\label{E.lifting 3 new}
        \actionnew{t,x}{G}\loglesssim N^{n-\frac{1}{2}R(G)}.
    \end{equation}
\end{proposition}
\begin{proof}
    We prove the proposition by induction on $n$. The case $n=1$ follows easily from the fact that $\sum_{x\in \bbZ} p(t,x)=1$ for every $t>0$. Now suppose \eqref{E.lifting 3 new} holds for all $1\le n'\le n-1$. We pick a non-root vertex $v$ with the smallest red degree. By the amenability of $G$, the average red degree is strictly less than $4$. $\reddeg(v)\le 3$. The multiset $N(v)$ is defined to be the multiset of neighbors of $v$.
    \begin{itemize}
        \item \textbf{Case 1} $\reddeg(v)\in\{0,1,2\}$. Denote by $G'$ the graph obtained from $G$ by deleting $v$ and all the edges incident to $v$. Then, \eqref{E.lifting 3 new} follows from the induction hypothesis on $G'$. More precisely, one uses $\sum_{t\le aN}\sum_{x\in \bbZ} p(t,x)=aN$ when $\reddeg(v)=0$, uses $\sum_{x\in \bbZ} p(t,x)=1$ and Lemma \ref{NL.1 red} when $\reddeg(v)=1$, and uses Lemma \ref{NL.1 red} when $\reddeg(v)=2$.
        \item \textbf{Case 2} $\reddeg(v)=3$. Consider a family of subsets of $N(v)$,
        $$\mathcal F:=\{A\subset N(v):there~is~a~vertex~set~S~s.t. \vroot, v\notin S,S\cap N(v)=A~and~R(S)=2|S|-2\}.$$
        As in Proposition \ref{P.property of F}, $\mathcal F$ defined above is closed under union and intersection. Moreover, $N(v)\notin \mathcal F$, which means that at least one vertex in $N(v)$ is not included in any element of $\mathcal F$. Pick such a vertex and call it $u$. Then pick a vertex $w\in N(v)$ different from $u$. Denote by $G’$ the graph obtained from $G$ by deleting $v$ and all the edges incident to $v$ and adding an additional red edge ${u,w}$. Then, $G’$ is amenable and \eqref{E.lifting 3 new} follows from the induction hypothesis on $G’$ and Lemma \ref{NL.lifting 2}.

    \end{itemize} 
\end{proof}
We now prove Theorem \ref{T.concentration on Xl 1}.
\begin{proof}[Proof of Theorem \ref{T.concentration on Xl 1}]
    For a fixed $\tree\in\treefamily_l^n$ and a proper partition $\pi$ Proposition \ref{P.binary tree} yields a constant $C$ and a full binary tree $\tilde{\tree}$ whose root has exactly $n$ children, such that
    $$\absolute{\sum_{\embed\in\embedsetnew{t,x}^\pi(\tree)}\prod_{e\in E(\tree)}\weightnew{\embed}(e)}\lesssim\actionnew{t,x}{\tree^\pi}\lesssim\actionnew{t+C,x}{\tilde\tree^\pi}.$$
    Let $\tree^\circ$ denote the naked part of $\tilde\tree$ and suppose it has $m+n$ edges. Then $\setsize{L(\tree)}=m+2n$. By applying Propositions \ref{P.lifting 1}, \ref{P.lifting 2} and \ref{P.lifting 3 new}, we obtain
    \begin{align*}
        \actionnew{t,x}{\tilde\tree^\pi}&\loglesssim\sum_{G'\in\mathcal G}N^{\frac{1}{2}(m-R(G)+R(G'))}\actionnew{t,x}{G'}\loglesssim \sum_{G'\in\mathcal G}N^{\frac{1}{2}(m-R(G)+R(G'))}\cdot N^{n-\frac{1}{2}R(G')}\\
        &\lesssim N^{\frac{1}{2}(m+2n-R(G))}\lesssim N^{\frac{1}{4}\setsize{L(\tree)}-\frac{1}{4}\sum_{\pi_i\in\pi}(\setsize{\pi_i}-4)^+}.
    \end{align*}
    Here, in the last inequality, we use $R(G)\ge\setsize{L(\tree)}/2+\sum_{\pi_i\in\pi}(\setsize{\pi_i}-4)^+/2$ and $\setsize{L(\tree)}=m+2n$. Using the above inequalities and \eqref{E.restrict to proper partitions new+} in \eqref{E.restrict to nonlocal partitions}, we may conclude
    \begin{align*}
        &\absolute{\bbE\left[\left(\xterm{l}(t,x)-\bbE[\xterm{l}(t,x)]\right)^n\right]}\\
        \loglesssim&\sum_{\tree\in\treefamily_l^n}N^{-(\frac{1}{4}+\rate)\setsize{L(\tree)}}\sum_{\substack{\pi\in\Pi(L(\tree))\\\pi~is~proper}}N^{\frac{1}{4}\sum_{\pi_i\in\pi}(\setsize{\pi_i}-8)^+}\sum_{\embed\in\embedsetnew{t,x}^\pi(\tree)}\prod_{e\in E(\tree)}\absolute{\weightnew{\embed}(e)}\\
        \loglesssim&\sum_{\tree\in\treefamily_l^n}N^{-(\frac{1}{4}+\rate)\setsize{L(\tree)}}\sum_{\substack{\pi\in\Pi(L(\tree))\\\pi~is~proper}}N^{\frac{1}{4}\sum_{\pi_i\in\pi}(\setsize{\pi_i}-8)^+}\cdot N^{\frac{1}{4}\setsize{L(\tree)}-\frac{1}{4}\sum_{\pi_i\in\pi}(\setsize{\pi_i}-4)^+}\\
        \lesssim& N^{-\rate\setsize{L(\tree)}}\lesssim N^{-ln\rate}.
    \end{align*}
    The last inequality uses the fact that any tree in $\treefamily_l^n$ has at least $ln$ leaves. This completes the proof.
\end{proof}

\section{Proof of Theorem \ref{T.main 2}}\label{S.Proof of T.main 2}

In this section, we prove Theorem \ref{T.main 2}, which concerns the regime $\alpha\ge 4$ and $M\ge 20/\rate$. In this case, truncation of $\noise$ is unnecessary, and the proof follows the same strategy  as before with slightly different definitions of $\xterm{l}$ and $\error{l}$, as well as different bounds.

\subsection{Renormalization}
Suppose $\driving$ has the following expansion near $0$:
$$\driving(x)=\frac{\beta}{8}x^2+R(x)x^4,$$
where $R(x)$ is a continuous function in a neighborhood of $0$, satisfying $R(0)=\driving^{(4)}(0)/24$. The renormalization terms are defined as follows:

\begin{equation*}
    \xterm{1}(t,x)=\average\xterm{1}(t,x)+N^{-\frac{1}{4}-\rate}\noise(t,x),
\end{equation*} The initial condition is $\xterm{1}(0,x)=0$ for all $x\in\bbZ$. Let $\kterm{1}(t,x)=\difference\xterm{1}(t,x)$. Next, define

\begin{equation*}
    \xterm{2}(t,x)=\average\xterm{2}(t,x)+\frac{\beta}{8}\kterm{1}(t,x)^2
\end{equation*} The initial condition is $\xterm{2}(0,x)=0$ for all $x\in\bbZ$. Let $\kterm{2}(t,x)=\difference\xterm{2}(t,x)$. 
For $l\ge 3$, suppose $\xterm{m}(t,x)$ and $\kterm{m}(t,x)$ are defined for all $1\le m\le l-1$ and $(t,x)\in\bbZ_+\times\bbZ$. Define
$$\xterm{l}(t,x)=\average\xterm{l}(t,x)+\frac{\beta}{4}\kterm{1}(t,x)\kterm{l-1}(t,x)$$
The initial condition is $\xterm{l}(0,x)=0$ for all $x\in\bbZ$. Let $\kterm{l}(t,x)=\difference\xterm{l}(t,x)$ and
$$\error{l}(t,x)=\function(t,x)-\sum_{m=1}^{l}\xterm{m}(t,x).$$
We may derive the following upper bound on moments of $\kterm{l}$.
\begin{theorem}\label{T.moments upper bound on K 2}
     For any fixed even $2\le n\le M$,
    $$\bbE\left[\left(\kterm{1}(t,x)\right)^n\right]\lesssim N^{-(\frac{1}{4}+\rate)n}.$$
\end{theorem}

\begin{theorem}\label{T.moments upper bound on K 3}
     For any fixed $l\ge 2$ and even $n\ge 2$ satisfying  $ln\le M$,
    $$\bbE\left[\left(\kterm{l}(t,x)\right)^n\right]\loglesssim N^{-(\frac{1}{2}+l\rate)n}.$$
\end{theorem}
As in Corollary \ref{C.upper bound on K 1}, we obtain the following by the same proof.
\begin{corollary}\label{C.upper bound on K 2}
    For fixed $a,b>0$,
    $$\bbP\left(\absolute{\kterm{1}(t,x)}\le N^{-\frac{1}{4}-0.6\rate},\forall~ (t,x)\in\trapezoid\right)=1-o(1).$$
    Moreover, for any fixed $2\le l\le \max\{2, \lceil4/\rate\rceil\}$,
    $$\bbP\left(\absolute{\kterm{l}(t,x)}\le N^{-\frac{1}{2}-0.6l\rate},\forall~ (t,x)\in\trapezoid\right)=1-o(1).$$
\end{corollary}
\begin{proof}
    
    The proof is identical to that of Corollary \ref{C.upper bound on K 1}. Apply Markov's inequality and Theorems \ref{T.moments upper bound on K 2} and \ref{T.moments upper bound on K 3} at each $(t,x)\in\trapezoid$, followed by a union bound.
    
\end{proof}
Thus, we can expand $\function$ into a sum of $\xterm{l}$ as we did in Theorem \ref{T.expansion 1}. Since the proof follows the same inductive strategy, we omit it for brevity.
\begin{theorem}\label{T.expansion 2}
    Let $L=\max\{2, \lceil4/\rate\rceil\}$,  and let $a,b>0$ be fixed. Then there is an event $\Omega_e$ with $\bbP(\Omega_e)=1-o(1)$, such that on $\Omega_e$, for all $(t,x)\in\trapezoid$,
    \begin{equation}\label{E.expansion 3}
        \error{L}(t,x)=o(N^{-\rate}).
    \end{equation}
    Consequently,  we have the following expansion, which holds with high probability:
    \begin{equation}\label{E.expansion 4}
        \function(t,x)=\sum_{l=1}^{L}\xterm{l}(t,x)+o(N^{-\rate}),~\forall~(t,x)\in\trapezoid
    \end{equation}
\end{theorem}

Although the upper bound on $\kterm{l}$ is different from that in the previous case, the behavior of $\xterm{l}$ is the same. We have the following two theorems, which are analogous to Theorems \ref{T.expectation of Xl 1} and \ref{T.concentration on Xl 1}.
\begin{theorem}\label{T.expectation of Xl 2}
    Let $\mu_k=\bbE[\noise^k]$ and fix $a>0$. For any $x$ and $t\le aN$,
    \begin{align*}
        \bbE[\xterm{2}(t,x)]&=\frac{\beta}{2}\mu_2N^{-\frac{1}{2}-2\rate}t+o(N^{-\rate}),\ \   \bbE[\xterm{3}(t,x)]=\frac{\beta^2}{6}\mu_3N^{-\frac{3}{4}-3\rate}t+o(N^{-\rate}),\\
        \bbE[\xterm{l}(t,x)]&=o(N^{-\rate}),\forall~l\ge 4.
    \end{align*}
\end{theorem}
\begin{theorem}\label{T.concentration on Xl 2}
      For any fixed $l\ge 2$ and even $n\ge 2$ satisfying $ln\le M$, the following estimate holds for any $x$ and $t\le aN$,
    $$\bbE \left[\left(\xterm{l}(t,x)-\bbE[\xterm{l}(t,x)]\right)^n\right]\loglesssim N^{-ln\rate}.$$
\end{theorem}
The proofs of these two theorems are the same as those of Theorems \ref{T.expectation of Xl 1} and \ref{T.concentration on Xl 1}. The present assumptions are stronger than those of Theorems \ref{T.expectation of Xl 1} and \ref{T.concentration on Xl 1}. Moreover, $\xterm{l}$ are defined with fewer terms, and the required moments are available for the fixed orders used here.

The remainder of the proof of Theorem \ref{T.main 2} is identical to the proof of Theorem \ref{T.main}, so we omit the details.
\subsection{Proof of Theorems \ref{T.moments upper bound on K 2} and \ref{T.moments upper bound on K 3}}
In this subsection, we will prove Theorems \ref{T.moments upper bound on K 2} and \ref{T.moments upper bound on K 3} and hence complete the proof of Theorem \ref{T.main 2}.
\begin{proof}[Proof of Theorem \ref{T.moments upper bound on K 2}]
    Recall that $\kterm{1}$ has the explicit formula
    $$\kterm{1}(t,x)=\sum_{s=1}^{t-1}\sum_{z\in\bbZ}\cha(t-s,x-z)N^{-\frac{1}{4}-\rate}\noise(s,z).$$
    Thus 
    \begin{align*}
        \bbE\left[\left(\kterm{1}(t,x)\right)^n\right]&=\bbE\left[\left(\sum_{s=1}^{t-1}\sum_{z\in\bbZ}\cha(t-s,x-z)N^{-\frac{1}{4}-\rate}\noise(s,z)\right)^n\right]\\
        &=N^{-(\frac{1}{4}+\rate)n}\sum_{s_1,\dots,s_n=1}^{t-1}\sum_{z_1,\dots,z_n\in\bbZ}\prod_{i=1}^{n}\cha(t-s_i,x-z_i)\bbE\left[\prod_{i=1}^{n}\noise(s_i,z_i)\right]\\
        &\lesssim N^{-(\frac{1}{4}+\rate)n}\sum_{k\ge 1}\sum_{\substack{a_1,\dots,a_k\ge 2\\a_1+\cdots+a_k=n}}\sum_{\substack{1\le s_1,\dots,s_k\le t-1\\z_1,\dots,z_k\in\bbZ\\(s_j,z_j)~are~distinct}}\prod_{j=1}^{k}\cha(t-s_j,x-z_j)^{a_j}\bbE\left[\noise(s_j,z_j)^{a_j}\right], 
    \end{align*}
    Only $a_1,\dots,a_k\ge 2$ contribute to the sum because $\bbE[\noise(s,0)]=0$ for any $s$. For any fixed $a_1,\dots,a_k$,
    each $\bbE\left[\noise(s_j,z_j)^{a_j}\right]$ is bounded by a finite constant since $n\le M$. Then 
    $$\sum_{\substack{1\le s_1,\dots,s_k\le t-1\\z_1,\dots,z_k\in\bbZ\\(s_j,z_j)~are~distinct}}\prod_{j=1}^{k}\absolute{\cha(t-s_j,x-z_j)}^{a_j}\lesssim\sum_{\substack{1\le s_1,\dots,s_k\le t-1\\z_1,\dots,z_k\in\bbZ}}\prod_{j=1}^{k}\absolute{\cha(t-s_j,x-z_j)}^{2}\le 4^k.$$
    Thus, there is a constant $C=C_n$ such that
    $$\bbE\left[\left(\kterm{1}(t,x)\right)^n\right]\le C_nN^{-(\frac{1}{4}+\rate)n}$$
    holds for any $(t,x)\in\bbZ_+\times\bbZ$.
\end{proof}
To prove Theorem \ref{T.moments upper bound on K 3}, we apply the graph lifting technique developed in Section \ref{S.Graph Lifting}. We start with the tree family
$$\treefamily_1:=\{\one\}.$$ 

For $l\ge 2$, $\treefamily_l$ is defined inductively:

$$\treefamily_l:={\extension{\tree_1\oneptunion\tree_2}:\tree_1\in\treefamily_{l-1},\tree_2\in\treefamily_{1}}.$$
Note that in this case, each tree family $\treefamily_l$ contains only one tree, denoted by $\tree_l$. The first few $\tree_l$ are shown in the following diagram. As in Proposition \ref{P.kterm represented by tree}, induction gives
$$\kterm{l}(t,x)=\frac{\beta^{l-1}}{2^{2l-1}}\action{t,x}{\tree_l}, \forall~l\ge 2.$$
Here, the definition of action $\action{t,x}{\cdot}$ is the same as that in \eqref{E.def of action of T}, but $\truncatednoise$ is replaced by $\noise$. If we denote by $\tree_l^n$ the one-point union of $n$ copies of $\tree_l$, then we have
$$\left(\kterm{l}(t,x)\right)^n=\left(\frac{\beta^{l-1}}{2^{2l-1}}\right)^n\action{t,x}{\tree_l^n}, \forall~l\ge 2.$$
Note that $\tree_l$ has exactly $l$ leaves, so $\tree_l^n$ has $ln$ leaves and we can write $\bbE[\action{t,x}{\tree_l^n}]$ as
$$\bbE[\action{t,x}{\tree_l^n}]=N^{-(\frac{1}{4}+\rate)ln}\sum_{\pi\in\Pi(L(\tree_l^n))}\sum_{\embed\in\embedset{t,x}^{\pi}(\tree_l^n)}\moment{\embed}{\pi}\prod_{e\in E(\tree_l^n)}\weight{\embed}(e),$$
where $$\moment{\embed}{\pi}:=\prod_{\pi_i\in\pi}\mu_{\setsize{\pi_i}}(\embed^{(1)}(\pi_i)) $$
is uniformly bounded since the total number of leaves $ln$ is bounded by $M$. The argument in Proposition \ref{P.restrict to proper partitions} can be applied in the same way, allowing us to consider only those proper partitions $ \pi\in\Pi (L (\tree_l^n)) $.
It therefore suffices to show 
$$\action{t,x}{\left(\tree_l^n\right)^\pi}\loglesssim N^{\frac{1}{4}(l-2)n}.$$
For simplicity, we shall write $\tree$ in place of $\tree_l^n$ from now on. Let $\tree^\circ$ denote the naked part of $\tree$. Figure \ref{F.T54-naked} provides an example of $T_5^4$ and its naked part.

\begin{figure}[htbp]
\centering
\begin{tikzpicture}[
    x=0.74cm, y=0.74cm,
    vertex/.style={circle, fill=black, inner sep=1.75pt},
    bedge/.style={draw=black,line width=0.95pt},
    lab/.style={font=\small}
]

% =========================================================
% Left: T_5^4
% =========================================================
\begin{scope}[xshift=0cm,yshift=0cm]

% root
\node[vertex] (r) at (2.35,5.65) {};

% four main columns
\node[vertex] (a1) at (0.70,4.25) {};
\node[vertex] (a2) at (0.70,3.15) {};
\node[vertex] (a3) at (0.70,2.05) {};
\node[vertex] (a4) at (0.70,0.95) {};
\node[vertex] (a5) at (0.70,-0.15) {};

\node[vertex] (b1) at (1.75,4.25) {};
\node[vertex] (b2) at (1.75,3.15) {};
\node[vertex] (b3) at (1.75,2.05) {};
\node[vertex] (b4) at (1.75,0.95) {};
\node[vertex] (b5) at (1.75,-0.15) {};

\node[vertex] (c1) at (2.80,4.25) {};
\node[vertex] (c2) at (2.80,3.15) {};
\node[vertex] (c3) at (2.80,2.05) {};
\node[vertex] (c4) at (2.80,0.95) {};
\node[vertex] (c5) at (2.80,-0.15) {};

\node[vertex] (d1) at (3.85,4.25) {};
\node[vertex] (d2) at (3.85,3.15) {};
\node[vertex] (d3) at (3.85,2.05) {};
\node[vertex] (d4) at (3.85,0.95) {};
\node[vertex] (d5) at (3.85,-0.15) {};

% diagonal leaves
\node[vertex] (la1) at (-0.25,3.45) {};
\node[vertex] (la2) at (-0.25,2.35) {};
\node[vertex] (la3) at (-0.25,1.25) {};
\node[vertex] (la4) at (-0.25,0.15) {};

\node[vertex] (lb1) at (1.05,3.45) {};
\node[vertex] (lb2) at (1.05,2.35) {};
\node[vertex] (lb3) at (1.05,1.25) {};
\node[vertex] (lb4) at (1.05,0.15) {};

\node[vertex] (lc1) at (2.10,3.45) {};
\node[vertex] (lc2) at (2.10,2.35) {};
\node[vertex] (lc3) at (2.10,1.25) {};
\node[vertex] (lc4) at (2.10,0.15) {};

\node[vertex] (ld1) at (3.15,3.45) {};
\node[vertex] (ld2) at (3.15,2.35) {};
\node[vertex] (ld3) at (3.15,1.25) {};
\node[vertex] (ld4) at (3.15,0.15) {};

% edges from root
\draw[bedge] (r) -- (a1);
\draw[bedge] (r) -- (b1);
\draw[bedge] (r) -- (c1);
\draw[bedge] (r) -- (d1);

% vertical chains
\draw[bedge] (a1) -- (a2) -- (a3) -- (a4) -- (a5);
\draw[bedge] (b1) -- (b2) -- (b3) -- (b4) -- (b5);
\draw[bedge] (c1) -- (c2) -- (c3) -- (c4) -- (c5);
\draw[bedge] (d1) -- (d2) -- (d3) -- (d4) -- (d5);

% diagonal leaf edges
\draw[bedge] (a1) -- (la1);
\draw[bedge] (a2) -- (la2);
\draw[bedge] (a3) -- (la3);
\draw[bedge] (a4) -- (la4);

\draw[bedge] (b1) -- (lb1);
\draw[bedge] (b2) -- (lb2);
\draw[bedge] (b3) -- (lb3);
\draw[bedge] (b4) -- (lb4);

\draw[bedge] (c1) -- (lc1);
\draw[bedge] (c2) -- (lc2);
\draw[bedge] (c3) -- (lc3);
\draw[bedge] (c4) -- (lc4);

\draw[bedge] (d1) -- (ld1);
\draw[bedge] (d2) -- (ld2);
\draw[bedge] (d3) -- (ld3);
\draw[bedge] (d4) -- (ld4);

\end{scope}

% =========================================================
% Right: naked part
% =========================================================
\begin{scope}[xshift=8.25cm,yshift=0cm]

% root
\node[vertex] (R) at (2.35,5.65) {};

% four vertical columns
\node[vertex] (A1) at (0.70,4.25) {};
\node[vertex] (A2) at (0.70,3.10) {};
\node[vertex] (A3) at (0.70,1.95) {};
\node[vertex] (A4) at (0.70,0.80) {};

\node[vertex] (B1) at (1.75,4.25) {};
\node[vertex] (B2) at (1.75,3.10) {};
\node[vertex] (B3) at (1.75,1.95) {};
\node[vertex] (B4) at (1.75,0.80) {};

\node[vertex] (C1) at (2.80,4.25) {};
\node[vertex] (C2) at (2.80,3.10) {};
\node[vertex] (C3) at (2.80,1.95) {};
\node[vertex] (C4) at (2.80,0.80) {};

\node[vertex] (D1) at (3.85,4.25) {};
\node[vertex] (D2) at (3.85,3.10) {};
\node[vertex] (D3) at (3.85,1.95) {};
\node[vertex] (D4) at (3.85,0.80) {};

% edges
\draw[bedge] (R) -- (A1);
\draw[bedge] (R) -- (B1);
\draw[bedge] (R) -- (C1);
\draw[bedge] (R) -- (D1);

\draw[bedge] (A1) -- (A2) -- (A3) -- (A4);
\draw[bedge] (B1) -- (B2) -- (B3) -- (B4);
\draw[bedge] (C1) -- (C2) -- (C3) -- (C4);
\draw[bedge] (D1) -- (D2) -- (D3) -- (D4);

\end{scope}

\end{tikzpicture}
\caption{$T_5^4$ and its naked part.}
\label{F.T54-naked}
\end{figure}
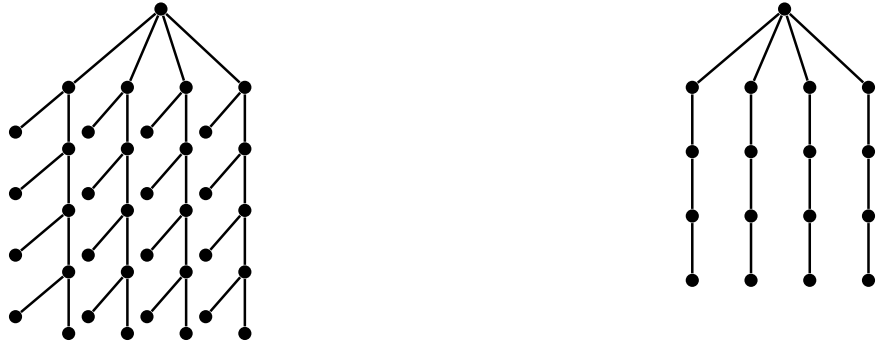

Suppose $v_1,\dots,v_n$ are the leaves of $\tree^\circ$. 
We say that a graph $G$ is \emph{good} if its backbone lies between $\tree^\circ$ and $(\tree^\circ)^\circ$ (in the subgraph sense), and satisfies the following three properties:

\RomanNumeralCaps{1}. Every non-root vertex is incident to at least one red edge.

\RomanNumeralCaps{2}. Every vertex in $\{v_1,\dots,v_n\}$ is incident to at least two red edges.

\RomanNumeralCaps{3}. For every vertex $v_j\in\{v_1,\dots,v_n\}$, it is not the case that the only red edges incident to $v_j$ or its parent are the two red edges between them.

We stress that all graphs under consideration are loopless, meaning that no edge is incident twice to the same vertex. In the remainder of the proof, whenever a loop appears, it is understood to be deleted automatically. Furthermore, whenever a red edge has multiplicity at least three, we remove extra parallel copies so that its multiplicity becomes exactly two.

The next proposition is an analogue of the leaf lifting step in Section \ref{Subs.Leaf lifting}.
\begin{proposition}[\textbf{Leaf lifting}]\label{P.Lifting 1 in Case 2}
    If $\tree=\tree_l^n$ for some $\l\ge 2,n\ge 1$ and $\pi$ is a proper partition of its leaves, then there exists a family of good graphs $\{G_\lambda\}_{\lambda\in\Lambda}$, all with backbone $\tree^\circ$, such that for any $x$ and $t\le aN$,
    \begin{equation}\label{E.Lifting 1 in Case 2}
        \action{t,x}{\tree^\pi}\loglesssim\sum_{\lambda\in\Lambda}\action{t,x}{G_\lambda}.
    \end{equation}
    Here $\Lambda$ is a finite index set.
\end{proposition}
\begin{proof}
 Assume that the parts of $\pi$ are $\pi_1,\ldots,\pi_r$. We will construct a sequence of two-colored backbone graph families $\{G_{\lambda_0}\}_{\lambda_0\in\Lambda_0},\dots,\{G_{\lambda_r}\}_{\lambda_r\in\Lambda_r}$ that interpolate between $\tree^\pi$ and our goal $\{G_\lambda\}_{\lambda\in\Lambda}$. To begin with, we let $\{G_{\lambda_0}\}_{\lambda_0\in\Lambda_0}=\{\tree^\pi\}$. This sequence satisfies, for any $0\le i\le r-1$ and any graph $G_{\lambda_i}, \lambda_i\in\Lambda_i$, 
    all the vertices in $G_{\lambda_i}$ are vertices of $\tree^\circ$ and $\pi_{i+1},\dots,\pi_r$. Moreover, for any $x$ and $t\le aN$,
    \begin{equation}\label{E.Lifting 1 in Sec 5}
        \action{t,x}{G_{\lambda_i}}\loglesssim\sum_{\lambda_{i+1}\in\Lambda_{i+1}}\action{t,x}{G_{\lambda_{i+1}}}.
    \end{equation}
    We construct these families inductively. Suppose we have constructed $\{G_{\lambda_i}\}_{\lambda_i\in\Lambda_i}$, $i\ge 0$. We pick any graph $G$ in this family and consider the vertices of $\pi_{i+1}$ in the original tree $\tree$.

    \begin{itemize}
        \item \textbf{Case 1:} If all the vertices in $\pi_{i+1}$ have distinct parents, suppose they are $u_1,\dots,u_m$.
        
        We delete $\pi_{i+1}$ and add red edges $\{u_j,u_{j+1}\}, 1\le j\le m-1$. We put the resulting graph into $\{G_{\lambda_{i+1}}\}_{\lambda_{i+1}\in\Lambda_{i+1}}$. Then, each $u_j,1\le j\le m$ is incident to at least one red edge, and inequality \eqref{E.Lifting 1 in Sec 5} follows from Lemma \ref{NL.lifting 1-1} and \ref{NL.lifting 1-2}.
        
        \item \textbf{Case 2:} If there exist two vertices in $\pi_{i+1}$ that have the same parent, 
        denote their parent by $u_1$; then $u_1$ has two children that are leaves in $\tree$; hence $u_1$ is a leaf of $\tree^\circ$. Suppose all the parents of the remaining vertices in $\pi_{i+1}$ are $u_2,\dots,u_m$. Since $\pi$ is proper, $\pi_{i+1}$ cannot be exactly the two children of $u_1$, $m\ge 2$. 
        For any permutation $\sigma$ of $\{1,2,\dots,m\}$, we delete $\pi_{i+1}$ and add double red edges $\{u_{\sigma(j)},u_{\sigma(j+1)}\}, 1\le j\le m-1$ (every red edge $\{u_{\sigma(j)},u_{\sigma(j+1)}\}$ is added twice). We put all the resulting graphs into $\{G_{\lambda_{i+1}}\}_{\lambda_{i+1}\in\Lambda_{i+1}}$. Then, each $u_j,1\le j\le m$ is incident to at least two red edges, and inequality \eqref{E.Lifting 1 in Sec 5} follows from Lemma \ref{NL.lifting 1-4}.
    \end{itemize}

    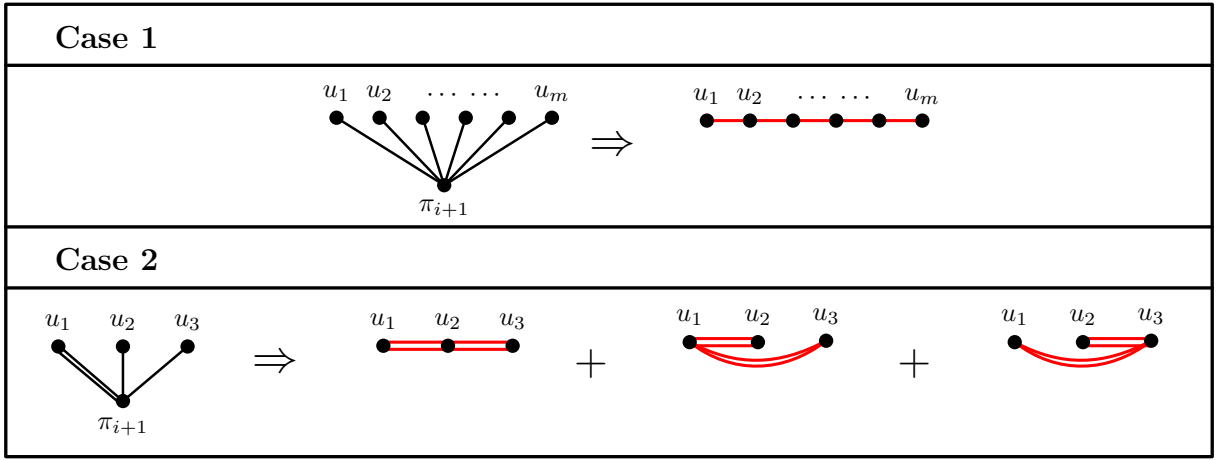
\begin{figure}[htbp]
\centering
\begin{tikzpicture}[
    x=0.95cm,y=0.95cm,
    line cap=round,line join=round,
    vtx/.style={circle,fill=black,inner sep=1.9pt},
    bedge/.style={draw=black,line width=0.95pt},
    redge/.style={draw=red,line width=1.15pt},
    casehead/.style={font=\large\bfseries},
    symlab/.style={font=\LARGE},
    ptlab/.style={font=\normalsize}
]

% =========================================================
% outer frame
% =========================================================
\draw[line width=1.25pt] (0,0.7) rectangle (16.8,7);

% horizontal separators
\draw[line width=1.25pt] (0,6.15) -- (16.8,6.15);
\draw[line width=1.25pt] (0,3.9) -- (16.8,3.9);
\draw[line width=1.25pt] (0,3.05) -- (16.8,3.05);

% headers
\node[anchor=west,casehead] at (0.55,6.55) {Case 1};
\node[anchor=west,casehead] at (0.55,3.45) {Case 2};

% =========================================================
% Case 1 : left graph
% =========================================================
\begin{scope}[xshift=5.80cm,yshift=4.75cm]
    \coordinate (t1) at (-1.50, 0.42);
    \coordinate (t2) at (-0.90, 0.42);
    \coordinate (t3) at (-0.30, 0.42);
    \coordinate (t4) at ( 0.30, 0.42);
    \coordinate (t5) at ( 0.90, 0.42);
    \coordinate (t6) at ( 1.50, 0.42);
    \coordinate (b)  at ( 0.00,-0.52);

    % fan edges
    \draw[bedge] (b)--(t1);
    \draw[bedge] (b)--(t2);
    \draw[bedge] (b)--(t3);
    \draw[bedge] (b)--(t4);
    \draw[bedge] (b)--(t5);
    \draw[bedge] (b)--(t6);

    \foreach \p in {t1,t2,t3,t4,t5,t6,b}
        \node[vtx] at (\p) {};

    % point labels
    \node[ptlab,above=0.1cm] at (t1) {$u_1$};
    \node[ptlab,above=0.1cm] at (t2) {$u_2$};
    \node[ptlab,above=0.1cm] at ($(t3)!0.5!(t4)$) {$\cdots$};
    \node[ptlab,above=0.1cm] at ($(t4)!0.5!(t5)$) {$\cdots$};
    \node[ptlab,above=0.1cm] at (t6) {$u_m$};
    \node[ptlab,below=0.07cm] at (b) {$\pi_{i+1}$};
\end{scope}

% implication arrow
\node[symlab] at (8.45,5.02) {$\Rightarrow$};

% =========================================================
% Case 1 : right graph
% =========================================================
\begin{scope}[xshift=10.70cm,yshift=4.75cm]
    \coordinate (a1) at (-1.50, 0.38);
    \coordinate (a2) at (-0.90, 0.38);
    \coordinate (a3) at (-0.30, 0.38);
    \coordinate (a4) at ( 0.30, 0.38);
    \coordinate (a5) at ( 0.90, 0.38);
    \coordinate (a6) at ( 1.50, 0.38);

    \draw[redge] (a1)--(a2)--(a3)--(a4)--(a5)--(a6);

    \foreach \p in {a1,a2,a3,a4,a5,a6}
        \node[vtx] at (\p) {};

    % point labels
    \node[ptlab,above=0.1cm] at (a1) {$u_1$};
    \node[ptlab,above=0.1cm] at (a2) {$u_2$};
    \node[ptlab,above=0.1cm] at ($(a3)!0.5!(a4)$) {$\cdots$};
    \node[ptlab,above=0.1cm] at ($(a4)!0.5!(a5)$) {$\cdots$};
    \node[ptlab,above=0.1cm] at (a6) {$u_m$};
\end{scope}

% =========================================================
% Case 2 : initial graph
% =========================================================
\begin{scope}[xshift=1.55cm,yshift=1.8cm]
    \coordinate (l) at (-0.90, 0.34);
    \coordinate (m) at ( 0.00, 0.34);
    \coordinate (r) at ( 0.90, 0.34);
    \coordinate (v) at ( 0.00,-0.42);

    % edge
    \draw[bedge] (m)--(v);
    \draw[bedge] (r)--(v);

    % double edge
    \draw[bedge] (-0.94,0.3)--(-0.08,-0.42);
    \draw[bedge] (-0.86,0.32)--( 0,-0.42);

    \foreach \p in {l,m,r,v}
        \node[vtx] at (\p) {};

    % point labels
    \node[ptlab,above=0.1cm] at (l) {$u_1$};
    \node[ptlab,above=0.1cm] at (m) {$u_2$};
    \node[ptlab,above=0.1cm] at (r) {$u_3$};
    \node[ptlab,below=0.1cm] at (v) {$\pi_{i+1}$};
\end{scope}

% implication arrow
\node[symlab] at (3.75,2) {$\Rightarrow$};

% =========================================================
% Case 2 : first resulting graph
% =========================================================
\begin{scope}[xshift=5.85cm,yshift=1.8cm]
    \coordinate (l) at (-0.90, 0.35);
    \coordinate (m) at ( 0.00, 0.35);
    \coordinate (r) at ( 0.90, 0.35);

    % double straight red edges
    \draw[redge] (-0.90,0.4)--(0.00,0.4)--(0.90,0.4);
    \draw[redge] (-0.90,0.3)--(0.00,0.3)--(0.90,0.3);

    \foreach \p in {l,m,r}
        \node[vtx] at (\p) {};

    % point labels
    \node[ptlab,above=0.1cm] at (l) {$u_1$};
    \node[ptlab,above=0.1cm] at (m) {$u_2$};
    \node[ptlab,above=0.1cm] at (r) {$u_3$};
\end{scope}

\node[symlab] at (8.15,2) {$+$};

 % =========================================================
% Case 2 : second resulting graph
% =========================================================
\begin{scope}[xshift=9.95cm,yshift=1.8cm]
    \coordinate (l) at (-0.95, 0.40);
    \coordinate (m) at ( 0.00, 0.40);
    \coordinate (r) at ( 0.95, 0.42);

    % short double top edges
    \draw[redge] (-0.95,0.45)--(0.00,0.45);
    \draw[redge] (-0.95,0.35)--(0.00,0.35);

    % two lower curved edges
    \draw[redge] (l) to[out=-28,in=208] (r);
    \draw[redge] (l) to[out=-38,in=218] (r);

    \foreach \p in {l,m,r}
        \node[vtx] at (\p) {};

    % point labels
    \node[ptlab,above=0.1cm] at (l) {$u_1$};
    \node[ptlab,above=0.1cm] at (m) {$u_2$};
    \node[ptlab,above=0.1cm] at (r) {$u_3$};
\end{scope}

\node[symlab] at (12.65,2) {$+$};

% =========================================================
% Case 2 : third resulting graph
% =========================================================
\begin{scope}[xshift=14.25cm,yshift=1.8cm]
    \coordinate (l) at (-0.95, 0.40);
    \coordinate (m) at ( 0.00, 0.40);
    \coordinate (r) at ( 0.95, 0.42);

    % short double top edges
    \draw[redge] (0.95,0.45)--(0.00,0.45);
    \draw[redge] (0.95,0.35)--(0.00,0.35);

    % two lower curved edges
    \draw[redge] (l) to[out=-28,in=208] (r);
    \draw[redge] (l) to[out=-38,in=218] (r);

    \foreach \p in {l,m,r}
        \node[vtx] at (\p) {};

    % point labels
    \node[ptlab,above=0.1cm] at (l) {$u_1$};
    \node[ptlab,above=0.1cm] at (m) {$u_2$};
    \node[ptlab,above=0.1cm] at (r) {$u_3$};
\end{scope}

\end{tikzpicture}

\caption{Two cases of Leaf lifting.}
\label{F.lifting 1 in case 2}
\end{figure}
 
    The procedures for the above cases are illustrated in Figure \ref{F.lifting 1 in case 2}. When $G$ ranges over all graphs in $\{G_{\lambda_i}\}_{\lambda_i\in\Lambda_i}$, the resulting $\{G_{\lambda_{i+1}}\}_{\lambda_{i+1}}$ satisfies the inductive hypothesis. Finally, we obtain $\{G_{\lambda_r}\}_{\lambda_r\in\Lambda_r}$ and take $\{G_{\lambda}\}_{\lambda\in\Lambda}$ to be the desired family. It is straightforward to check that any graph in this family satisfies properties \RomanNumeralCaps{1}, \RomanNumeralCaps{2}, and \RomanNumeralCaps{3}.

\end{proof}
We next develop the backbone-pruning procedure. 

\begin{proposition}[\textbf{Backbone pruning}]\label{P.Lifting 2 in Case 2}
    If $G$ is a good graph, then there exists a family of good graphs $\{G_\theta\}_{\theta\in\Theta}$ all of which have $(\tree^\circ)^\circ$ as backbone, such that for any $x$ and $t\le aN$,
    \begin{equation}\label{E.Lifting 2 in Case 2}
        \action{t,x}{G}\loglesssim\sum_{\theta\in\Theta}\action{t,x}{G_\theta}.
    \end{equation}
\end{proposition}
Note that, since $G_\theta$ contains no vertex in ${v_1,\dots,v_n}$, the “goodness” of $G_\theta$ only means that every non-root vertex is incident to at least one red edge.
\begin{proof}
    It suffices to prove the following. If $G$ is a good graph and has at least one vertex in $\{v_1,\dots,v_n\}$, then there exists a family of good graphs $\{G_\lambda\}_{\lambda\in\Lambda}$, each having fewer vertices in $\{v_1,\dots,v_n\}$, such that for any $x$ and $t\le aN$,
    \begin{equation}\label{E. Lifting 2 in Sec 5}
        \action{t,x}{G}\loglesssim\sum_{\lambda\in\Lambda}\action{t,x}{G_{\lambda}}.
    \end{equation}
    If the preceding statement holds, then we can apply it to $G$ finitely many times until we obtain a family of good graphs that satisfy \eqref{E.Lifting 2 in Case 2} and all have $(\tree^\circ)^\circ$ as backbone. It therefore suffices to prove this statement. 

    Suppose the vertex set of $G$ is the vertex set of $(\tree^\circ)^\circ$ together with $\{v_1,\dots,v_i\}$, and $v_i$ is incident to the fewest red edges among $\{v_1,\dots,v_i\}$. Denote by $\reddeg(v)$ the number of red edges incident to $v$. The procedures for the following cases are illustrated in Figure \ref{F.lifting 2 in case 2}.   

    \textbf{Case 1:} If $\reddeg(v_i)=2$ and the two red edges incident to $v_i$ connect to different vertices $u,w$, we delete $v_i$ and add a red edge $\{u,w\}$ to obtain a new graph $G_\lambda$. Then, \eqref{E. Lifting 2 in Sec 5} follows from Lemma \ref{NL.sum of cha} and \ref{NL.lifting 2}. One can check directly that \RomanNumeralCaps{1} and \RomanNumeralCaps{2} are satisfied. If \RomanNumeralCaps{3} holds for this $G_\lambda$, then we are done. Otherwise, $G_\lambda$ violates \RomanNumeralCaps{3}, and then $u$ and $w$ must be $v_j\in\{v_1,\dots,v_{i-1}\}$ and its parent. Also, in $G$, the only red edges incident to $v_i$ or $v_j$ (without loss of generality, $u=v_j$) are $\{v_i,v_j\},\{v_i,w\}$ and $\{v_j,w\}$. Then, we can delete $v_j$ from $G$ and add a red edge $\{v_i,w\}$ to get another $G_\lambda$. Under this operation, $\{v_i,w\}$ is a double red edge in $G_\lambda$. \eqref{E. Lifting 2 in Sec 5} follows from Lemma \ref{NL.sum of cha} and \ref{NL.lifting 2}. It is straightforward to check that this $G_\lambda$ satisfies \RomanNumeralCaps{1}, \RomanNumeralCaps{2} and \RomanNumeralCaps{3}.
    
    \textbf{Case 2:} If $\reddeg(v_i)=2$ and the two red edges incident to $v_i$ connect to the same vertex $u$, let $w$ be the parent of $v_i$. 

    If $w=u$, then $w$ must be incident to red edges other than the two incident to $v_i$. In this case, we can simply delete $v_i$ to obtain $G_\lambda$. $G_\lambda$ is good and \eqref{E. Lifting 2 in Sec 5} follows from Lemma \ref{NL.1 red 1 black}.

    If $w\neq u$, then we delete $v_i$ and add a red edge $\{w,u\}$ to get $G_\lambda$. \eqref{E. Lifting 2 in Sec 5} follows from Lemmas \ref{NL.sum of cha} and \ref{NL.lifting 1-2}. One can check that $G_\lambda$ satisfies \RomanNumeralCaps{1} and \RomanNumeralCaps{3}. If it also satisfies \RomanNumeralCaps{2}, then we are done. Otherwise, $u$ must be some $v_j\in\{v_1,\dots,v_{i-1}\}$ and the only red edges incident to $v_j$ in $G$ are double red edges $\{v_i,v_j\}$. In this case, we delete both $v_i$ and $v_j$ to get $G_\lambda$. This $G_\lambda$ is good and \eqref{E. Lifting 2 in Sec 5} follows from applying Lemmas \ref{NL.sum of cha} and \ref{NL.lifting 1-2} twice.

    \textbf{Case 3:} If $\reddeg(v_i)\ge 3$. Suppose $v_i$ is connected by red edges to $u_1,\dots,u_k$. They need not be distinct, but each vertex appears at most twice. Let $u_0$ be the parent of $v_i$. For every $1\le j\le k$, denote by $G_j$ the graph obtained from $G$ by deleting $v_i$ and adding red edges $\{u_j,u_{j'}\},1\le j'\neq j\le k$ . For every pair of $1\le j<j'\le k$ such that $u_j\neq u_{j'}$, denote by $G_{j,j'}$ the graph obtained by deleting $v_i$ and adding red edges $\{u_0,u_{m}\},1\le m\neq j,j'\le k$ and $\{u_j,u_{j'}\}$ from G. Let $\{G_\lambda\}_{\lambda\in\Lambda}$ be the set of graphs consisting of all the $G_j$ and $G_{j,j'}$ above. \eqref{E. Lifting 2 in Sec 5} follows from Lemmas \ref{NL.sum of cha} and \ref{NL.lifting 2}. In each $G_\lambda$, each of $u_1,\dots,u_k$ is incident to at least one red edge. \RomanNumeralCaps{1} is satisfied. Since each vertex appears at most twice in $u_1,\dots,u_k$, the number of red edges incident to it in $G_\lambda$ is at least its number in $G$ minus one. Recall that $v_i$ is incident to the fewest red edges among $\{v_1,\dots,v_i\}$ and $\reddeg(v_i)\ge 3$ in $G$. Thus, each vertex in $\{v_1,\dots,v_{i-1}\}$ is incident to at least two red edges in $G_\lambda$. \RomanNumeralCaps{2} is satisfied. If all $G_\lambda$ satisfy \RomanNumeralCaps{3}, then we are done. Otherwise, one of them violates  \RomanNumeralCaps{3}. Then, there exists a $v_j\in\{v_1,\dots,v_{i-1}\}$ such that the only red edges incident to $v_j$ in $G$ are a red edge between $v_j$ and its parent $w$, and a double red edge between $v_i$ and $v_j$. Thus $\reddeg(v_j)=3$ in $G$ and $\reddeg(v_i)$ must also equal three by our assumption. To satisfy property \RomanNumeralCaps{3}, we must add a red edge $\{v_j,w\}$ when constructing $G_\lambda$. Thus, $w$ is also connected to $v_i$ via a red edge in $G$. All the red edges incident to either $v_i$ or $v_j$ are shown in Figure \ref{F.lifting 2 in case 2}. In this case, we delete both $v_i$ and $v_j$ to get $G_\lambda$. This $G_\lambda$ is good and \eqref{E. Lifting 2 in Sec 5} follows from applying Lemma \ref{NL.sum of cha} and \ref{NL.lifting 1-2} twice.
\end{proof}

\input{figures/lifting_2_in_case_2}  

Finally, we establish the star reduction and complete the proof.

\begin{proposition}[\textbf{Star reduction}]\label{P.lifting 3 in Case 2}
    If $G$ is a good graph with backbone $(\tree^\circ)^\circ$, then for any $x$ and $t\le aN$,
    \begin{equation}\label{E.lifting 3 in Case 2}
        \action{t,x}{G}\loglesssim N^{\frac{1}{4}(l-2)n}.
    \end{equation}
\end{proposition}
\begin{proof}
    For any graph $H$, we denote by $z(H)$ the number of non-root vertices of $H$ that are not incident to any red edge. Since $G$ is good, $z(G)=0$ at the beginning.

    We use the following claim.
    Let $H$ be a graph with an arbitrary backbone $\tree$ and $v$ be one of its leaves. Let $\tree'$ be the tree obtained by deleting $v$ from $\tree$.

    \begin{itemize}
        \item\textbf{Case 1:} If $v$ is not incident to any red edge, let $H'$ be the graph obtained by deleting $v$ from $H$. It is a graph with backbone $\tree'$, and for any $x$ and $t\le aN$,
        \begin{equation}\label{E.lifting 3 in Case 2-1}
            \action{t,x}{H}\lesssim N^{\frac{1}{2}}\action{t,x}{H'}.
        \end{equation}
        Moreover, $z(H')=z(H)-1$.
        \item\textbf{Case 2:} If $v$ is incident to at least one red edge in $H$, then there exists a family of graphs $\{H_\lambda\}_{\lambda\in\Lambda}$, all having  $\tree'$ as backbone, such that  for any $x$ and $t\le aN$,
        \begin{equation}\label{E.lifting 3 in Case 2-2}
            \action{t,x}{H}\loglesssim\sum_{\lambda\in\Lambda}\action{t,x}{H_\lambda}.
        \end{equation}
        Moreover, $z(H_\lambda)\le z(H)+1$ for any $\lambda$.
    \end{itemize}

    \begin{figure}[htbp]
\centering
\resizebox{0.98\textwidth}{!}{%
\begin{tikzpicture}[
    x=1cm,y=1cm,
    line cap=round,line join=round,
    vtx/.style={circle,fill=black,inner sep=1.7pt},
    bedge/.style={draw=black,line width=0.95pt},
    redge/.style={draw=red,line width=1.05pt},
    casehead/.style={font=\large\bfseries},
    symlab/.style={font=\Large},
    lab/.style={font=\small}
]

% =========================================================
% outer frame
% =========================================================
\draw[line width=1.2pt] (0,0.80) rectangle (17.80,8.55);
\draw[line width=1.2pt] (0,7.65) -- (17.80,7.65);
\draw[line width=1.2pt] (0,5.05) -- (17.80,5.05);
\draw[line width=1.2pt] (0,4.20) -- (17.80,4.20);

% headers
\node[anchor=west,casehead] at (0.45,8.10) {Case 1:};
\node[anchor=west,casehead] at (0.45,4.55) {Case 2:};

% =========================================================
% Case 1 formula
% =========================================================
\node[symlab] at (2.95,6.35) {$\mathcal{L}_{t,x}$};
\node[symlab] at (3.75,6.35) {$\Bigg ($};
\node[symlab] at (6.80,6.35) {$\Bigg )$};
\node[symlab] at (7.65,6.35) {$\lesssim$};
\node[symlab] at (8.85,6.35) {$N^{1/2}\,\mathcal{L}_{t,x}$};
\node[symlab] at (10.15,6.35) {$\Bigg ($};
\node[symlab] at (13.20,6.35) {$\Bigg )$};

% =========================================================
% Case 1: left graph
% =========================================================
\begin{scope}[xshift=5.25cm,yshift=5.77cm]
    \coordinate (A)  at ( 0.00, 1.45);
    \coordinate (L)  at (-1.18, 0.78);
    \coordinate (M)  at (-0.40, 0.78);
    \coordinate (U0) at ( 0.38, 0.78);
    \coordinate (R)  at ( 1.20, 0.78);
    \coordinate (BL) at (-1.18,-0.08);
    \coordinate (U2) at (-0.67,-0.08);
    \coordinate (M2) at (-0.2,-0.08);
    \coordinate (V)  at ( 0.15,-0.08);
    \coordinate (B3) at ( 0.56,-0.08);
    \coordinate (BR) at ( 1.20,-0.08);

    \draw[bedge] (A)--(L) (A)--(M) (A)--(U0) (A)--(R);
    \draw[bedge] (L)--(BL);
    \draw[bedge] (M)--(U2) (M)--(M2);
    \draw[bedge] (U0)--(V) (U0)--(B3);
    \draw[bedge] (R)--(BR);

    \draw[redge] (L)--(U2);
    \draw[redge] (L)--(V);
    \draw[redge] (V)--(R);
    \draw[redge] (B3)--(BR);
    \draw[redge] (M2)--(V);

    \foreach \p in {A,L,M,U0,R,BL,U2,M2,V,B3,BR}
        \node[vtx] at (\p) {};

    \node[lab] at (-1.18,-0.36) {$v$};
\end{scope}

% =========================================================
% Case 1: right graph
% =========================================================
\begin{scope}[xshift=11.62cm,yshift=5.77cm]
    \coordinate (A)  at ( 0.00, 1.45);
    \coordinate (L)  at (-1.18, 0.78);
    \coordinate (M)  at (-0.40, 0.78);
    \coordinate (U0) at ( 0.38, 0.78);
    \coordinate (R)  at ( 1.20, 0.78);
    \coordinate (U2) at (-0.67,-0.08);
    \coordinate (M2) at (-0.2,-0.08);
    \coordinate (V)  at ( 0.15,-0.08);
    \coordinate (B3) at ( 0.56,-0.08);
    \coordinate (BR) at ( 1.20,-0.08);

    \draw[bedge] (A)--(L) (A)--(M) (A)--(U0) (A)--(R);
    \draw[bedge] (M)--(U2) (M)--(M2);
    \draw[bedge] (U0)--(V) (U0)--(B3);
    \draw[bedge] (R)--(BR);

    \draw[redge] (L)--(U2);
    \draw[redge] (L)--(V);
    \draw[redge] (V)--(R);
    \draw[redge] (B3)--(BR);
    \draw[redge] (M2)--(V);

    \foreach \p in {A,L,M,U0,R,U2,M2,V,B3,BR}
        \node[vtx] at (\p) {};
\end{scope}

% =========================================================
% Case 2 symbols
% =========================================================
\node[symlab] at (3.98,2.55) {$\Rightarrow$};
\node[symlab] at (7.35,2.55) {$+$};
\node[symlab] at (10.68,2.55) {$+$};
\node[symlab] at (14.02,2.55) {$+$};

% =========================================================
% Case 2: H
% =========================================================
\begin{scope}[xshift=2.12cm,yshift=2.05cm]
    \coordinate (A)  at ( 0.00, 1.42);
    \coordinate (L)  at (-1.10, 0.78);
    \coordinate (M)  at (-0.38, 0.78);
    \coordinate (U0) at ( 0.36, 0.78);
    \coordinate (R)  at ( 1.12, 0.78);
    \coordinate (BL) at (-1.10,-0.06);
    \coordinate (U2) at (-0.6,-0.06);
    \coordinate (M2) at (-0.15,-0.06);
    \coordinate (V)  at ( 0.2,-0.06);
    \coordinate (B3) at ( 0.56,-0.06);
    \coordinate (BR) at ( 1.12,-0.06);

    \draw[bedge] (A)--(L) (A)--(M) (A)--(U0) (A)--(R);
    \draw[bedge] (L)--(BL);
    \draw[bedge] (M)--(U2) (M)--(M2);
    \draw[bedge] (U0)--(V) (U0)--(B3);
    \draw[bedge] (R)--(BR);

    \draw[redge] (L)--(U2);
    \draw[redge] (L)--(V);
    \draw[redge] (V)--(R);
    \draw[redge] (B3)--(BR);
    \draw[redge] (M2)--(V);
    
    \foreach \p in {A,L,M,U0,R,BL,U2,M2,V,B3,BR}
        \node[vtx] at (\p) {};

    \node[lab] at (-1.33,1.05) {$u_1$};
    \node[lab] at ( 0.07,0.75) {$u_0$};
    \node[lab] at ( 1.33,1.05) {$u_3$};
    \node[lab] at (-0.15,-0.35) {$u_2$};
    \node[lab] at ( 0.2,-0.35) {$v$};
    \node[lab] at ( 0.00,-0.9) {$H$};
\end{scope}

% =========================================================
% Case 2: H_0
% =========================================================
\begin{scope}[xshift=5.72cm,yshift=2.05cm]
    \coordinate (A)  at ( 0.00, 1.42);
    \coordinate (L)  at (-1.10, 0.78);
    \coordinate (M)  at (-0.38, 0.78);
    \coordinate (U0) at ( 0.36, 0.78);
    \coordinate (R)  at ( 1.12, 0.78);
    \coordinate (BL) at (-1.10,-0.06);
    \coordinate (U2) at (-0.6,-0.06);
    \coordinate (M2) at (-0.15,-0.06);
    \coordinate (B3) at ( 0.56,-0.06);
    \coordinate (BR) at ( 1.12,-0.06);

    \draw[bedge] (A)--(L) (A)--(M) (A)--(U0) (A)--(R);
    \draw[bedge] (L)--(BL);
    \draw[bedge] (M)--(U2) (M)--(M2);
    \draw[bedge] (U0)--(B3);
    \draw[bedge] (R)--(BR);

    \draw[redge] (L)--(U2);
    \draw[redge] (M2)--(U0);
    \draw[redge] (U0)--(R);
    \draw[redge] (B3)--(BR);

    \foreach \p in {A,L,M,U0,R,BL,U2,M2,B3,BR}
        \node[vtx] at (\p) {};

    \node[lab] at (-1.33,1.05) {$u_1$};
    \node[lab] at ( 0.07,0.75) {$u_0$};
    \node[lab] at ( 1.33,1.05) {$u_3$};
    \node[lab] at (-0.15,-0.35) {$u_2$};
    \node[lab] at ( 0.00,-0.9) {$H_0$};
\end{scope}

% =========================================================
% Case 2: H_1
% =========================================================
\begin{scope}[xshift=9.05cm,yshift=2.05cm]
    \coordinate (A)  at ( 0.00, 1.42);
    \coordinate (L)  at (-1.10, 0.78);
    \coordinate (M)  at (-0.38, 0.78);
    \coordinate (U0) at ( 0.36, 0.78);
    \coordinate (R)  at ( 1.12, 0.78);
    \coordinate (BL) at (-1.10,-0.06);
    \coordinate (U2) at (-0.6,-0.06);
    \coordinate (M2) at (-0.15,-0.06);
    \coordinate (B3) at ( 0.56,-0.06);
    \coordinate (BR) at ( 1.12,-0.06);

    \draw[bedge] (A)--(L) (A)--(M) (A)--(U0) (A)--(R);
    \draw[bedge] (L)--(BL);
    \draw[bedge] (M)--(U2) (M)--(M2);
    \draw[bedge] (U0)--(B3);
    \draw[bedge] (R)--(BR);

    \draw[redge] (L)--(U2);
    \draw[redge] (L)--(M2);
    \draw[redge] (L) to[bend left=12] (R);
    \draw[redge] (B3)--(BR);

    \foreach \p in {A,L,M,U0,R,BL,U2,M2,B3,BR}
        \node[vtx] at (\p) {};

    \node[lab] at (-1.33,1.05) {$u_1$};
    \node[lab] at ( 0.07,0.75) {$u_0$};
    \node[lab] at ( 1.33,1.05) {$u_3$};
    \node[lab] at (-0.15,-0.35) {$u_2$};
    \node[lab] at ( 0.00,-0.9) {$H_1$};
\end{scope}

% =========================================================
% Case 2: H_2
% =========================================================
\begin{scope}[xshift=12.38cm,yshift=2.05cm]
    \coordinate (A)  at ( 0.00, 1.42);
    \coordinate (L)  at (-1.10, 0.78);
    \coordinate (M)  at (-0.38, 0.78);
    \coordinate (U0) at ( 0.36, 0.78);
    \coordinate (R)  at ( 1.12, 0.78);
    \coordinate (BL) at (-1.10,-0.06);
    \coordinate (U2) at (-0.6,-0.06);
    \coordinate (M2) at (-0.15,-0.06);
    \coordinate (B3) at ( 0.56,-0.06);
    \coordinate (BR) at ( 1.12,-0.06);

    \draw[bedge] (A)--(L) (A)--(M) (A)--(U0) (A)--(R);
    \draw[bedge] (L)--(BL);
    \draw[bedge] (M)--(U2) (M)--(M2);
    \draw[bedge] (U0)--(B3);
    \draw[bedge] (R)--(BR);

    \draw[redge] (L)--(U2);
    \draw[redge] (L)--(M2);
    \draw[redge] (M2)--(R);
    \draw[redge] (B3)--(BR);

    \foreach \p in {A,L,M,U0,R,BL,U2,M2,B3,BR}
        \node[vtx] at (\p) {};

    \node[lab] at (-1.33,1.05) {$u_1$};
    \node[lab] at ( 0.07,0.75) {$u_0$};
    \node[lab] at ( 1.33,1.05) {$u_3$};
    \node[lab] at (-0.15,-0.35) {$u_2$};
    \node[lab] at ( 0.00,-0.9) {$H_2$};
\end{scope}

% =========================================================
% Case 2: H_3
% =========================================================
\begin{scope}[xshift=15.70cm,yshift=2.05cm]
    \coordinate (A)  at ( 0.00, 1.42);
    \coordinate (L)  at (-1.10, 0.78);
    \coordinate (M)  at (-0.38, 0.78);
    \coordinate (U0) at ( 0.36, 0.78);
    \coordinate (R)  at ( 1.12, 0.78);
    \coordinate (BL) at (-1.10,-0.06);
    \coordinate (U2) at (-0.38,-0.06);
    \coordinate (M2) at (-0.02,-0.06);
    \coordinate (B3) at ( 0.56,-0.06);
    \coordinate (BR) at ( 1.12,-0.06);

    \draw[bedge] (A)--(L) (A)--(M) (A)--(U0) (A)--(R);
    \draw[bedge] (L)--(BL);
    \draw[bedge] (M)--(U2) (M)--(M2);
    \draw[bedge] (U0)--(B3);
    \draw[bedge] (R)--(BR);

    \draw[redge] (L)--(U2);
    \draw[redge] (L) to[bend left=12] (R);
    \draw[redge] (M2)--(R);
    \draw[redge] (B3)--(BR);

    \foreach \p in {A,L,M,U0,R,BL,U2,M2,B3,BR}
        \node[vtx] at (\p) {};

    \node[lab] at (-1.33,1.05) {$u_1$};
    \node[lab] at ( 0.07,0.75) {$u_0$};
    \node[lab] at ( 1.33,1.05) {$u_3$};
    \node[lab] at (-0.02,-0.35) {$u_2$};
    \node[lab] at ( 0.00,-0.9) {$H_3$};
\end{scope}

\end{tikzpicture}%
}
\caption{Two cases of Star reduction.}
\label{F.lifting 3 in case 2}
\end{figure}
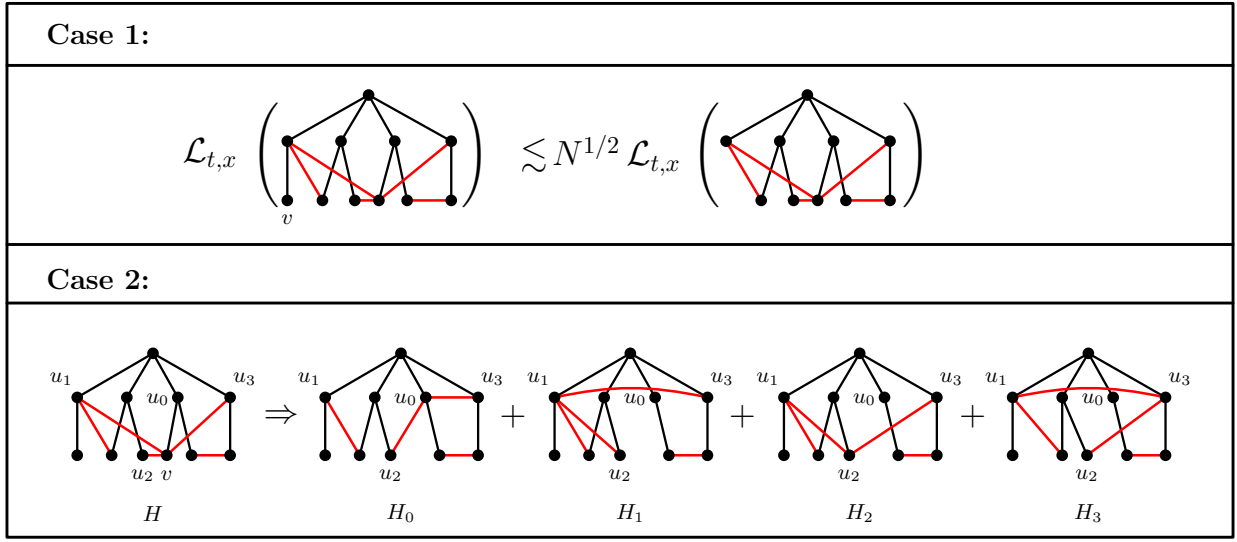

    Figure \ref{F.lifting 3 in case 2} illustrates a specific $H$ and vertex $v$. If the claim holds, we can iteratively apply it to $G$ until only the root remains. During this procedure, $z(\cdot)$ remains nonnegative throughout and ends at zero. So the number of vertices $v$ belonging to Case 1 is at most the number of vertices $v$ belonging to Case 2. Since $G$ has $(l-2)n$ non-root vertices, Case 1 occurs at most $(l-2)n/2$ times and hence \eqref{E.lifting 3 in Case 2} is established. It therefore suffices to prove the claim.
    
    If $v$ belongs to Case 1, then \eqref{E.lifting 3 in Case 2-1} follows from Lemma \ref{NL.sum of cha}. Moreover, $z(H')=z(H)-1$.

    If $v$ belongs to Case 2, suppose $v$ is connected by red edges to $u_1,\dots,u_k$. Let $u_0$ be the parent of $v$. Then, for any $1\le j\le k$, let $H_j$ be the graph obtained from $H$ by deleting $v$ and adding red edges $\{u_i,u_j\}: 1\le i\neq j \le k$. Also let $H_0$ be the graph obtained from $H$ by deleting $v$ and adding red edges

    $\{u_0,u_i\}:2\le i\le k$. Then, for any $0\le j\le k$, $z(H_j)\le z(H)+1$ and \eqref{E.lifting 3 in Case 2-2} follows from Lemma \ref{NL.sum of cha} and \ref{NL.lifting 2}.
\end{proof}
\begin{proof}[Proof of Theorem \ref{T.moments upper bound on K 3}]
    By Propositions \ref{P.Lifting 1 in Case 2}, \ref{P.Lifting 2 in Case 2} and \ref{P.lifting 3 in Case 2}, we have
    \begin{align*}
        \bbE\left[\left(\kterm{l}(t,x)\right)^n\right]&\lesssim N^{-(\frac{1}{4}+\rate)ln}\sum_{\pi~is~proper}\action{t,x}{(\tree_l^n)^\pi}\loglesssim N^{-(\frac{1}{4}+\rate)ln}\cdot N^{\frac{1}{4}(l-2)n}=N^{-(\frac{1}{2}+l\rate)n}.
    \end{align*}
\end{proof}

\section*{Acknowledgements}
We thank Yuchen Liao for helpful discussions. Y.Z. and F.L. thank Arka Adhikari, Fanhao Kong, Weijun Xu, Rongfeng Sun and Jinjiong Yu for useful comments and references, and also thank Xinyi Li for his kind support of our project.  D.L. is  supported by the National Natural Science Foundation of China (Grant No. 12371157).   T.W. is supported by the National Key Research and Development Project (Grant No. 2025YFA1017600), Anhui Postdoctoral Scientific Research Program Foundation (Grant No. 2025B1055), the Fundamental Research Funds for the Central Universities (Grant No.  WK0010250106), University of Science and Technology of China-Southwest University of Science and Technology Counterpart Cooperation and Development Joint Fund (Grant No. KY0010002501), and the Open Research Fund of Hubei Key Laboratory of Mathematical Sciences (Central China Normal University, Wuhan 430079, P. R. China).  Y.Z. is supported by the National Key R\&D Program of China (No.\ 2021YFA1002700), and the Beijing Natural Science Foundation (JQ26001). 

\section*{Statement on the use of AI}

Artificial intelligence (AI) tools were not used in developing any of the mathematical results or proofs presented in this paper. The paper was drafted without the use of AI tools. ChatGPT (versions 5.4-5.6) was used only to polish the language of the final draft and to generate figures.

\appendix
\section{Some Lemmas}\label{S.Some Lemmas}
In this section, we state and prove several numerical lemmas. Recall that $p(t,x)$ is the transition probability of the simple symmetric random walk, $\cha(t,x)=\difference p(t,x)=p(t-1,x-1)-p(t-1,x+1)$ and $\cha’(t,x)=C_{\cha}(1+t+x^2)^{-1}$, where $C_{\cha}$ is an absolute constant. Lemmas \ref{NL.upper bound on cha} to \ref{NL.1 red 1 black} are easily checked, and we omit their proofs here.
\begin{lemma}\label{NL.upper bound on cha}
     For any $t\ge 1$,
     $$\sup_{x\in\bbZ}\left\{p(t,x)\right\}=O\left(\frac{1}{\sqrt{t+1}}\right),\quad  \quad \sup_{x\in\bbZ}\left\{\absolute{\cha(t,x)}\right\}=O\left(\frac{1}{t+1}\right).$$
\end{lemma}
\begin{lemma}\label{NL.1 red}
    For any $0\le t\le aN$,
    $$\sum_{s=0}^{aN-1}\frac{1}{\sqrt{\absolute{t-s}+1}}=O(\sqrt N).$$
    Moreover, for any $0\le t_1,t_2\le aN$,
    $$\sum_{s=0}^{aN-1}\frac{1}{\sqrt{\absolute{t_1-s}+1}}\cdot\frac{1}{\sqrt{\absolute{t_2-s}+1}} =O(\log N).$$
\end{lemma}
\begin{lemma}\label{NL.sum of cha} 
    For any $0\le t\le aN, x\in\bbZ$,
    $$\sum_{z\in\bbZ}\cha'(t,x-z)=O\left(\frac{1}{\sqrt{t+1}}\right), \quad \sum_{s=1}^{t-1}\sum_{z\in\bbZ}\cha'(t-s,x-z)=O\left(\sqrt N\right).$$
\end{lemma}
\begin{lemma}\label{NL.1 red 1 black}
     For any $t_1, t_2\le aN$ and $x\in\bbZ$,
     $$\sum_{s=1}^{t_1-1}\sum_{z\in\bbZ}\cha'(t_1-s,x-z)\cdot\frac{1}{\sqrt{\absolute{t_2-s}+1}}=O(\log N).$$
\end{lemma}
\begin{lemma}\label{NL.lifting 1-1}
    For any fixed $k\ge 2$, $t_1\le t_2\le\cdots\le t_k\le aN$ and any $x_1,\ldots,x_k\in\bbZ$,
    \begin{equation}\label{E.graph lifting 1}
    \sum_{s=1}^{t_1-1}\sum_{z\in\bbZ}\prod_{i=1}^{k}\cha'(t_i-s,x_i-z)\lesssim \frac{\log N}{\sqrt{t_2-t_1+1}}\cdot\prod_{i=3}^{k}\frac{1}{t_i-t_1+1}.
    \end{equation}
\end{lemma}
\begin{proof}
    We first consider the case $k=2$.

    \begin{align*}
        \sum_{s=1}^{t_1-1}\sum_{z\in\bbZ}\absolute{\cha'(t_1-s,x_1-z)\cha'(t_2-s,x_2-z)}
        \le&\sum_{s=1}^{t_1-1}\sup_{z\in\bbZ}\left\{\absolute{\cha'(t_2-s,x_2-z)}\right\}\sum_{z\in\bbZ}\absolute{\cha'(t_1-s,x_1-z)}\\
        \lesssim&\sum_{s=1}^{t_1-1}\frac{1}{t_2-s+1}\cdot\frac{1}{\sqrt{t_1-s+1}}\lesssim\frac{\log N}{\sqrt{t_2-t_1+1}},
    \end{align*}
    where the last inequality follows by splitting the sum according to whether $t_1-s\le t_2-t_1$ or $t_1-s> t_2-t_1$.
    For general $k\ge 3$, \eqref{E.graph lifting 1} reduces to the $k=2$ case because for $3\le i\le k$,

    $$\sup_{z\in\bbZ}\left\{\absolute{\cha'(t_i-s,x_i-z)}\right\}=O\left(\frac{1}{t_i-s+1}\right)=O\left(\frac{1}{t_i-t_1+1}\right).$$
\end{proof}
\begin{lemma}\label{NL.lifting 1-2}
    If $t_1\le t_2\le\cdots\le t_k\le aN$ and $t^{(1)}, t^{(2)},\dots,t^{(k)}$ is any rearrangement of them, then
    $$\frac{1}{\sqrt{t_2-t_1+1}}\cdot\prod_{i=3}^{k}\frac{1}{t_i-t_1+1}\le\prod_{i=1}^{k-1}\frac{1}{\sqrt{|t^{(i+1)}-t^{(i)}|+1}}.$$
\end{lemma}
\begin{proof}
    Since $t_1\le t^{(i+1)},t^{(i)}$, we have
    $$|t^{(i+1)}-t^{(i)}|\le \max\{t^{(i+1)},t^{(i)}\}-t_1.$$
    Then, 
    $$\prod_{i=1}^{k-1}\frac{1}{\sqrt{|t^{(i+1)}-t^{(i)}|+1}}\ge\prod_{i=1}^{k-1}\frac{1}{\sqrt{\max\{t^{(i+1)},t^{(i)}\}-t_1+1}}.$$
    Note that in the multiset $\{\max\{t^{(i+1)},t^{(i)}\}:1\le i\le k-1\}$, $t_2$ appears at most once, and each $t_i,3\le i\le k$ appears at most twice. The proof is complete.
\end{proof}
\begin{lemma}\label{NL.lifting 1-3}
    If $t_1,t_2\le aN$ and $x_1,x_2\in\bbZ$, then 
    $$\sum_{s=1}^{\min\{t_1,t_2\}-1}\sum_{z\in\bbZ}\cha'(t_1-s,x_1-z)\cha'(t_2-s,x_2-z)^2\lesssim\frac{1}{\absolute{t_1-t_2}+1}.$$
\end{lemma}
\begin{proof}
    If $t_2\le t_1$, then
    \begin{align*}
        \sum_{s=1}^{t_2-1}\sum_{z\in\bbZ}\cha'(t_1-s,x_1-z)\cha'(t_2-s,x_2-z)^2
        &\le\sup_{\substack{s\le t_2-1\\z\in\bbZ}}\{\cha'(t_1-s,x_1-z)\}\sum_{s=1}^{t_2-1}\sum_{z\in\bbZ}\cha'(t_2-s,x_2-z)^2\\
        &\lesssim\sup_{\substack{s\le t_2-1\\z\in\bbZ}}\{\cha'(t_1-s,x_1-z)\}=O\left(\frac{1}{t_1-t_2+1}\right).
    \end{align*}
    The case $t_1<t_2$ is analogous.
\end{proof}
\begin{lemma}\label{NL.lifting 1-4}
    If $t_1\le t_2\le\cdots\le t_m\le aN$, then for any $1\le i\le m$ and $x_1,x_2,\dots,x_m\in\bbZ$,
    $$\sum_{s=1}^{t_1-1}\sum_{z\in\bbZ}\left[\cha'(t_i-s,x_i-z)\prod_{j=1}^{m}\cha'(t_j-s,x_j-z)\right]\lesssim\prod_{j=1}^{m-1}\frac{1}{t_{j+1}-t_j+1}.$$
\end{lemma}
\begin{proof}
 Let
$$
I_i:=\sum_{s=1}^{t_1-1}\sum_{z\in\mathbb Z}
\Bigl[\cha'(t_i-s,x_i-z)\prod_{j=1}^m \cha'(t_j-s,x_j-z)\Bigr].
$$
We first consider the case $2\le i\le m$. For any $2\le j\le m$, Lemma \ref{NL.upper bound on cha} gives
$$
\sup_{1\le s\le t_1-1}\sup_{z\in\mathbb Z}
\cha'(t_j-s,x_j-z)
=O\!\left(\frac1{t_j-t_{j-1}+1}\right).
$$
Applying this bound yields

$$
I_i
\lesssim
\left(\prod_{j=i+1}^m \frac1{t_j-t_{j-1}+1}\right)
\left(\prod_{j=2}^{i-1} \frac1{t_j-t_{j-1}+1}\right)
\sum_{s=1}^{t_1-1}\sum_{z\in\mathbb Z}
\cha'(t_1-s,x_1-z)\cha'(t_i-s,x_i-z)^2.
$$
Note that Lemma \ref{NL.lifting 1-3} yields
$$
\sum_{s=1}^{t_1-1}\sum_{z\in\mathbb Z}
\cha'(t_1-s,x_1-z)\cha'(t_i-s,x_i-z)^2
\lesssim \frac1{t_i-t_1+1}.
$$
Since $t_i-t_1\ge t_i-t_{i-1}$, we obtain
$$
I_i\lesssim \prod_{j=1}^{m-1}\frac1{t_{j+1}-t_j+1},
\qquad 2\le i\le m.
$$
It remains to consider the case $i=1$. The same argument gives
$$
I_1
\lesssim
\left(\prod_{j=2}^m \frac1{t_j-t_{j-1}+1}\right)
\sum_{s=1}^{t_1-1}\sum_{z\in\mathbb Z}\cha'(t_1-s,x_1-z)^2.
$$
For the remaining sum,  Lemmas \ref{NL.upper bound on cha} and \ref{NL.sum of cha} imply
\begin{align*}
    \sum_{s=1}^{t_1-1}\sum_{z\in\mathbb Z}\cha'(t_1-s,x_1-z)^2
&\le
\sum_{s=1}^{t_1-1}
\Bigl(\sup_{z\in\mathbb Z}\cha'(t_1-s,x_1-z)\Bigr)
\sum_{z\in\mathbb Z}\cha'(t_1-s,x_1-z)\\
&\lesssim
\sum_{s=1}^{t_1-1}
\frac1{t_1-s+1}\cdot \frac1{\sqrt{t_1-s+1}}
=
O(1).
\end{align*}
This completes the proof.    
\end{proof}
\begin{lemma}\label{NL.lifting 2}
    For a fixed $d\ge 1$ and any $1\le t_0,\ldots,t_d\le aN$ with $0\le j_0,\ldots,j_d\le d$ satisfying $j_i\neq i$ for every $0\le i\le d$,
    $$\sum_{s=1}^{aN}\prod_{i=0}^{d}\frac{1}{\sqrt{\absolute{t_i-s}}+1}\lesssim\log N\sum_{i=0}^{d}\prod_{\substack{0\le j\le d\\ j\neq i,j_i}}\frac{1}{\sqrt{\absolute{t_j-t_i}}+1}.$$
\end{lemma}
\begin{proof}
    Without loss of generality, after relabeling, $t_0\le t_1\le\cdots\le t_d$. We partition the interval $[1, aN]$ into the subintervals
    \begin{align*}
        J_0=\left[1,\frac{t_0+t_1}{2}\right]\cap\bbZ,\ \     J_i=\left[\frac{t_{i-1}+t_{i}}{2},\frac{t_{i}+t_{i+1}}{2}\right]\cap\bbZ,~\forall~1\le i\le d-1,\ \
        J_d=\left[\frac{t_{d-1}+t_{d}}{2},aN\right]\cap\bbZ.
    \end{align*}
    It suffices to show that, for any $0\le i\le d$,
    \begin{equation}\label{E.NL.lifting 2}
        \sum_{s\in J_i}\prod_{j=0}^{d}\frac{1}{\sqrt{\absolute{t_j-s}}+1}\lesssim\log N\prod_{\substack{0\le j\le d\\ j\neq i,j_i}}\frac{1}{\sqrt{\absolute{t_j-t_i}}+1}.
    \end{equation}
    
    \textbf{Case 1}: If $1\le i\le d-1$, we have
    \begin{align*}
        \sum_{s\in J_i}\prod_{j=0}^{d}\frac{1}{\sqrt{\absolute{t_j-s}}+1}
        &\le\prod_{\substack{0\le j\le d\\j\neq i,i-1}}\max_{s\in J_i}\left\{\frac{1}{\sqrt{\absolute{t_j-s}}+1}\right\}\cdot\sum_{s\in J_i}\frac{1}{\sqrt{\absolute{t_i-s}}+1}\cdot\frac{1}{\sqrt{\absolute{t_{i-1}-s}}+1}\\
        &\lesssim\log N\prod_{\substack{0\le j\le d\\j\neq i,i-1}}\frac{1}{\sqrt{\absolute{t_j-t_i}}+1}.
    \end{align*}
    Similarly, 
    $$\sum_{s\in J_i}\prod_{j=0}^{d}\frac{1}{\sqrt{\absolute{t_j-s}}+1}\lesssim\log N\prod_{\substack{0\le j\le d\\j\neq i,i+1}}\frac{1}{\sqrt{\absolute{t_j-t_i}}+1}.$$
    Note that the right-hand side of \eqref{E.NL.lifting 2} is minimized when $j_i=i-1$ or $i+1$. This concludes the proof of \eqref{E.NL.lifting 2}.

    \textbf{Case 2}: If $i=0$ or $d$, the same argument gives
    $$\sum_{s\in J_0}\prod_{j=0}^{d}\frac{1}{\sqrt{\absolute{t_j-s}}+1}\lesssim\log N\prod_{\substack{1\le j\le d\\j\neq 0,1}}\frac{1}{\sqrt{\absolute{t_j-t_0}}+1}$$
    and  $$\sum_{s\in J_d}\prod_{j=0}^{d}\frac{1}{\sqrt{\absolute{t_j-s}}+1}\lesssim\log N\prod_{\substack{0\le j\le d\\j\neq d-1,d}}\frac{1}{\sqrt{\absolute{t_j-t_d}}+1}.$$
    Then, \eqref{E.NL.lifting 2} is valid since its right-hand side is minimized when $j_0=1$ or $j_d=d-1$.
\end{proof}
\begin{lemma}\label{NL.reflection identity}
    For any $1\le t_1\le t_2$ and $x_1,x_2\in\bbZ$, 
    $$\sum_{\substack{0\le s\le t_1-1\\z\in\bbZ}}\cha(t_1-s,x_1-z)\cha(t_2-s,x_2-z)=4p\left(t_2-t_1,x_1-x_2\right)-4p\left(t_1+t_2,x_1-x_2\right).$$
    In particular, for any $t \in \bbZ_+$,
    $$\sum_{\substack{0\le s\le t-1\\z\in\bbZ}}\cha(t-s,z)^2=4-4p(2t,0).$$
\end{lemma}
\begin{proof}
For a formal power series $a(q)=\sum_{n\in\bbZ}a_nq^n$ and an arbitrary integer $m$, denote by $\left[q^m\right](a(q))$ the coefficient of $q^m$ in $a(q)$. The standard coefficient representation is 

$$p(t,x)=\left[q^x\right]\left(\left(\frac{q+q^{-1}}{2}\right)^t\right).$$
Then, 
\begin{align*}
    \cha(t,x)&=p(t-1,x-1)-p(t-1,x+1)=\left[q^{x-1}\right]\left(\left(\frac{q+q^{-1}}{2}\right)^{t-1}\right)-\left[q^{x+1}\right]\left(\left(\frac{q+q^{-1}}{2}\right)^{t-1}\right)\\
    &=\left[q^x\right]\left((q-q^{-1})\left(\frac{q+q^{-1}}{2}\right)^{t-1}\right).
\end{align*}
By symmetry, we also have
$$\cha(t,x)=\left[q^{-x}\right]\left((q^{-1}-q)\left(\frac{q+q^{-1}}{2}\right)^{t-1}\right).$$
Thus, the sum can be written as
\begin{align*}
    &\sum_{s=0}^{t_1-1}\sum_{z\in\bbZ}\cha(t_1-s,x_1-z)\cha(t_2-s,x_2-z)\\
    =&\sum_{s=0}^{t_1-1}\sum_{z\in\bbZ}\left[q^{x_1-z}\right]\left((q-q^{-1})\left(\frac{q+q^{-1}}{2}\right)^{t_1-s-1}\right)\cdot\left[q^{z-x_2}\right]\left((q^{-1}-q)\left(\frac{q+q^{-1}}{2}\right)^{t_2-s-1}\right)\\
    =&\left[q^{x_1-x_2}\right]\left((2-q^2-q^{-2})\sum_{s=0}^{t_1-1}\left(\frac{q+q^{-1}}{2}\right)^{t_1+t_2-2s-2}\right)\\
    =&4\left[q^{x_1-x_2}\right]\left(\left(\frac{q+q^{-1}}{2}\right)^{t_2-t_1}-\left(\frac{q+q^{-1}}{2}\right)^{t_1+t_2}\right)=4p(t_2-t_1,x_1-x_2)-4p(t_1+t_2,x_1-x_2).
\end{align*}
The second identity is a special case of the first obtained by taking $x_1=x_2=0$ and $t_1=t_2=t$.
\end{proof}
\section{Proof of Proposition \ref{P.binary tree}}\label{S.Proof of Proposition P.binary tree}
\begin{proof}
    We construct a sequence of trees $\tree=\tree_0,\ldots,\tree_C=\tilde\tree$ interpolating between $\tree$ and $\tilde\tree$. They all satisfy the condition that every non-root, non-leaf vertex has at least two children. They also have naturally identified leaf sets carrying the proper partition $\pi$, and for any $0\le i\le C-1$,
    \begin{equation}\label{E.interpolation}
        \action{t,x}{\tree_i^\pi}\lesssim\action{t+1,x}{\tree_{i+1}^\pi}
    \end{equation}
    for the fixed proper partition $\pi$ and all $(t,x)$. The construction is given by induction. Assume we have obtained a sequence up to $T_i$. If $T_i$ is not a full binary tree, then it must have a non-root vertex $u$ with at least three children. Suppose $u$ has $k$ children $v_1,\ldots,v_k,~k\ge 3$. We pick two of them by the following rules:
    \begin{itemize}
        \item \textbf{Case 1:} At least one of $v_1,\dots,v_k$ is not a leaf. WLOG, $v_1$ is not a leaf. If $v_1$ has two children, both of which are leaves, and $v_2$ is a leaf which, together with two children of $v_1$, forms a part of $\pi$, then we pick $v_1$ and $v_3$. Otherwise, we pick $v_1$ and $v_2$.
        \item \textbf{Case 2:} All of $v_1,\dots,v_k$ are leaves, but they are not in the same part of $\pi$. Suppose $v_i$ and $v_j$ do not belong to the same part, and we pick them.
        \item \textbf{Case 3:} All of $v_1,\dots,v_k$ are leaves, and they are in the same part of $\pi$. In this case, either $k\ge 4$ or that part of $\pi$ containing them contains some other leaf. We simply pick $v_1$ and $v_2$.
    \end{itemize}
    Suppose the selected pair is $v_i$ and $v_j$. We then delete the edges $[u,v_i], [u,v_j]$ and add a new vertex $w$ with three edges $[u,w],[w,v_i],[w,v_j]$. Let this new tree be $T_{i+1}$. Since the operation producing $T_{i+1}$ from $T_i$ preserves leaves, there is a natural identification between $L(T_i)$ and $L(T_{i+1})$. In each of the three cases, $\pi$ remains a proper partition of $L(T_{i+1})$. Now, for any given $\pi$ and $\embed\in\embedset{t,x}^\pi(T_i)$, let $\embed'$ be
    \begin{align*}
            \embed'(v)=\begin{cases}
            \embed(v) & \text{if $v$ is neither $w$ nor an ancestor of $w$ in $T_{i+1}$}\\
            \embed(u) & \text{if $v$ is $w$}\\
            \embed(v)+(1,0) & \text{if $v$ is an ancestor of $w$ in $T_{i+1}$}
        \end{cases}
    \end{align*}
   which is an element of $\embedset{t+1,x}(T_{i+1})$. Note that $\embed'$ coincides with $\embed$ on all leaves; consequently, it induces the same partition $\pi$, that is, $\embed' \in \embedset{t+1,x}^\pi(T_{i+1})$. It can be verified directly that $\pi$ remains a proper partition of $L(T_{i+1})$.
   We also use $\cha'(s,z) \lesssim \cha'(s+1,z)$ for any $s \ge 1$ and $z \in \mathbb{Z}$. Combined with the fact that $\weight{\embed'}([u,w]) = \cha'(1,0) > 0$, this gives
    $$\prod_{e\in E(T_i^\pi)}\weight{\embed}(e)\lesssim\prod_{e\in E(T_{i+1}^\pi)}\weight{\embed'}(e).$$ 
    Since the mapping $\embed \mapsto \embed'$ is injective, \eqref{E.interpolation} follows by summing this inequality over all $\embed$. Thus, we have constructed $T_{i+1}$ and, by induction, the entire sequence $T_0, T_1, \ldots$. In this sequence, each subsequent term contains exactly one more vertex than the preceding one, while the number of leaves remains unchanged. Eventually, $T_C$ is a full binary tree for some constant $C = C(T)$. This completes the proof.
\end{proof}
\printbibliography

\end{document}